\documentclass[11pt]{article}

\usepackage{amsmath,amssymb,amsthm,vmargin}
\usepackage{graphicx}
\usepackage{hyperref}
\usepackage{authblk}

\theoremstyle{theorem}

\newtheorem{theorem}{Theorem}
\newtheorem{lemma}[theorem]{Lemma}
\newtheorem{proposition}[theorem]{Proposition}
\newtheorem{corollary}[theorem]{Corollary}
\theoremstyle{definition}
\newtheorem{definition}[theorem]{Definition}
\newtheorem{assumption}[theorem]{Assumption}
\newtheorem{remark}[theorem]{Remark}

\numberwithin{equation}{section}
\numberwithin{theorem}{section}

\DeclareMathOperator{\diag}{diag}

\newcommand{\ds}{\displaystyle}

\let\Re\undefined
\let\Im\undefined
\DeclareMathOperator{\Re}{Re}
\DeclareMathOperator{\Im}{Im}

\title{Asymptotics of matrix valued orthogonal  polynomials on $[-1,1]$}

\author[1]{Alfredo Dea\~{n}o\thanks{alfredo.deanho@uc3m.es}}
\author[2]{Arno B.J. Kuijlaars\thanks{arno.kuijlaars@kuleuven.be}}
\author[3]{Pablo Rom\'an\thanks{pablo.roman@unc.edu.ar}}
\affil[1]{Department of Mathematics, Universidad Carlos III de Madrid, Spain}
\affil[2]{Department of Mathematics,  Katholieke Universiteit Leuven, Leuven, Belgium}
\affil[3]{FaMAF-CIEM, Universidad Nacional de C\'ordoba, Argentina}

\begin{document}
\maketitle


\abstract{We analyze the large degree asymptotic behavior of matrix valued orthogonal polynomials (MVOPs), with a weight that consists of a Jacobi scalar factor and a matrix part. Using the Riemann--Hilbert formulation for MVOPs and the Deift--Zhou method of steepest descent, we obtain asymptotic expansions for the MVOPs as the degree tends to infinity, in different regions of the complex plane (outside the interval of orthogonality, on the interval away from the endpoints and in neighborhoods of the endpoints), as well as for the matrix coefficients in the three-term recurrence relation for these MVOPs. The asymptotic analysis follows the work of Kuijlaars, McLaughlin, Van Assche and Vanlessen on scalar Jacobi-type orthogonal polynomials, but it also requires  several different factorizations of the matrix part of the weight, in terms of eigenvalues/eigenvectors and using a matrix Szeg\H{o} function. We illustrate the results with two main examples, MVOPs of Jacobi and Gegenbauer type, coming from group theory.}

\section{Introduction and statement of results} \label{sec1}

\subsection{Introduction}

In this paper, we are interested in the large degree asymptotic behavior of matrix valued orthogonal
polynomials (MVOPs), with orthogonality defined on $[-1,1]$. The weight
matrix $W$ on $[-1,1]$ is of size $r \times r$, and we take
it of the form
\begin{equation} \label{Wdef}
	W(x) = (1-x)^{\alpha} (1+x)^{\beta} H(x) \end{equation}
with $\alpha, \beta > -1$ and where the matrix valued function 
$H(x)$ satisfies the following: 
\begin{assumption} \label{assumptions}
	\begin{itemize}
		\item[(a)] $H(x)$ is an $r \times r$ complex valued matrix for $x \in [-1,1]$,
		\item[(b)] $H(x)$ is Hermitian positive definite for $x \in (-1,1)$,
		\item[(c)] $H(x)$ is real analytic on $[-1,1]$,
		\item[(d)] $H(-1)$ and $H(1)$ are
		not identically zero. 
	\end{itemize}
\end{assumption}
The real analyticity means that $H$ has an analytic extension 
to a neighborhood  of $[-1,1]$ in the
complex plane that we will also denote by $H$.
The requirement in (b) is that $H(x)$ is a Hermitian matrix,
i.e., $H(x) = H(x)^\ast$, for every $x \in (-1,1)$, with positive
eigenvalues. Then by real analyticity $H(-1)$ and $H(1)$
are Hermitian non-negative definite, but not necessarily
positive definite, as some of the eigenvalues could vanish at $\pm 1$.
However, not all eigenvalues can vanish because of the 
requirement in (d). 
In our examples the matrix valued function $H$ is polynomial in $x$,
and $H(\pm 1)$ will be singular.

If $H(x)$ is a diagonal matrix for every $x \in [-1,1]$, then the
MVOPs reduce to $r$ usual scalar orthogonal polynomials with
weight functions of the type $w(x)=(1-x)^{\alpha} (1+x)^{\beta} h(x)$, 
where $h(x)$ is analytic in a neighborhood of $[-1,1]$.
Strong asymptotics for these kind of orthogonal polynomials
was obtained with Riemann-Hilbert methods by Kuijlaars, McLaughlin,
Van Assche and Vanlessen in \cite{KMVV04}, and the
present paper can be viewed as a matrix valued extension 
of that work.

The monic MVOP $P_n$ is  defined by the property that for $m,n\geq 0$,
\begin{equation} \label{MVOPdef}
	\int_{-1}^1 P_n(x) W(x) P_m(x)^\ast dx = \delta_{n,m} \Gamma_n
\end{equation}
with a positive definite matrix $\Gamma_n$,
where $P_n(x) = x^n I_r + \cdots$ is a matrix valued
polynomial of degree $n$ whose leading coefficient
is the identity matrix $I_r$. The integral in \eqref{MVOPdef} 
is taken entrywise.
Under Assumption \ref{assumptions},
existence and uniqueness of the sequence $(P_n)_n$ is guaranteed.

Matrix orthogonal polynomials have appeared in many different contexts in the literature in the last years. Following classical ideas in the scalar case, Dur\'an and Gr\"{u}nbaum in \cite{DG05a,DG05b} studied MVOPs from the perspective of eigenfunctions of second order differential operators with matrix coefficients. This work has produced a large number of contributions in the literature, extending classical identities for scalar OPs to the matrix case. A general analysis of the matrix Bochner problem (the classification of $N\times N$ weight matrices whose associated MVOPs are eigenfunctions of a second order differential operator) has been recently addressed by Casper and Yakimov in \cite{CY22}, using techniques from noncommutative algebra. 



From the point of view of group theory and representation theory, the study of matrix valued spherical functions has led to families of MVOPs associated to compact symmetric spaces. The first example of this connection is given by Gr\"{u}nbaum, Pacharoni and Tirao in \cite{GPT02} for the symmetric pair $(G, K) = (\mathrm{SU}(3),\mathrm{U}(2))$, see also \cite{PT07,PT13,RT06}. Another approach was
developed in \cite{KvPR12, KvPR13} for the $(\mathrm{SU}(2) \times \mathrm{SU}(2), \diag)$, and later extended to 
a more general set-up in the context of the so-called multiplicity free pairs. 
In particular, \cite{KdlRR17} gives a detailed study of the Gegenbauer matrix valued orthogonal
polynomials, which can be considered as matrix valued analogues of the Chebyshev
polynomials, i.e., the spherical polynomials on $(\textrm{SU}(2) \times \textrm{SU}(2), \diag)$, better known as the characters on SU(2), see also \cite{AKR17} for the quantum group case.

The Riemann--Hilbert formulation for MVOPs appears in the works of Gr\"{u}nbaum, de la Iglesia and Mart\'inez-Finkelshtein \cite{GIM11}, and Cassatella-Contra and Ma\~{n}as \cite{CM12}, as a generalization of the classical result of Fokas, Its and Kitaev \cite{FIK92}. 
This formulation has been used in several examples, like Hermite and Laguerre--type MVOPs in \cite{CdI2014,CdI2018} or matrix
biorthogonal polynomials in \cite{BFM21,BFFM22},
in order to obtain algebraic and differential identities for MVOPs that can be seen as non-commutative analogues of well known identities in the theory of integrable systems, such as the Toda lattice equation or Painlev\'e equations.  

Asymptotic results for MVOPs obtained from the Riemann--Hilbert formulation using the Deift--Zhou method \cite{DZ93} 
of steepest descent are much more scarce. In the last few years, MVOPs have appeared in the area of integrable probability, 
more precisely in the study of random tilings of plane figures. We mention the recent work by Duits and Kuijlaars \cite{DK21} 
and Berggren and Duits \cite{BD19}
on periodic tilings of the Aztec diamond, as well as the papers by Charlier \cite{Ch21} and by 
Groot and Kuijlaars \cite{GK21} on doubly periodic lozenge tilings of a hexagon. In these cases, an essential step 
in the asymptotic analysis is the connection between matrix orthogonality in the complex plane and scalar orthogonality 
on suitable curves in a Riemann surface. 

Our results are strong asymptotic formulas for $P_n(z)$ as $n \to \infty$, for $z$ in three regions in the complex plane,
namely in the exterior region $\mathbb C \setminus [-1,1]$, in the oscillatory region $(-1,1)$ 
away from the endpoints, and near the endpoints. An important aspect of this work is the fact that we use different factorizations of the weight matrix for the asymptotic analysis: in the outer region and on the interval $(-1,1)$, we use a matrix Szeg\H{o} function $D$, which is obtained from a matrix spectral factorization of the weight on the unit circle; in neighborhoods of the endpoints, we use the spectral decomposition of $W(x)$, since the possible vanishing of the eigenvalues at $z=\pm 1$ is essential in the construction of the local parametrices. The same methodology allows us to include asymptotic expansions for the recurrence coefficients as well.

Throughout we assume that $W$ is of the form \eqref{Wdef}
with $H$ satisfying Assumption \ref{assumptions}, and $P_n$
is the degree $n$  monic MVOP satisfying \eqref{MVOPdef}.
We use $A^T$ to denote the transpose of a matrix $A$ and $A^*$ for
its Hermitian transpose.
For a matrix valued function $A(z)$ defined for $z \in \mathbb C \setminus \Sigma$, where $\Sigma$ is an oriented contour, we use
$A_+(x)$ ($A_-(x)$) for the limits of $A(z)$ as $z \to x \in \Sigma$
from the $+$-side ($-$-side). The $+$-side ($-$-side) is on our left (right) as
we follow $\Sigma$ according to
its orientation. 

\subsection{Factorizations of the weight matrix}

Our asymptotic results rely on three factorizations
of the weight matrix.  
\subsubsection{First factorization}
The first one is the familiar
spectral decomposition of $H(x)$
\begin{equation} \label{HQLambda} 
	H(x) = Q(x) \Lambda(x) Q(x)^*, \qquad x \in [-1,1] \end{equation}
with a unitary matrix $Q(x)$ and a diagonal matrix
\begin{equation} \label{Lambdadef} \Lambda(x) = \diag\left(\lambda_1(x), \ldots, \lambda_r(x)\right) \end{equation}
containing the eigenvalues $\lambda_j(x)$, $j=1, \ldots, r$ of $H(x)$. 
The assumption that $H$ is real analytic on $[-1,1]$ 
has the following important consequence.

\begin{lemma} \label{Rellich}
	$Q(x)$ and $\Lambda(x)$ can (and will) be taken to be real analytic on $[-1,1]$.
\end{lemma}
\begin{proof}
	This is a well-known theorem of Rellich, see \cite{Wim86} or \cite[Theorem 1.4.4]{Simon15}.
\end{proof}
We choose $Q(x)$ and $\Lambda(x)$ as in Lemma \ref{Rellich},
and we continue to use $Q$ and $\Lambda$ for their analytic
continuations to a neighborhood of $[-1,1]$ in the
complex plane. Then each eigenvalue $\lambda_j(x)$, $j=1, \ldots, r$
is analytic in that same neighborhood of $[-1,1]$, and it 
satisfies $\lambda_j(x) > 0$ for $x \in (-1,1)$ because
of Assumption \ref{assumptions} (b), but $\lambda_j(x)$
could be zero at $x = \pm 1$, since we do not assume
positive definiteness of $H$ at the endpoints. 

\begin{definition} \label{def13}
We define for $j=1, \ldots, r$,
\begin{enumerate}
	\item[(a)] $n_j$ is the order of vanishing of $\lambda_j(x)$
	at $x=1$, where we put $n_j = 0$ if $\lambda_j(1) > 0$, and
	\begin{equation} \label{alphajdef}
		\alpha_j = \alpha + n_j,
	\end{equation}
	\item[(b)] $m_j$ is the order of vanishing of $\lambda_j(x)$
	at $x=-1$, where we put $m_j = 0$ if $\lambda_j(-1) > 0$,
	and
	\begin{equation} \label{betajdef}
		\beta_j = \beta + m_j.
	\end{equation}
\end{enumerate} 
\end{definition} 
Because of Assumption \ref{assumptions} (d) at least one
of the numbers $n_1, \ldots, n_r$ is equal to zero, and similarly
for the $m_j$'s. Thus we have
\[ \min \{n_j \mid j=1, \ldots, r\} =
	\min \{m_j \mid j=1, \ldots, r\} = 0. \]
We emphasize that $\lambda_j(x)$, for $j=1, \ldots, r$ are
the eigenvalues of $H(x)$, and so  by \eqref{Wdef} 
the eigenvalues of $W(x)$ are  
$(1-x)^{\alpha} (1+x)^{\beta} \lambda_j(x)$ for $j=1, \ldots, r$. 

\subsubsection{Second factorization}
The second factorization of $W(x)$ is less familiar.
\begin{proposition} \label{prop12}
	There  exists an analytic matrix valued function 
		$D : \mathbb C \setminus [-1,1] \to \mathbb C^{r \times r}$  with boundary values
		$D_{\pm}$ on $(-1,1)$ satisfying
		\begin{equation} \label{WDfactor}
			W(x) = D_-(x) D_-(x)^\ast = D_+(x) D_+(x)^\ast,
		\end{equation}
		where $D(z)$ is invertible for every $z \in \mathbb C
		\setminus [-1,1]$, and such that 
		\begin{equation}\label{eq:Dinf}
			D(\infty) = \lim_{z \to \infty} D(z) 
		\end{equation}
		exists and is invertible as well.  
\end{proposition}
Proposition \ref{prop12} follows from Lemma \ref{lem:Ddef} below.

A similar factorization, but for weight matrices on the unit circle appeared in \cite{BD19}, 
in the study of correlation functions for determinantal processes involving infinite Toeplitz minors, 
which arise in random tilings of certain planar domains.

\begin{remark} \label{remark13}
	We consider $D(z)$ as a matrix valued Szeg\H{o} function.
	It arises from a matrix spectral factorization of the
	weight matrix $W$. It is unique up to a constant unitary
	matrix. That is, if $D$ satisfies the conditions 
	of Proposition \ref{prop12} and $U$ is a unitary matrix, independent of $z$, then $DU$
	satisfies the conditions as well. Uniqueness of the
	matrix valued Szeg\H{o} function is guaranteed if we
	require that $D(\infty)$ is a positive definite Hermitian matrix. 
	We call this the normalized matrix valued Szeg\H{o} function.

	If $W(x)$ is real valued for $x \in (-1,1)$, 
	then the normalized matrix valued Szeg\H{o} function
	$D$ will satisfy the symmetry condition
	\begin{equation} \label{Dreal} 
		D(\overline{z}) = \overline{D(z)}, \qquad z \in \mathbb C \setminus [-1,1]. \end{equation}
	In that case $D_-(x) = \overline{D_+(x)}$ and the factorization
	\eqref{WDfactor} can be alternatively written as
	\begin{equation} \label{WDfactor2}
		W(x) = D_-(x) D_+(x)^T = D_+(x) D_-(x)^T.
	\end{equation}
	Also $D(\infty)$ is a positive definite real matrix in this case.
\end{remark}

\subsubsection{Third factorization}
The third factorization is very much related to the spectral
decomposition \eqref{HQLambda}. 
We use modified eigenvalues
\begin{equation} \label{lambdajdef}
		\widetilde{\lambda}_j = (-1)^{n_j} \lambda_j,
			\qquad  j=1, \ldots, r,
\end{equation}
and
\begin{equation} \label{Lambdatildef}
		\widetilde{\Lambda} = \diag (\widetilde{\lambda}_1, \ldots, \widetilde{\lambda}_r ). 
	\end{equation}
Recall from Definition \ref{def13} that $n_j$ denotes the order of vanishing of $\lambda_j$
at $x=1$. Thus  $\widetilde{\lambda}_j(x) > 0$
for $x \in (1,1+\delta)$ for some $\delta > 0$,
and we use $\widetilde{\lambda}_j(x)^{1/2}$ to denote
its positive square root. This has an analytic continuation
to a neighborhood of $[-1,1]$ with a branch cut along 
$(-\infty,1]$ that we also denote by $\widetilde{\lambda}_j^{1/2}$.
Then we define
\begin{equation} \label{Lambdatilsqrtdef}
	\widetilde{\Lambda}^{1/2} = \diag \left(\widetilde{\lambda}_1^{1/2}, \ldots, \widetilde{\lambda}_r^{1/2} \right), 
\end{equation}
and
\begin{equation}  \label{Vdef}
	V(z) = (z-1)^{\alpha/2} (z+1)^{\beta/2} 
		Q(z) \widetilde{\Lambda}(z)^{1/2}. 
\end{equation}
which is defined and analytic with a branch cut along
$(-\infty,1]$. In particular it is defined and analytic
in $D(1,\delta) \setminus (1-\delta,1]$ for some $\delta >0$.

We will use $V$ for the local analysis around $1$. Near $-1$
we have a similarly defined matrix valued function.
We define
\begin{equation} \label{Lambdahatdef} \widehat{\Lambda} = \diag \left( \widehat{\lambda}_1,
	\ldots, \widehat{\lambda}_r\right), \quad
	\widehat{\lambda}_j = (-1)^{m_j} \lambda_j, \end{equation}
so that $\widehat{\lambda}_j(x) > 0$ for $x \in (-1-\delta, -1)$
for some $\delta > 0$. Then we define
\begin{equation} \label{Vhatdef}
		\widehat{V}(z) = (1-z)^{\alpha/2} (-1-z)^{\beta/2}
		Q(z) \widehat{\Lambda}^{1/2}(z),
	\end{equation}
defined with a branch cut along $[-1,\infty)$. 

The third factorization of $W$ is as follows:
\begin{lemma} We have for $x \in (-1,1)$,
	\begin{equation} \label{WVfactor} 
	\begin{aligned} 
	W(x) & = V_-(x) V_-(x)^\ast = V_+(x) V_+(x)^{\ast} \\
	& = \widehat{V}_-(x) \widehat{V}_-(x)^\ast = 
		\widehat{V}_+(x) \widehat{V}_+(x)^{\ast},
	\end{aligned}
\end{equation}
where $V$ and $\widehat{V}$ are defined by \eqref{Vdef} and \eqref{Vhatdef}.
\end{lemma}
\begin{proof}
	This follows by straightforward calculation from
	the definitions \eqref{Vdef} and \eqref{Vhatdef}.
	See also Lemma \ref{lemma34} for details. 
\end{proof}

Comparing \eqref{WVfactor} and \eqref{WDfactor} we
see that $V$ and $\widehat{V}$ share the same factorization
property with the matrix valued Szeg\H{o} function $D$.
Actually $D^{\pm}(x)^{-1} V_{\pm}(x)$ and $D_{\pm}(x)^{-1} \widehat{V}_{\pm}(x)$ are unitary matrices for every $x \in (-1,1)$,
see formula \eqref{DVprod} below. 
For our asymptotic results we need their values at the endpoints. 
\begin{lemma} \label{lemma17}
	The two limits
	\begin{align}  \label{U1def}
		U_1 = \lim_{z \to 1} D(z)^{-1} V(z), \qquad
		U_{-1} = \lim_{z \to -1} D(z)^{-1} \widehat{V}(z)
	\end{align}
	exist, and define unitary matrices $U_1$ and $U_{-1}$.
\end{lemma}
The proof of Lemma \ref{lemma17} is in Section \ref{subsec358}.

\subsection{Asymptotics in the exterior region}
Throughout the paper, we need the conformal map
\begin{equation} \label{phidef} \varphi(z) = z + (z^2-1)^{1/2},
	\qquad z \in \mathbb C \setminus [-1,1]
\end{equation}
from $\mathbb C \setminus [-1,1]$ to the exterior of the unit circle.
Our first result is the asymptotics of $P_n(z)$ as $n \to \infty$ 
for $z \in \mathbb C \setminus [-1,1]$. 
The main term in the asymptotic formula \eqref{Pnoutside} is not new
as it is known at least since \cite{AN83}, where it is
proved under weaker assumptions as well, namely $W$
is assumed to satisfy a matrix Szeg\H{o} condition on $[-1,1]$
with a finite number of mass points outside
$[-1,1]$). See also \cite{Ko10} for an infinite number of mass points.

\begin{theorem} \label{theorem14}
Let $W$ be the weight matrix \eqref{Wdef} with $H$
satisfying Assumption \ref{assumptions}. Let $D$ be
the matrix Szeg\H{o} function associated with $W$
as in Proposition \ref{prop12}. Then as $n\to\infty$ the monic MVOP $P_n$
has an asymptotic series expansion
 \begin{equation} \label{Pnoutside} \frac{2^n P_n(z)}{\varphi(z)^n} 
 	\sim  \frac{\varphi(z)^{1/2}}{\sqrt{2} (z^2-1)^{1/4}}
 			D(\infty) \left[ I_r + \sum_{k=1}^{\infty} \frac{\Pi_{k}(z)}{n^k} \right] D(z)^{-1},
 		\quad z \in \mathbb C \setminus [-1,1], \end{equation}
 uniformly for $z$ in compact subsets of $\mathbb C \setminus [-1,1]$, where each $\Pi_k$ is an analytic function
 in $\mathbb C \setminus [-1,1]$. The first one is
\begin{multline} \label{Pi1def}
	\Pi_1(z) = - \frac{1}{8(\varphi(z)-1)}
	U_1 \diag \left( 4\alpha_1^2 - 1,
	\ldots, 4\alpha_r^2 - 1 \right)
	U_1^{-1} \\
	+ 	\frac{1}{8(\varphi(z)+1)} 
	U_{-1} \diag \left( 4\beta_1^2 - 1,
	\ldots, 4\beta_r^2 - 1 \right)
	U_{-1}^{-1},
\end{multline} 
where $U_1$ and $U_{-1}$ are as in \eqref{U1def}
and the parameters $\alpha_j$ and $\beta_j$ for $j=1,\ldots, r$
are given by \eqref{alphajdef} and \eqref{betajdef}.
\end{theorem}
The proof of Theorem \ref{theorem14} is in Section \ref{subsec41}.

The leading term in  \eqref{Pnoutside} is known. The limit
\begin{equation} \label{Pi0def}
	\lim_{n\to \infty} \frac{2^n P_n(z)}{\varphi(z)^n} =
	\frac{\varphi(z)^{1/2}}{\sqrt{2} (z^2-1)^{1/4}} D(\infty) D(z)^{-1},
	\quad z \in \mathbb C \setminus [-1,1], \end{equation}
can equivalently be written as
\[ \lim_{n \to \infty} (2z)^n P_n\left(\frac{z+z^{-1}}{2} \right)
	= \frac{1}{(1-z^2)^{1/2}}  D(\infty) D\left(\frac{z+z^{-1}}{2}\right)^{-1} 
	\qquad |z| < 1,
\] 
which corresponds to the asymptotics stated in 
\cite[Theorem 2]{AN83} and  \cite{Ko10}.
The analogous result for MVOP on the unit circle dates back to \cite{YK78} and \cite{DGK78}.

The existence of a full asymptotic expansion is new, as 
well as the explicit
form \eqref{Pi1def} of the first subleading term.
\begin{remark}
The expression \eqref{Pi1def} simplifies if $n_j=m_j = 0$
for every $j=1, \ldots, r$, since in that case the
two diagonal matrices in \eqref{Pi1def} are multiples
of the identity matrix. Then \eqref{Pi1def} reduces to
\begin{align*} \Pi_1(z) & = \left( 
	- \frac{4\alpha^2-1}{8(\varphi(z)-1)}
		+ \frac{4\beta^2-1}{8 (\varphi(z) + 1)} \right) I_r.
	\end{align*}
which is consistent with the formula given in \cite[formula (1.13)]{KMVV04} for the scalar case.
\end{remark}

\subsection{Asymptotics on the interval $(-1,1)$}
The MVOP $P_n$ has oscillatory behavior on the interval $(-1,1)$.
Theorem \ref{theorem15} should be compared with Theorem 2 (f) in  \cite{AN83}, where the boundary
values are given in $L^2$ sense, while our asymptotic
formula \eqref{eq:innerasymp} holds uniformly on compact
subsets of $(-1,1)$.
\begin{theorem} \label{theorem15}
With the same assumptions as in Theorem \ref{theorem14},
we have uniformly for $x$ in compact subsets of $(-1,1)$,	
	\begin{multline} \label{eq:innerasymp0} 	
		2^n P_n(x) = 	\frac{1}{\sqrt{2} \sqrt[4]{1-x^2}} 
		D(\infty )\\
		\times
	\left( e^{i \left(n + \frac{1}{2} \right) \arccos(x) - 
		\frac{\pi i}{4}} D_+(x)^{-1}
	+
	e^{-i \left(n+ \frac{1}{2} \right) \arccos(x) + 
		\frac{\pi i}{4}}	D_-(x)^{-1} \right)
	+ \mathcal{O}(n^{-1}), \end{multline}	
	as $n \to \infty$.
	
	In case $W(x)$ is real symmetric for every $x \in (-1,1)$,
	then the MVOP $P_n(x)$ is real
	valued for real $x$. Then, if we use
	the normalized Szeg\H{o} function as in Remark \ref{remark13},
	we have uniformly for $x$ in compact subsets of $(-1,1)$,
	\begin{equation}\label{eq:innerasymp}
		\begin{aligned}
			2^n P_n(x)
			&=	\frac{\sqrt{2}}{(1-x^2)^{\frac{1}{4}}} D(\infty)	\Re\left(e^{i\left(n+\frac{1}{2}\right)\arccos x-\frac{\pi i}{4}}D_+(x)^{-1}\right)+\mathcal{O}(n^{-1}),
		\end{aligned}
	\end{equation}
	where the real part of the matrix is taken entrywise.
	\end{theorem}
	We prove Theorem \ref{theorem15} in Section \ref{subsec42}.

\subsection{Asymptotics near the endpoints $z=\pm 1$}

Near the endpoints $\pm 1$ we find asymptotic
formulas in terms of Bessel functions. For the
scalar case $r=1$, the following is known (see Theorem 1.13 of
\cite{KMVV04}): there exists $\delta > 0$
such that for $x \in (1-\delta, 1)$ we have 
\begin{equation} \label{Pnnear1scalar} 
	P_n(x) = \frac{D(\infty)}{2^n \sqrt{W(x)}}
	\frac{\sqrt{n\pi \arccos x}}{(1-x^2)^{1/4}}
	\begin{pmatrix} \cos(\zeta(x)) & \sin(\zeta(x)) 
		\end{pmatrix} \left(I_{2} + \mathcal{O}(n^{-1}) \right)
		\begin{pmatrix} J_{\alpha}(n \arccos x) \\
			J_{\alpha}'(n \arccos x) \end{pmatrix} 
		\end{equation}
as $n \to \infty$, where $J_{\alpha}$ is the Bessel
function of the first kind and order $\alpha$ and
$\zeta(x)$ is a certain explicit function that
depends on the weight $W$. 

In the matrix valued generalization of \eqref{Pnnear1scalar}, 
it turns out that Bessel functions of various orders appear.
The orders of the Bessel functions in the asymptotics near $1$
are determined by the parameters $\alpha_j$ introduced
in \eqref{alphajdef}. 

We write
\begin{equation} \label{Jvecalphadef}
J_{\vec{\alpha}}(x) =
	\diag\left(J_{\alpha_1}(x), \ldots, J_{\alpha_r}(x)
	\right) \end{equation}
for the diagonal matrix containing the Bessel function
of orders $\alpha_1, \ldots, \alpha_r$ on the diagonal,
and similarly for $\left(J_{\vec{\alpha}}\right)'(x)$. 
We also use  
\begin{equation} \label{Azdef}
	A(z) = \frac{(z+1)^{1/2} + (z-1)^{1/2}}{\sqrt{2}} D(z)^{-1} V(z),
\end{equation}
with the principal branch of the square roots and
$V$ is defined in \eqref{Vdef}.

\begin{theorem} \label{theorem16}
We make the same assumptions as in Theorem \ref{theorem14},
and we let $Q$ and $A$ be given by \eqref{HQLambda} and \eqref{Azdef}.
Then there exists $\delta>0$ such that for $x\in(1-\delta,1)$,
\begin{multline} \label{Pnnear1}
	P_n(x) \sqrt{W(x)} 
	= \frac{\sqrt{\pi n \arccos x}}{2^{n} (1-x^2)^{1/4}} D(\infty) 
	\begin{pmatrix} \frac{A_+(x) + A_-(x)}{2} & 
			\frac{A_+(x) - A_-(x)}{2i} \end{pmatrix} \left(I_{2r} + \mathcal{O}(n^{-1}) \right)  \\ 	\times 
	\begin{pmatrix}  
		J_{\vec{\alpha}}(n \arccos x) \\
		\left( J_{\vec{\alpha}}\right)'(n \arccos x) 
	\end{pmatrix} Q(x)^\ast,
\end{multline}
with $J_{\vec{\alpha}}$ as in \eqref{Jvecalphadef}.
\end{theorem}
From \eqref{Azdef} we obtain
\[ A_{\pm}(x) = e^{\pm i \arccos x} D_{\pm}^{-1} V_{\pm}(x) \]
and $A_{\pm}(x)$ turn out to be unitary matrices for $x \in (-1,1)$,
see  \eqref{DVprod} below. The
limit 
\begin{equation} \label{Adef} U_1 = \lim_{z \to 1} A(z) =
	\lim_{z \to 1} D(z)^{-1} V(z) \end{equation}
exists and is also a unitary matrix. It agrees with \eqref{U1def}.

If $W$ is real symmetric, then $A_-(x) = \overline{A_+(x)}$,
and $A$ is a real orthogonal matrix. Then we may write
\begin{equation} \label{Apmpolar} 
	A_{\pm} (x) = U_1 e^{\pm i Z(x)} \end{equation}
with a Hermitian matrix valued function $Z(x)$ that varies analytically and $Z(x) \to O_r$ as $x \to 1-$.
In fact we have $Z(x) = \mathcal O\left(\sqrt{1-x}\right)$
as $x \to 1-$.
Then 
\[ \frac{A_+(x) + A_-(x)}{2}  = U_1 \cos Z(x),
	\quad \frac{A_+(x) - A_-(x)}{2i} =  U_1 \sin Z(x) \]
and we obtain the following.
\begin{corollary}
	If $W(x)$ is real symmetric for every $x \in (-1,1)$,
	then  \eqref{Pnnear1} takes the form
	\begin{multline*}
		P_n(x) \sqrt{W(x)} 
		= \frac{\sqrt{\pi n \arccos x}}{2^{n} (1-x^2)^{1/4}} 
		D(\infty) 
		U_1 \begin{pmatrix} \cos (Z(x)) & 
			\sin (Z(x)) \end{pmatrix} \left(I_{2r} + \mathcal{O}(n^{-1}) \right)  \\ 	\times 
		\begin{pmatrix}  
			J_{\vec{\alpha}}(n \arccos x) \\
			\left( J_{\vec{\alpha}}\right)'(n \arccos x) 
		\end{pmatrix} Q(x)^T.
	\end{multline*}
\end{corollary}
We obtain from Theorem \ref{theorem16} the Mehler-Heine asymptotics at $z=1$.

\begin{theorem} \label{theorem17}
Suppose the weight matrix $W$ satisfies Assumptions \ref{assumptions}. Suppose $\lambda_j$, $j=1, \ldots, r$
be the eigenvalues of $H$ as in \eqref{HQLambda}, \eqref{Lambdadef} and let  
\begin{equation} \label{eq:cj} 
	c_j = 2^{-\alpha_j + \beta} \lim_{x \to 1} \frac{\lambda_j(x)}{(1-x)^{n_j}},
	\quad \text{ for } j=1, \ldots, r, 
	\end{equation}
where $\alpha_j = \alpha + n_j$ as in \eqref{alphajdef}.   
Then we have the following Mehler--Heine asymptotics of
the monic MVOP associated with $W$:
\begin{multline} \label{eq:MHgeneral}
	\lim_{n\to\infty} \frac{2^{n}}{ \sqrt{n\pi}}  P_n\left(\cos \frac{\theta}{n} \right) Q\left(\cos \frac{\theta}{n}\right) 
	\diag\left(c_1^{1/2} n^{-\alpha_1}, \ldots, 
	c_r^{1/2} n^{-\alpha_r} \right) \\
	= D(\infty) U_1
	\diag\left(\theta^{-\alpha_1} J_{\alpha_1}(\theta),
		\ldots, \theta^{-\alpha_r} J_{\alpha_r}(\theta) \right)
\end{multline}
with $Q$ and $U_1$ given by \eqref{HQLambda} and \eqref{U1def} 
and $D$ the matrix valued Szeg\H{o} function. 
\end{theorem}
The proofs of Theorems \ref{theorem16} and \ref{theorem17} are
in Section \ref{subsec43}.

Analogous results hold near $-1$, with Bessel
functions of order $\beta_j = \beta + m_j$.

\subsection{Asymptotics of recurrence coefficients}
The monic MVOPs satisfy a three term recurrence relation:
\begin{equation}\label{eq:TTRR}
xP_n(x)=P_{n+1}(x)+B_nP_{n}(x)+C_nP_{n-1}(x),
\end{equation}
with initial values $P_{-1}(x)=0_r$ and $P_0(x)=I_r$,
see e.g.~\cite{DPS08}. From the Riemann-Hilbert asymptotic
analysis that we present in this paper,  one can obtain large $n$ asymptotics for the recurrence coefficients $B_n$ and $C_n$,
see also \cite{KMVV04} for the scalar case.

Recall that  $\alpha_j$ and $\beta_j$ for $j=1, \ldots, r$
are defined in \eqref{alphajdef} and \eqref{betajdef}.

\begin{theorem} \label{theorem18}
Suppose the weight matrix $W$ satisfies the assumptions
of Theorem \ref{theorem14}.
The recurrence coefficients $B_n$ and $C_n$ in \eqref{eq:TTRR} admit asymptotic expansions  of the form
\begin{equation} \label{eq:BnCnasymp}
    B_n\sim \sum_{k=2}^{\infty}
    \frac{\mathcal{B}_{k}}{n^k}, \qquad
    C_n\sim \frac{1}{4}I_{2r}+\sum_{k=2}^{\infty}\frac{\mathcal{C}_{k}}{n^k}, \qquad n\to\infty,
\end{equation}
with certain computable $r\times r$ matrices $\mathcal{B}_k$, $\mathcal{C}_k$,
for $k=2, 3,  \ldots$.

We have an explicit formula for $\mathcal{B}_2$,
\begin{multline}  \label{b2def} 
	\mathcal{B}_2 =  
	-\frac{1}{16} D(\infty) U_1 \diag 
	\left(4\alpha_1^2 - 1, \ldots, 4 \alpha_r^2 - 1 \right)
	U_1^{-1} D(\infty)^{-1} \\
	 + \frac{1}{16} D(\infty) U_{-1}
	\diag \left( 4\beta_1^2 - 1, \ldots, 4 \beta_r^2-1 \right) 
	U_{-1}^{-1} D(\infty)^{-1},
\end{multline}
where $U_1$ and $U_{-1}$ are given  by \eqref{U1def},
and $D$ is the matrix valued Szeg\H{o} function.
\end{theorem}
The proof is in Section \ref{subsec44}.
The matrices $\mathcal{B}_k$ and $\mathcal{C}_k$ in \eqref{eq:BnCnasymp}
are, in principle, explicitly computable in an iterative manner.
However, the computations become very involved with increasing $k$, and we limit ourselves in Theorem \ref{theorem18}
to the explicit form of $\mathcal{B}_2$. 

If $\alpha_j = \alpha$ and $\beta_j=\beta$ for every $j$,
then \eqref{b2def} simplifies to $\mathcal{B}_2  = \frac{\beta^2-\alpha^2}{4} I_r$, which is  consistent with the formula given in \cite[(1.30)]{KMVV04} for the scalar case. In the scalar case more terms are given in 
\cite[Theorem 1.10]{KMVV04}.


\section{Two examples}\label{sec:examples}
In this section we discuss two examples that arise from the study 
of matrix valued ortho\-gonal polynomials associated to compact symmetric pairs. We find it remarkable that in both
examples the matrix Szeg\H{o} function $D(z)$ can be computed
explicitly. 

\subsection{A Jacobi weight}
Our first example is a family of Jacobi-type matrix orthogonal polynomials which is connected with the matrix valued spherical functions associated to the compact symmetric pair $(\mathrm{SU}(n+1), \mathrm{SU}(n-1))$. This is the result of  a series of papers, starting with \cite{GPT02} and later extended in \cite{RT06, PT07, PR08, PT13}. The weight matrix is given in \cite[Corollary 3.3 and Theorem 3.4]{PT07}. 

Let $\alpha,\beta >-1$, $0<k<\alpha+1$ and $\ell\in \mathbb{N}_0$.  We consider the $(\ell+1)\times (\ell+1)$ weight matrix
\begin{equation}
	\label{eq:weight-Jacobi-NxN}
	W(x) = (1-x)^\alpha (1+x)^\beta  H(x), \quad  H(x) = \Psi(x) T \Psi(x)^T,\qquad x\in[-1,1],
\end{equation}
where $\Psi(x)$ is upper triangular and  $T$ is a constant diagonal matrix. Explicitly, we have
\[ T_{j,j} = \binom{\ell+k-1-j}{\ell-j}\binom{\alpha-k+j}{j},
	\qquad j=0, \ldots, \ell, \]
and
\begin{equation} \label{eq:Psixij} \Psi(x)_{i,j}= \binom{j}{i} 2^{-\frac{\ell-j}{2}-i} (1+x)^{\ell-\frac{j}{2}} (1-x)^{i}, \qquad  0 \leq i \leq j \leq \ell. \end{equation}

We note that the orthogonality interval in \cite{PT07} is $[0,1]$, so  in \eqref{eq:weight-Jacobi-NxN} we have made a change of 
variables to $[-1,1]$ in order to match with the setup in Assumption \ref{assumptions} of the present paper. We have also interchanged the exponents $\alpha$ and $\beta$ in order to be consistent with standard notation for Jacobi polynomials, that we also
follow in this paper, and we take as $\Psi(x)$ the transpose of the corresponding matrix from \cite{PT07}.

The matrix part $H$ of the weight \eqref{eq:weight-Jacobi-NxN} 
has the factorized form
\begin{multline} \label{eq:Hfactors}
	H(x) = \diag\left(1, 1-x, \ldots, (1-x)^{\ell}\right)
	R \\
	\times \diag\left( (1+x)^{\ell}, (1+x)^{\ell-1}, \ldots,
		1+x, 1 \right) 
		R^T \diag\left(1, 1-x, \ldots, (1-x)^{\ell}\right) 
		\end{multline}
with a constant upper triangular matrix $R$ containing the  entries
\[ R_{i,j} = \binom{j}{i} 2^{-\frac{\ell-j}{2}-i} T_{j,j}^{1/2},
	\qquad 
	0 \leq i \leq j \leq \ell. \] 
Thus the entries of $H(x)$ are polynomial in $x$.

For any choice of invertible upper triangular matrix $R$, we can compute the matrix 
Szeg\H{o} function for $H$ explicitly, and this will
allow us to make the asymptotic results explicit for this
class of examples.  

\begin{proposition} \label{prop:example1}
	Let $R$ be any invertible upper triangular matrix. Then
	the matrix Szeg\H{o} function $D_H$ for the matrix weight 
	\eqref{eq:Hfactors} is 	equal to
	\begin{multline} \label{eq:DHexample1} 
		D_H(z) = \diag\left(1, 1-z, \ldots, (1-z)^{\ell}\right)
		R \\
	\times		\diag\left( (z+1)^{\frac{\ell}{2}} \varphi(z)^{- \frac{\ell}{2}},
		(z+1)^{\frac{\ell-1}{2}} \varphi(z)^{-\frac{\ell+1}{2}}, \ldots, \varphi(z)^{-\ell} \right),
	\quad z \in \mathbb C \setminus [-1,1] \end{multline}
	with principal branches of the fractional powers, where
	we recall that $\varphi$ is the conformal map \eqref{phidef}. The matrix  Szeg\H{o} function for $W$ given by 
	\eqref{eq:weight-Jacobi-NxN} is
	\begin{equation} \label{eq:DWexample1} D(z) =  \frac{(z+1)^{\frac{\beta}{2}}(z-1)^{\frac{\alpha}{2}}}{\varphi(z)^{\frac{\alpha+\beta}{2}}}
	D_H(z), \qquad z\in \mathbb C \setminus [-1,1], \end{equation}
with $D_H$ given by \eqref{eq:DHexample1}.
\end{proposition}
\begin{proof}
	Let $D_H(z)$ be defined by \eqref{eq:DHexample1}. We show
	that it satisfies the requirements for the matrix Szeg\H{o} function
	of $H$.
	
The diagonal entries in the last factor on the right-hand side
of \eqref{eq:DHexample1} are 
\[ (z+1)^{\frac{\ell-j}{2}} \varphi(z)^{-\frac{\ell+j}{2}}
	= \left(\frac{(z+1)^{1/2}}{\varphi(z)^{1/2}}\right)^{\ell-j}   \varphi(z)^{-j}, \qquad j=0, \ldots, \ell. 	\]
These entries are analytic in $\mathbb C \setminus [-1,1]$,
since $\frac{(z+1)^{1/2}}{\varphi(z)^{1/2}}$, which
may be initially defined for $z \in \mathbb C \setminus (-\infty,1]$,
has an analytic continuation across $(-\infty,-1)$. 
Hence $D_H$ is analytic in $\mathbb C \setminus [-1,1]$. 

For $z \in \mathbb C \setminus [-1,1]$ the factors on the
right-hand side of \eqref{eq:DHexample1} are invertible matrices,
and therefore $D_H(z)$ is invertible for $z \in \mathbb C \setminus [-1,1]$.
As $z \to \infty$, we have $\frac{z+1}{\varphi(z)} \to \frac{1}{2}$
and $(z-1)^i \varphi(z)^{-j} \to \frac{1}{2}$ if $i=j$ and
 $(z-1)^i \varphi(z)^{-j} \to 0$ if $i < j$. 
 Since $R$ is upper triangular, it then follows that $D_H(z)$ tends to a
 a diagonal matrix with nonzero diagonal entries $ (-1)^j 2^{-\frac{\ell+j}{2}} R_{j,j}$,  $j=0, \ldots, \ell$.
Hence $D_H(\infty)$ exists and is invertible as well.  
 
 Finally, the identity $H(x) = D_{H-}(x) D_{H+}(x)^T = D_{H+}(x) D_{H-}(x)^T$ for
 $x \in (-1,1)$  is immediate from \eqref{eq:Hfactors} and \eqref{eq:DHexample1} and the fact that $\varphi_+(x) \varphi_-(x) = 1$ for $x \in (-1,1)$.
 
 The formula \eqref{eq:DWexample1} for the matrix Szeg\H{o} function for $W$ follows from \eqref{eq:DHexample1}, and the fact that the scalar prefactor in \eqref{eq:DWexample1} is
 the Szeg\H{o} function for the standard Jacobi weight $(1-x)^{\alpha}(1+x)^{\beta}$.
\end{proof}

Next we work out the details of the different asymptotic expansions in the case of a $2\times 2$ matrix valued weight, which corresponds to $\ell = 1$ in \eqref{eq:weight-Jacobi-NxN}. Up to an inessential scalar factor $\frac{k}{p}$, we have 
\begin{equation} \label{example1}
	W(x) 	=
	(1-x)^{\alpha}(1+x)^{\beta}H(x),
	\qquad 	H(x)
	=\frac{1}{4} 
	\begin{pmatrix}
		4+2p+2px & 2(1-x) \\  2(1-x) &  (1-x)^2
	\end{pmatrix},
\end{equation}
with $p = k (\alpha+1-k)^{-1} > 0$, and $H(x)$ depends 
on the parameter $p$ only. 

\begin{figure}[t]
	\centering
	\includegraphics[width=\linewidth]{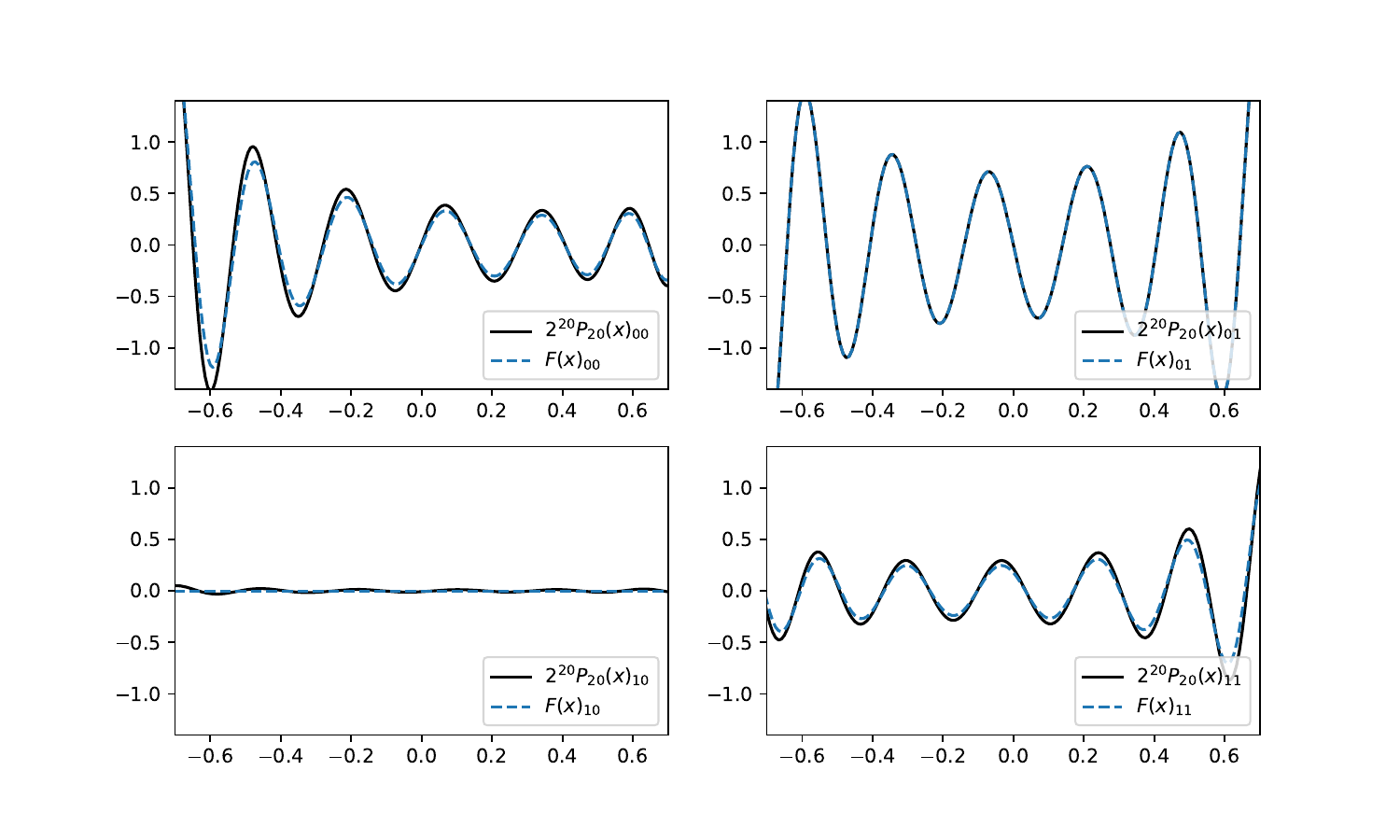}
	\caption{ Plot of the entries of the scaled Jacobi MVOP  $2^{n}P_n(x)$ for $n=20$ and $(\alpha,\beta,k)=(1,2,1)$ (solid line). The approximation $F^{\rm inner}(x)$ (dashed line) is given in formula \eqref{eq:F_Jacobi_inner} of Corollary \ref{cor:asymp-Jacobi2}. We observe that the $(2,1)$ entry is of order $\mathcal{O}(n^{-1})$, since the $(2,1)$ entry of the leading term in \eqref{eq:F_jacobi_outer} is $0$.}
	\label{fig:jacobi_inner}
\end{figure}

\begin{corollary}
\label{cor:asymp-Jacobi2}
	The monic MVOP $P_n$ associated with the weight matrix 
	\eqref{example1} has the following asymptotic behavior as $n \to \infty$:
	\begin{enumerate}
	    \item[\rm (a)] For $z \in \mathbb C \setminus [-1,1]$,
	\begin{equation}
	\label{eq:Jacobi_outer}
	\frac{2^n P_n(z)}{\varphi(z)^n} 
	= 
	F^{\rm outer}(z) \left(I_2 + \mathcal{O}(n^{-1})\right),
	\end{equation}
	where 
	\begin{equation}
	\label{eq:F_jacobi_outer}
	F^{\rm outer}(z) = 	\left(\frac{\varphi(z)}{2}\right)^{\frac{\alpha+\beta+1}{2}}
\frac{1}
{2(z-1)^{\frac{2\alpha+3}{4}}
		(z+1)^{\frac{2\beta+3}{4}}
		(1-\varphi(z))}
			\begin{pmatrix}
			(\varphi(z)-1)^2 & 4\varphi(z)\\
			0 & \varphi(z)+1
		\end{pmatrix}.
	\end{equation}
    \item[\rm (b)] For $x\in(-1,1)$,
	\begin{equation}
	\label{eq:Jacobi_inner}
	2^n P_n(x) 
	= 
	F^{\rm inner}(x)+ \mathcal{O}(n^{-1}),\qquad x\in(-1,1),
	\end{equation}
	where
	\begin{equation} \label{eq:F_Jacobi_inner}
		F^{\rm inner}(x) = \frac{2^{- \frac{\alpha+\beta}{2}}}
		{(1-x)^{\frac{2\alpha+3}{4}}(1+x)^{\frac{2\beta+3}{4}} } \\
			\begin{pmatrix}
				(1-x) \cos( \gamma(x)) &-2  \cos(\gamma(x))\\ 
				0 & -\frac{1}{\sqrt{2}} \sqrt{1+x} \cos\left(\gamma(x)+\frac{\theta(x)}{2}\right)
			\end{pmatrix},	
	\end{equation}
	with $\theta(x)=\arccos x$ and
	\begin{equation}\label{eq:gammax}
		\gamma(x)
		=
		\left(n+1+\frac{\alpha+\beta}{2}\right)\arccos x-\frac{\alpha\pi}{2}-\frac{\pi}{4}.
	\end{equation}
	\item[\rm (c)] Mehler--Heine asymptotics near $z=1$:
\begin{multline*}
\lim_{n\to\infty} \frac{2^{n}}{\sqrt{n\pi}}  P_n\left(\cos \frac{\theta}{n} \right) Q\left(\cos \frac{\theta}{n}\right) 
	\diag\left(c_1^{1/2} n^{-\alpha},c_2^{1/2} n^{-\alpha-2} \right)\\
=
\frac{2^{-\frac{\alpha+\beta}{2}-2}}{\sqrt{1+p}}
\begin{pmatrix}
2p & 2\sqrt{p}\\
-1 & \sqrt{p}
\end{pmatrix}
\begin{pmatrix}
\theta^{-\alpha} J_{\alpha}(\theta) & 0\\
0 & \theta^{-\alpha-2}J_{\alpha+2}(\theta)
\end{pmatrix},
\end{multline*}
where the constants are
\[
c_1=
2^{-\alpha+\beta} \frac{1+p}{p},\qquad
c_2
=
2^{-\alpha-4+\beta} \frac{p}{1+p}.
\]
\end{enumerate} 
\end{corollary}	

See Figure~\ref{fig:jacobi_inner} for a plot of the four
entries of $2^nP_n(x)$ on the interval $[-1,1]$ for
the value $n=20$, together with a plot of the entries
of the approximation \eqref{eq:F_Jacobi_inner}. 

\begin{proof}
The matrix Szeg\H{o} function for the $2\times 2$ weight $W$ from \eqref{example1} is
\begin{equation}\label{eq:Dp_example1}
D(z) = 
	\frac{(z-1)^{\frac{\alpha}{2}}
	(z+1)^{\frac{\beta}{2}} }{\varphi(z)^{\frac{\alpha + \beta}{2}}}
	D_H(z), \qquad
			D_H(z) = 
		\frac{1}{4} 
			\begin{pmatrix} 2 \sqrt{p} (1+\varphi(z)^{-1}) & 4\varphi(z)^{-1} \\ 0 & -(1-\varphi(z)^{-1})^2 \end{pmatrix}.
\end{equation}
This follows from \eqref{eq:DHexample1} with $\ell=1$:
after  removing the scalar factor $\sqrt{k}/\sqrt{p}$, we obtain
	\[ D_H(z) = \frac{1}{4} 
	\begin{pmatrix} 
	 2 \sqrt{2 p}	\sqrt{z+1} \varphi(z)^{-1/2} & 4 \varphi(z)^{-1} \\
		0 & 2(1-z) \varphi(z)^{-1}	\end{pmatrix}. \]
	This leads to \eqref{eq:Dp_example1} because of the two
	identities
	$ \sqrt{2} \sqrt{z+1} = \varphi(z)^{1/2} + \varphi(z)^{-1/2}$ and
	 $ 2z=\varphi(z) + \varphi(z)^{-1}$.
From \eqref{eq:Dp_example1}, we obtain the limit behavior
\begin{equation} \label{eq:Dinf_example1}
	D(\infty)= 2^{-\frac{\alpha+\beta}{2}} D_H(\infty) = 
	2^{-\frac{\alpha+\beta}{2}-2} \begin{pmatrix} 2 \sqrt{p} & 0 \\
		0 & - 1 \end{pmatrix}.	\end{equation}
With this information, we can apply Theorem \ref{theorem14} to obtain the outer asymptotics of $P_n$ as stated in part (a).
\medskip 

For the inner asymptotics, we note that the weight $W(x)$ given by \eqref{example1} is real symmetric on $[-1,1]$, so by \eqref{eq:innerasymp} we only need $D_+(x)^{-1}$.
Write $\varphi_{\pm}(x)=e^{\pm i\theta(x)}$ with $\theta(x)=\arccos(x)$.
From \eqref{eq:Dp_example1}, we obtain 
\begin{align*}
D_{+}(x)
& =
e^{-\frac{i}{2} \left((\alpha+\beta) \theta(x)-\alpha\pi\right)}
(1-x)^{\frac{\alpha}{2}} (1+x)^{\frac{\beta}{2}}
	\begin{pmatrix}
		\sqrt{\frac{p}{2}} \sqrt{1+x} e^{-i \theta(x)/2}  &
		e^{-i\theta(x)}  \\
		0 & \frac{1}{2}(1-x) e^{-i \theta(x)}
	\end{pmatrix},
\end{align*}
and therefore
\[  D_{+}(x)^{-1}
	=
	\frac{\sqrt{2}\, e^{\frac{i}{2}\left((\alpha+\beta+1)\theta(x)-\alpha\pi\right)}}
	{\sqrt{p} (1+x)^{\frac{1+\beta}{2}} (1-x)^{1+\frac{\alpha}{2}}}
	\begin{pmatrix}
		1-x & -2 \\
		0 & \sqrt{2p}\sqrt{1+x}e^{i \theta(x)/2}
	\end{pmatrix}.
	\]
Hence 
\begin{multline*}
	\Re \left(e^{i\left(n+\frac{1}{2}\right)\theta(x)-\frac{\pi i}{4}} D_+(x)^{-1} \right) \\
	=
	\frac{\sqrt{2}}{\sqrt{p} 
	(1+x)^{\frac{1+\beta}{2}} (1-x)^{1+\frac{\alpha}{2}}}
	\begin{pmatrix}
		(1-x) \cos( \gamma(x)) &-2 \cos(\gamma(x))\\ 
		0 & \sqrt{2p} \sqrt{1+x} \cos\left(\gamma(x)+\frac{\theta(x)}{2}\right)
		\end{pmatrix},	
		\end{multline*}
		where the phase function $\gamma(x)$ is given by \eqref{eq:gammax}. Also $D(\infty) = 2^{-\frac{\alpha+\beta}{2}} D_{H}(\infty)$ and $D_H(\infty)$ is given by \eqref{eq:Dinf_example1}. 
	Using this in the formula \eqref{eq:innerasymp} of Theorem \ref{theorem15}, we obtain the result of part (b). 

\medskip
The Mehler--Heine asymptotics of part (c) follows from Theorem \ref{theorem17}. The eigenvalues $\lambda_{1,2}$ of the matrix $H(x)$ in \eqref{example1} can be computed explicitly, and as $x\to\pm 1$, they behave as follows:
		\begin{equation}\label{eq:eigsHp_asymp}
			\begin{aligned}
				\lambda_1(x)=1+p+\mathcal{O}(x-1),\qquad
				\lambda_1(x)&=2+\mathcal{O}(x+1),\\
				\lambda_2(x)=\frac{p}{4(1+p)}(x-1)^2+\mathcal{O}((x-1)^3),\qquad
				\lambda_2(x)&=\frac{p}{4}(x+1)+\mathcal{O}((x+1)^2).
			\end{aligned}
		\end{equation}
Therefore, the exponents are
$\alpha_1=\alpha$ and $\alpha_2=\alpha+2$, and the constants from \eqref{eq:cj} are
\[
\begin{aligned}
    c_1&=2^{-\alpha_1+\beta}\lim_{x\to 1}\frac{\lambda_1(x)}{(1-x)^{n_1}}
    =
    2^{-\alpha+\beta}\lim_{x\to 1}\lambda_1(x)=2^{-\alpha+\beta} \frac{1+p}{p},\\
    c_2&=2^{-\alpha_2+\beta}\lim_{x\to 1}\frac{\lambda_2(x)}{(1-x)^{n_2}}
    =
    2^{-\alpha-2+\beta}\lim_{x\to 1}\frac{\lambda_2(x)}{(1-x)^2}
    =
    2^{-\alpha-4+\beta} \frac{p}{1+p}.
\end{aligned}
\]
We can also calculate the matrix $U_1$ using the explicit expressions for $D_H(z)$ using \eqref{eq:Dp_example1}, as well as $V(z)$:
\[
\begin{aligned}
    V(z)&=(z-1)^{\alpha/2}(z+1)^{\beta/2} Q(z)\tilde{\Lambda}(z)^{1/2},
 \end{aligned}
\]
where $\tilde{\Lambda}(z)^{1/2}=\diag (\lambda_1^{1/2}(z),\lambda_2^{1/2}(z))$ and we use the normalised matrix of eigenvectors:
		\[
		Q(z)
		=
		\begin{pmatrix}
			\frac{1}{\sqrt{1+\rho_2(z)^2}} & -\frac{1}{\sqrt{1+\rho_1(z)^2}}\\
			-\frac{\rho_{2}(z)}{\sqrt{1+\rho_2(z)^2}} & \frac{\rho_1(z)}{\sqrt{1+\rho_1(z)^2}}
		\end{pmatrix},\qquad
		\rho_{1,2}(z)=
		\frac{2\lambda_{1,2}(z)-2-p-pz}{z-1}. \]
Using this information, we can calculate the matrix $U_1$ given by \eqref{U1def} explicitly:
\begin{equation} \label{U1Jacobi}
U_1
	=\lim_{z\to 1} D^{-1}(z)V(z)=
\frac{1}{\sqrt{1+p}}
\begin{pmatrix}
\sqrt{p} & 1\\
1 &-\sqrt{p}
\end{pmatrix}.
\end{equation}
Then, we combine it with $D(\infty)$ in \eqref{eq:Dinf_example1}, and the right hand side of the Mehler--Heine asymptotic formula
\eqref{eq:MHgeneral} becomes 
\begin{multline*}
D(\infty)A\,\diag \left(\theta^{-\alpha_1}J_{\alpha_1}(\theta),\theta^{-\alpha_2}J_{\alpha_2}(\theta)\right)\\
=
\frac{2^{-\frac{\alpha+\beta}{2}-2}}{\sqrt{1+p}}
\begin{pmatrix}
2p & 2\sqrt{p}\\
-1 & \sqrt{p}
\end{pmatrix}
\begin{pmatrix}
\theta^{-\alpha} J_{\alpha}(\theta) & 0\\
0 & \theta^{-\alpha-2}J_{\alpha+2}(\theta)
\end{pmatrix},
\end{multline*}
which completes the proof of part (c).
\end{proof}

By a zero of a matrix valued polynomial $P_n$, one commonly means a zero 
of the determinant of $P_n$. If $P_n$ is a matrix valued orthogonal
polynomial with respect to an a.e. positive definite weight matrix on $[-1,1]$,
then it is known that all zeros of $P_n$ are in $[-1,1]$.
The multiplicity is at most $r$ if $r$ is the size of $P_n$, see \cite[Theorem 1.1]{DL96}.

From \eqref{eq:F_Jacobi_inner} one gets an asymptotic formula
for $\det P_n(x)$, $x \in (-1,1)$ as $n \to \infty$.
From the determinant of the matrix part of \eqref{eq:F_Jacobi_inner}, we conclude that to leading order the zeros of $P_n$
come from the solutions of $\cos(\gamma(x)) = 0$ and
$\cos(\gamma(x) + \frac{\theta(x)}{2}) = 0$. That is
\begin{equation}
    \label{eq:zeros_J_group1}
 x = \cos \left(\frac{\frac{\alpha}{2}\pi + \frac{3\pi}{4} + k \pi}{n+1+\frac{\alpha+\beta}{2}} \right), \qquad k \in \mathbb Z
 \end{equation}
and
\begin{equation}
\label{eq:zeros_J_group2}
x = \cos \left(\frac{\frac{\alpha}{2}\pi + \frac{3\pi}{4} + k \pi}{n+\frac{3}{2}+\frac{\alpha+\beta}{2}} \right), \qquad k \in \mathbb Z.
\end{equation}
The zeros come in two groups, see Figure \ref{fig:zeros_Jacobi}.

\begin{figure}[t]
	\centering
	\includegraphics[width=0.8\linewidth]{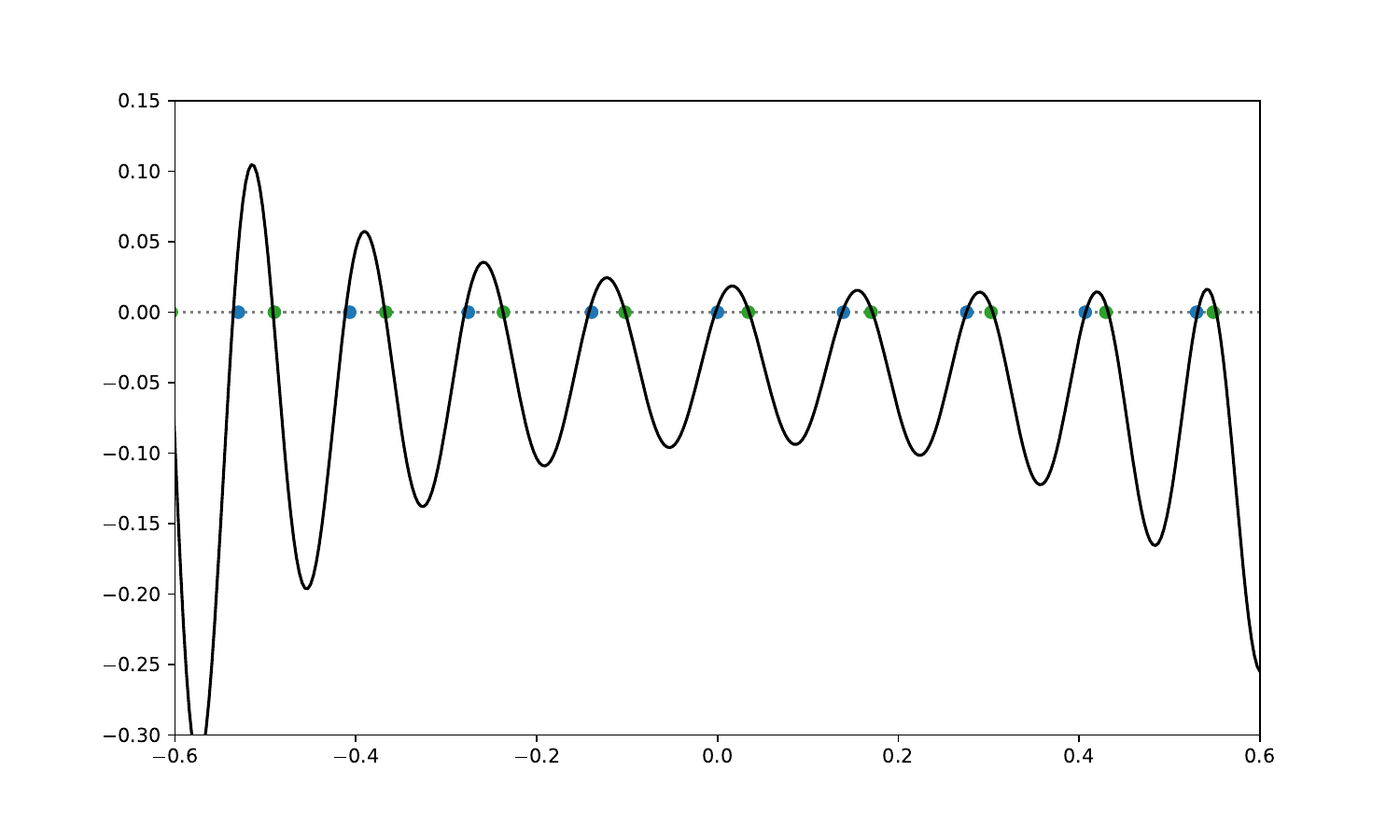}
	\caption{Plot of the determinant of the scaled Jacobi MVOP $2^{20}P_{20}$ for $(\alpha,\beta, k) = (1,2,1)$ (solid line). The blue dots are solutions of \eqref{eq:zeros_J_group1} and the green dots are solutions of \eqref{eq:zeros_J_group2}.}
	\label{fig:zeros_Jacobi}
\end{figure}


\begin{remark}
Using formula \eqref{eq:Hfactors}, in this example we can actually calculate the order of vanishing of the eigenvalues at $z=\pm 1$ for general $\ell$: $\alpha_k = \alpha + 2k-2$, $\beta_k=\beta + k-1$ for $k=0, \ldots, \ell$.
\end{remark}

Regarding the recurrence coefficients, we have the following result:
\begin{corollary}
    The recurrence coefficients for the Jacobi weight \eqref{example1} have the following asymptotic behavior as $n\to\infty$:
    $B_n =  \frac{\mathcal{B}_2}{n^2} + \mathcal{O}(n^{-3})$ 
    with 
	\begin{align} \label{eq:BnCnGegenbauer1}
		\mathcal{B}_2 & = \frac{1}{4} 
		\begin{pmatrix} (\beta+1)^2 - \alpha^2 & 0 \\ 0 
			& \beta^2 - (\alpha+2)^2 \end{pmatrix}  - \frac{\alpha+1}{1+p}
			\begin{pmatrix} 1 & 2 p \\
				\frac{1}{2} & -1 \end{pmatrix}  
	\end{align}
and $C_n = \frac{1}{4} I_2 + \mathcal{O}(n^{-2})$.
\end{corollary}

\begin{proof} In view of Theorem \ref{theorem18} we only need
	to verify the expression \eqref{eq:BnCnGegenbauer1}.
    We have the matrix $U_1$ from
    \eqref{U1Jacobi}, and we can similarly compute
    \[
    U_{-1}
    =\lim_{z\to -1} D^{-1}(z)\widehat{V}(z)
    =
    \begin{pmatrix}
    0 & -1 \\ -1 & 0
    \end{pmatrix}.    
     \]
     Therefore, we obtain from this and \eqref{eq:Dinf_example1}
     \[
     \begin{aligned}
    D(\infty) U_1
    =
    \frac{2^{-\frac{\alpha+\beta}{2}-2}}{\sqrt{1+p}}
    \begin{pmatrix}
    2p & -2\sqrt{p}\\
    -1 & -\sqrt{p}
    \end{pmatrix}, \qquad 
    D(\infty) U_{-1}
    =
    2^{-\frac{\alpha+\beta}{2}-2}
    \begin{pmatrix}
    0 & -2\sqrt{p}\\
    1 & 0
    \end{pmatrix}.
    \end{aligned}
    \]
    From \eqref{eq:eigsHp_asymp} we have $\alpha_1 = \alpha$, 
    $\alpha_2= \alpha+2$, $\beta_1=\beta$ and $\beta_2=\beta+1$.
    Using all this in \eqref{b2def}, we obtain the expression 
    \eqref{eq:BnCnGegenbauer1} for $\mathcal{B}_2$.
\end{proof}

\begin{remark}
The previous result is consistent with the explicit recurrence coefficients for matrix Jacobi polynomials. However these coefficients have rather complicated expressions. They can be calculated using the approach of shift operators given in \cite{KdlRR17} and by extensive use of Maple. 
For reasons of space, we omit the expression of $B_n$. The coefficient $C_n$ has a closed factorized form given by
\begin{multline*}
C_n = \frac{4n(\alpha+n+1)}{(k+n)(\alpha+\beta+2n+1)(\alpha+\beta+2n+2)}  \begin{pmatrix} 1 & 0 \\ \frac {\alpha-k+1}
{\alpha+n+\beta-k+2} & 1 \end{pmatrix} \\
\times \begin{pmatrix} 
\frac {(\alpha+\beta+n+1)
 (\alpha+\beta+n-k+2)(\beta+n)(k+n-1)}{(\alpha+\beta+2\,n+1)(\alpha+\beta+2\,n)(\alpha+\beta+n+1-
k)} & -\frac{k}{(k+n)}
\\ 0 & \frac {(\alpha+\beta+n+2)(\alpha+\beta+n+1-k)(\beta+n+1)(k+n+1)}{
(\alpha+\beta+n-k+2)(\alpha+\beta+2\,n+2)(\alpha + \beta+2\,n+3)} 
\end{pmatrix} \\
\times
 \begin{pmatrix} 1 & 0 \\ -\frac {\alpha-k+1
}{\alpha+n+1+\beta-k} & 1 
\end{pmatrix}.
\end{multline*}
The coefficient $C_n$ is written in terms of the parameter $k$, which is related to $p$ as in \eqref{example1}.
It indeed satisfies $C_n = \frac{1}{4} I_2 + \mathcal{O}(n^{-2})$
as $n \to \infty$.
 \end{remark}

\subsection{A Gegenbauer weight}
The second example is a family of matrix valued Gegenbauer-type polynomials,
introduced in \cite{KdlRR17} and is a one parameter extension of
\cite{KvPR12, KvPR13}. 
Let $K$ be the
constant matrix
\begin{equation} \label{eq:K} K_{i,j}=K_i(j,1/2,2\ell),
	\qquad i,j\in \{0,\ldots,2\ell\}.
\end{equation}
where $K_n(x,p,N)$ is the Krawtchouk polynomial, see e.g.~\cite{KLS2010} or \cite[\S 18.19]{DLMF}. For $\ell\in \frac12 \mathbb{N}_0$ and $\nu>0$, we consider $(2\ell+1)\times(2\ell+1)$ weight matrix
\begin{equation}
	\label{eq:form_weight_Chebyshev}
	W(x)= (1-x^2)^{\nu-\frac12} H(x), \qquad H(x) = 
	\Psi(x)T\Psi(x)^\ast, \qquad x\in(-1,1),
\end{equation}
where $\Psi(x) = K\Upsilon(x)K$ and $\Upsilon(x)$, $T$ are diagonal matrices with entries 
\begin{align*} \Upsilon(x)_{j,j} & =
e^{\frac{j\pi i}{2}}
\binom{2\ell}{j}(1-x)^{\frac{j}{2}}(1+x)^{\ell-\frac{j}{2}}, \\
T_{j,j} &=  2^{-6\ell-1} \frac{(2\nu+2\ell)_{2\ell+1}} {(\nu+\tfrac12)_{2\ell}}\sum_{j=0}^{2\ell} \binom{2\ell}{j} \frac{(\nu)_j}{(\nu+2\ell-j)_j}.
\end{align*}
This factorization for the weight matrix is taken from \cite[Theorem 3.1]{vPR16}. The weight $W$ coincides with that in \cite[Definition 2.1]{KdlRR17} for $\nu>0$ and with that in \cite{KvPR12} for $\nu=1$. 
The matrix $H$ is a matrix polynomial in $x$. This follows from \cite{KdlRR17} or from a direct computation using the above expressions. Note that the even diagonal entries of $\Upsilon$ are real and the odd diagonal entries are purely imaginary. The fact that $H$ has polynomial entries with real coefficients relies on particular properties of the Krawtchouk polynomials in the entries of the matrix $K$ and the matrix $\Upsilon$, see \cite[Corollary 3.7, Remark 3.8]{vPR16}.

Let $\xi^{\frac{j}{2}}(z)$ be the function 
$$\xi^{\frac{j}{2}}(z) = \left( \frac{z-1}{z+1} \right)^{\frac{j}{2}}, \qquad j=0,\ldots,2\ell,$$
with principal branches of the fractional powers, so that for $x\in(-1,1)$, we have 
$$\xi^{\frac{j}{2}}_\pm(x) = e^{\pm\frac{j\pi i}{2}} \left( \frac{1-x}{1+x} \right)^{\frac{j}{2}},$$
where $\pm$ indicates boundary values from the left (right) of the interval $(-1,1)$. The matrix $H$ has the factorized form:
\begin{equation}
    \label{eq:Hnu}
 H(x) = (1+x)^{2\ell} K \diag\left(1, \xi^{\frac12}_+(x), \ldots, \xi^{\ell}_+(x)\right) R
R^T  \diag\left(1, \xi^{\frac12}_-(x), \ldots, \xi^{\ell}_-(x)\right) K^T,
\end{equation}
where
\begin{equation}
    \label{eq:RGegenbauer}
 R=\diag\left(\binom{2\ell}{0},(-1)\binom{2\ell}{1},\ldots, (-1)^{2\ell}\binom{2\ell}{2\ell}\right)K T^{\frac12}.
\end{equation}

\begin{proposition}
\label{prop:example2}
Let $R$ be an invertible matrix. Then the matrix Szeg\H{o} function $D_{H}$ for the matrix weight \eqref{eq:Hnu}  is
	\begin{equation} 
	\label{eq:DHexample2} 
	D_{H}(z) =\frac{(1+z)^\ell}{\varphi(z)^\ell} K \diag\left(1, \xi(z)^{\frac12}, \ldots, \xi(z)^{\ell}\right) R,
	\end{equation}
	for $z \in \mathbb C \setminus [-1,1]$
	with principal branches of the fractional powers.
	
    The matrix Szeg\H{o} function for the weight $W$	is 
    $$D(z) =\frac{(z^2-1)^{\frac{\nu}{2}-\frac{1}{4}}}{\varphi(z)^{\nu-\frac{1}{2}}} D_{H}(z).$$
\end{proposition}

\begin{proof}
The entries $\xi^{\frac{j}{2}}(z)$ of the diagonal matrix in $D_{H}(z)$ are analytic in $\mathbb{C}\setminus[-1,1]$, so  $D_{H}(z)$ is an analytic function on $\mathbb{C}\setminus[-1,1]$. On the other hand, the factor $(1+z)^\ell\varphi(z)^{-\ell}$ is analytic in $\mathbb{C}\setminus[-1,1]$ for $\ell\in \mathbb{N}_0$ and has an analytic continuation across $(-\infty, -1)$ for fractional values of $\ell\in \frac12\mathbb{N}_0$. Therefore $D_{H}$ is analytic in $\mathbb{C}\setminus[-1,1]$. 

Moreover, since the factors on the right hand side of  \eqref{eq:DHexample2} are invertible,  $D_{H}(z)$ is invertible for  $z\in \mathbb{C}\setminus[-1,1]$. As $z\to \infty$ we have that  $\xi(z)^{\frac{j}{2}} \to 1$ and $(z+1)^\ell\varphi(z)^{-\ell} \to 2^{-\ell}$. Therefore the limit $D_{H}(\infty)$ exists and is invertible. 

For $x\in(-1,1)$, by taking $\pm$ boundary values we get $H(x) = D_{H,+}(x)D_{H,-}(x)^T,$
which coincides with  \eqref{eq:Hnu}.

The matrix Szeg\H{o} function for the weight $W$ follows from \eqref{eq:DHexample2} and \eqref{eq:form_weight_Chebyshev} by using that the scalar prefactor in \eqref{eq:DHexample2} is the Szeg\H{o} function for the Gegenbauer weight $(1-x^2)^{\nu-\frac12}$.
\end{proof}
\begin{remark}
As in the previous example, any choice of invertible matrix $R$ gives the matrix Szeg\H{o} function for $H(z)$ explicitly in \eqref{eq:Hnu}. However, in order for $H(z)$ to be real valued on $(-1,1)$, we need the specific matrix $R$ in \eqref{eq:RGegenbauer}.
\end{remark}

		
\begin{figure}[t]
	\centering
	\includegraphics[width=\linewidth]{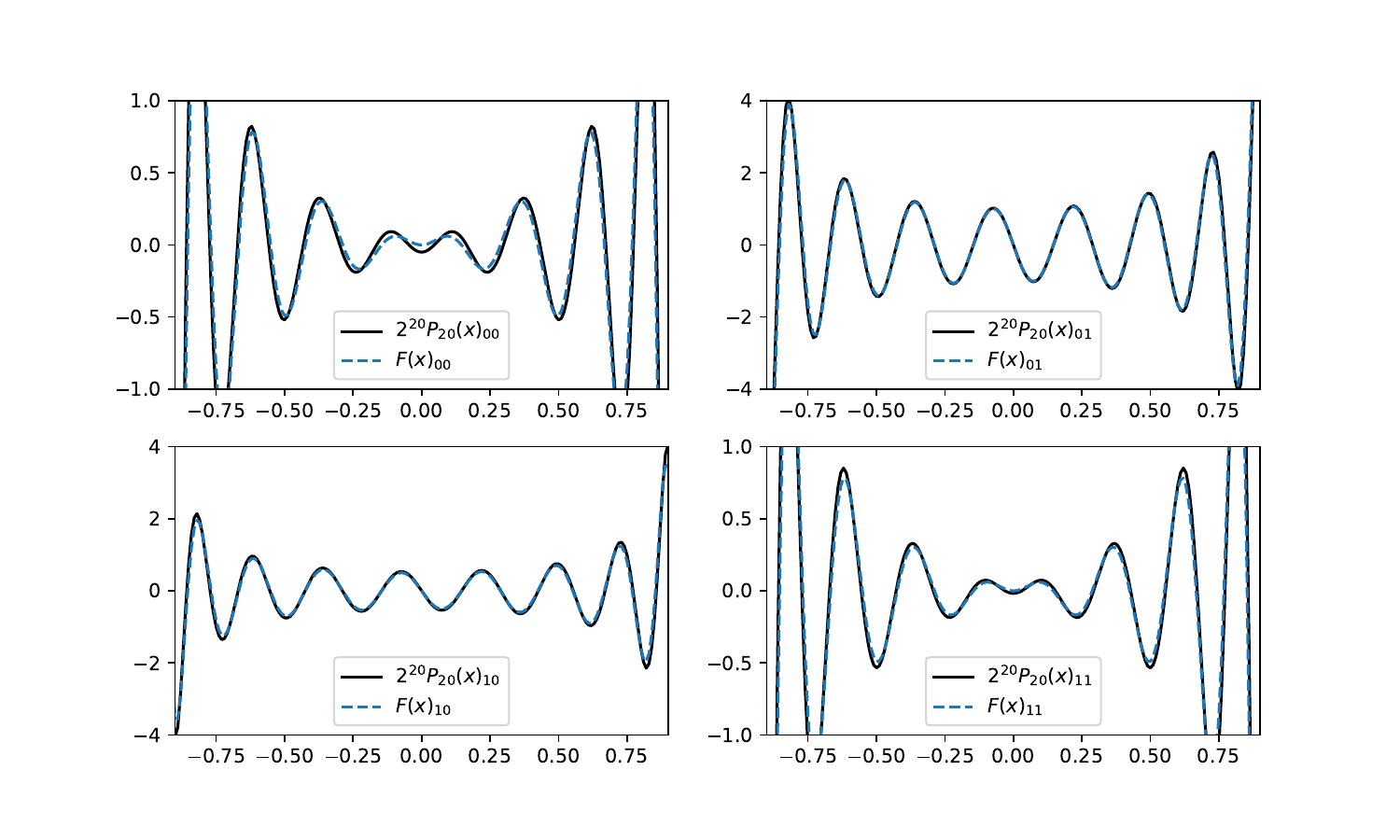}
	\caption{Plot of the entries of the scaled Gegenbauer MVOP $2^{n}P_n(x)$ for $n=20$ and $\nu=\frac12$ (solid line). The approximation $F^{\rm inner}(x)$ (dashed line) is given in formula \eqref{eq:F_Gegenbauer_interval} of Corollary \ref{cor:global-Gegenbauer}.}
	\label{fig:gegenbauer2}
\end{figure}

The weight $W$  \eqref{eq:form_weight_Chebyshev} is an instance of a reducible weight matrix. More precisely, for any $n\in\mathbb{N}$ let $I_{n}$ be the $n \times n$ identity matrix and let $J_{n}$ be the $n \times n$ matrix
	$$
		J_{n}=\sum_{i=0}^{n-1} E_{i, n-1-i},
		$$
		where $E_{i,j}$ indicates the matrix with all $0$ entries, except for a $1$ in the position $(i,j)$, and let $Y$ be given by
		\begin{align} \label{defY}
			Y&=\frac{1}{\sqrt{2}} \begin{pmatrix}
				I_{\ell+\frac{1}{2}} & J_{\ell+\frac{1}{2}} \\
				-J_{\ell+\frac{1}{2}} & I_{\ell+\frac{1}{2}}
			\end{pmatrix}, \quad \ell \in \mathbb{N}_0+\frac12, \qquad Y =\frac{1}{\sqrt{2}}\begin{pmatrix}
				I_{\ell} & 0 & J_{\ell} \\
				0 & \sqrt{2} & 0 \\
				-J_{\ell} & 0 & I_{\ell}
			\end{pmatrix}, \quad  \ell \in  \mathbb{N}_0.
		\end{align}
		We note that $Y$ is orthogonal, i.e $YY^T = I_{2\ell+1}$. The weight matrix $W(x)$ satisfies:
		\begin{equation}
		\label{eq:block-dec-Gegenb}
		YW(x)Y^T = \begin{pmatrix} W_2(x) & 0 \\ 0 & W_1(x) \end{pmatrix},
		\end{equation}
		where $W_1$ and $W_2$ have a strictly lower dimension, see for instance \cite[Theorem 6.5]{KvPR12} and \cite[Proposition 2.6]{KdlRR17}. 
		Given \eqref{eq:block-dec-Gegenb}, it is straightforward to check from the orthogona\-lity property that, if $\widetilde{P}_n(x)$ are monic MVOPs with respect to the original weight $W(x)$, then $Y\widetilde{P}_n(x)Y^T$ are monic MVOPs orthogonal with respect to the weight in block form $\begin{pmatrix} W_2(x) &0\\ 0 & W_1(x)\end{pmatrix}$. Moreover, using the uniqueness  property of the family of monic MVOPs, see also \cite[Corollary 5.6]{KvPR12}, we have
    \[
    Y\widetilde{P}_n(x)Y^T=\begin{pmatrix} P_n(x) &0\\ 0 & Q_n(x)\end{pmatrix},
    \]
    where $P_n(x)$ and $Q_n(x)$ are monic MVOPs with respect to the blocks $W_2(x)$ and $W_1(x)$ respectively.
		
		For $\ell=\tfrac12$, the weight $W$ decomposes into two $1\times 1$ blocks and therefore reduces to a scalar situation. The first nontrivial example is for $\ell=1$, where the weight $W$ decomposes into an irreducible $2\times 2$ block $W_2$ and a $1\times 1$ block $W_1$. This irreducible $2\times 2$ block $W_2$ is, 
		up to the scalar factor $\frac{2\nu+1}{2+\nu}$, the following:
		
		\begin{equation} \label{example2}
			W_2(x) = (1-x^2)^{\nu-1/2}  \begin{pmatrix}
				2(\nu+1) x^2 + 2 \nu & (2\nu+1) \sqrt{2}x \\
				(2\nu+1) \sqrt{2} x & \nu x^2 + \nu +1
			\end{pmatrix}, \qquad -1 < x < 1,
		\end{equation}
		with $\nu > 0$. 

\begin{corollary}
\label{cor:global-Gegenbauer}
The monic MVOP $P_n$ associated with the weight matrix 
\eqref{example2} has the following asymptotic bahevior as $n\to\infty$:
\begin{enumerate}
    \item[\rm (a)] For $z\in \mathbb{C}\setminus [-1,1]$,
    \begin{equation*}
\frac{2^n P_n(z)}{\varphi(z)^n} = F^{\rm outer}(z) \left(I_2+\mathcal{O}(n^{-1})\right),\qquad n\to \infty,
\end{equation*}
where
\begin{equation}
\label{eq:F_Gegenbauer_global}
F^{\rm outer}(z) = 
\frac{\varphi(z)^{\nu+1}}{2^{\nu+2}(z^2-1)^{\frac{\nu}{2}+1}}  
\begin{pmatrix} 2z & -2\sqrt{2} \\ -\sqrt{2} & 2z \end{pmatrix}.
\end{equation}
\item[\rm (b)] For $x\in(-1,1)$,
\begin{equation*}
2^nP_n(x) = F^{\rm inner}(x) +\mathcal{O}(n^{-1}),\qquad 
n \to \infty,
\end{equation*}
where
\begin{equation}
    \label{eq:F_Gegenbauer_interval}
F^{\rm inner}(x) = \frac{2^{-\nu-1}}{(1-x^2)^{\frac{\nu}{2}+1}}
		\cos\left((n+\nu+1)\theta(x)-\frac{\nu\pi}{2}\right)
		\begin{pmatrix}
			2x & -2\sqrt{2}\\
			-\sqrt{2} & 2x
		\end{pmatrix},
\end{equation}	
and $\theta(x)=\arccos x$.

\item[\rm (c)] Mehler--Heine asymptotics near $z=1$:
\begin{multline*}
\lim_{n\to\infty} \frac{2^{n}}{\sqrt{n\pi}}  P_n\left(\cos \frac{\theta}{n} \right) Q\left(\cos \frac{\theta}{n}\right) 
	\diag \left(c_1^{1/2} n^{-\nu+\frac{1}{2}},c_2^{1/2} n^{-\nu-\frac{3}{2}} \right)\\
	=
	\frac{2^{-\nu-\frac{1}{2}}}{\sqrt{1+2\nu}}
\begin{pmatrix}
\sqrt{2}(\nu+1) & -2\sqrt{\nu(\nu+1)}\\
\nu &\sqrt{\nu(\nu+1)}
\end{pmatrix}
\begin{pmatrix}
\theta^{-\nu+\frac{1}{2}} J_{\nu-\frac{1}{2}}(\theta) & 0\\
0 & \theta^{-\nu-\frac{3}{2}}J_{\nu+\frac{3}{2}}(\theta)
\end{pmatrix},
\end{multline*}	
where the constants are
$c_1=3(1+2\nu)$ and $ c_2=\frac{2\nu(1+\nu)}{3(1+2\nu)}$.
\end{enumerate}
\end{corollary}

\begin{proof}
The matrix Szeg\H{o} function for $W_2(x)$ in \eqref{example2} is
		\begin{equation}
		    \label{eq:D2-Gegenbauer}
		D(z) = \frac{(z^2-1)^{\nu/2-1/4}}{\varphi(z)^{\nu-1/2}}  \frac{1}{\varphi(z)} \begin{pmatrix}
        \sqrt{2(\nu+1)}z &
			\sqrt{2\nu}  \\
			\sqrt{\nu+1}  & 
		    \sqrt{\nu}z
		\end{pmatrix},
		\end{equation}
and as a consequence
		\begin{equation}\label{eq:Dinf_Gegenbauer}
			D_2(\infty) 
			= 
			2^{-\nu-\frac{1}{2}}
			\begin{pmatrix}
				\sqrt{2(\nu+1)} & 0\\
				0 & \sqrt{\nu}
			\end{pmatrix}.
		\end{equation}
This follows from conjugating the matrix
Szeg\H{o} function from Proposition \ref{prop:example2} for $\ell=1$ with $Y$
coming from \eqref{defY}. Thus, we obtain up to the factor $(2\nu+1)^{\frac12}(2+\nu)^{-\frac12}$:
$$YD(z)Y^T = \begin{pmatrix} D_2(z) & 0 \\ 0 & D_1(z) \end{pmatrix},$$
where $D_2(z)$ is given by \eqref{eq:D2-Gegenbauer}, using the fact that $\varphi(z)+\varphi(z)^{-1}=2z$. From this and the block decomposition of the weight $W$ in \eqref{eq:block-dec-Gegenb}, we verify that $D_2(z)$ is the matrix Szeg\H{o} function for $W_2(z)$. Then, the matrix $D_2(\infty)$ is obtained directly by taking the limit of \eqref{eq:D2-Gegenbauer} as $z\to \infty$.

With this information, direct calculation 
gives 
\begin{equation}
D_2(\infty)D_2^{-1}(z) 
=
\frac{\varphi(z)^{\nu+\frac{1}{2}}}
{2^{\nu+\frac32}(z^2-1)^{\frac{\nu}{2}+\frac34}} 
\begin{pmatrix} 
2z & -2\sqrt{2}\\
-\sqrt{2} & 2z
\end{pmatrix},
\end{equation}
and application of Theorem \ref{theorem14} gives the result
from part (a).

\medskip
  
For $x\in(-1,1)$, we have from \eqref{eq:D2-Gegenbauer} the boundary value
		\[ 
		e^{i\left(n+\frac{1}{2}\right)\theta-\frac{\pi i}{4}}
		D_{2+}(x)^{-1} 
		= 
		\frac{(1-x^2)^{-\frac{\nu}{2}-\frac{3}{4}} }{\sqrt{2\nu(\nu+1)}}
		e^{i\left(n+\nu+1)\theta(x)-\frac{\nu\pi}{2}\right)} 
		\begin{pmatrix} 
			\sqrt{\nu} x & -\sqrt{2\nu} \\
			-\sqrt{\nu+1} & \sqrt{2(\nu+1)} x 
		\end{pmatrix},
		\]
%
Then, from this, the fact that $W(x)$ is real symmetric on $[-1,1]$ and Theorem \ref{theorem15} we have the inner asymptotics for $x \in (-1,1)$,
\begin{equation}\label{eq:innerasymp_Gegenbauer}
\begin{aligned}
	2^n P_n(x)
	&=
	\frac{\sqrt{2}}{(1-x^2)^{\frac{1}{4}}} D_2(\infty)
	\Re \left[e^{i\left(n+\frac{1}{2}\right)\theta(x)-\frac{\pi i}{4}}D_{2+}(x)^{-1}\right] + \mathcal{O}(n^{-1})\\
	&=
		\frac{2^{-\nu}}{\sqrt{2\nu(\nu+1)}(1-x^2)^{\frac{\nu}{2}+1}} 
		\cos\left((n+\nu+1)\theta-\frac{\nu\pi}{2}\right)\\
		&\times \begin{pmatrix}
			\sqrt{2(\nu+1)} & 0\\
			0 & \sqrt{\nu}
		\end{pmatrix}
		\begin{pmatrix} 
			\sqrt{\nu} x & -\sqrt{2\nu} \\
			-\sqrt{\nu+1} & \sqrt{2(\nu+1)} x 
		\end{pmatrix}+\mathcal{O}(n^{-1})\\
		&=
		\frac{2^{-\nu-1}}{(1-x^2)^{\frac{\nu}{2}+1}}
		\cos\left((n+\nu+1)\theta-\frac{\nu\pi}{2}\right)
		\begin{pmatrix}
			2x & -2\sqrt{2}\\
			-\sqrt{2} & 2x
		\end{pmatrix}
		+\mathcal{O}(n^{-1}),
	\end{aligned}
\end{equation}
where we have used \eqref{eq:Dinf_Gegenbauer}. This 
proves part (b).

\medskip

The eigenvalues $\lambda_{1,2}$ of the matrix part of \eqref{example2} are explicit, and as $x\to 1$ they satisfy
\begin{equation}
	\begin{aligned}
		\lambda_1(x)
		&=
		3(1+2\nu)+2(2+3\nu)(x-1)+\mathcal{O}((x-1)^2),\\
		\lambda_2(x)
		&=
		\frac{8\nu(1+\nu)}{3(1+2\nu)}(x-1)^2+\mathcal{O}((x-1)^3),
	\end{aligned}
\end{equation}
so $n_1=0$ and $n_2=2$, and the exponents are
\[
\alpha_1=\nu-\tfrac{1}{2}, \qquad \alpha_2=\nu-\tfrac{1}{2}+2=\nu+\tfrac{3}{2}.
\]
The constants in this example are 
\[
\begin{aligned}
    c_1&=2^{-\alpha_1+\beta}\lim_{x\to 1}\frac{\lambda_1(x)}{(1-x)^{n_1}}
    =
    \lim_{x\to 1}\lambda_1(x)=3(1+2\nu),\\
    c_2&=2^{-\alpha_2+\beta}\lim_{x\to 1}\frac{\lambda_2(x)}{(1-x)^{n_2}}
    =
    2^{-2}\lim_{x\to 1}\frac{\lambda_2(x)}{(1-x)^2}=\frac{2\nu(1+\nu)}{3(1+2\nu)},
\end{aligned}
\]
We can also calculate the matrix $U_1$ using the explicit expressions for $D_2(z)$ in \eqref{eq:D2-Gegenbauer}, as well as $V(z)$:
\[
\begin{aligned}
    V(z)&=(z^2-1)^{\frac{\nu}{2}-\frac{1}{4}}Q(z)\tilde{\Lambda}(z)^{1/2},
 \end{aligned}
\]
where $\tilde{\Lambda}(z)^{1/2}=\diag (\lambda_1(z)^{1/2},\lambda_2(z)^{1/2})$ and we use the normalized matrix eigenvectors
\begin{equation}
	Q(z)
	=
	\begin{pmatrix}
		\frac{1}{\sqrt{1+\rho_2(z)^2}} & -\frac{1}{\sqrt{1+\rho_1(z)^2}}\\
		-\frac{\rho_2(z)}{\sqrt{1+\rho_2(z)^2}} & \frac{\rho_1(z)}{\sqrt{1+\rho_1(z)^2}}
	\end{pmatrix},
	\qquad 
	\rho_{1,2}(z)
	=
	\frac{\lambda_{1,2}(z)-(\nu z^2+\nu+1)}{\sqrt{2}(1+2\nu)z}.
\end{equation}

From this and \eqref{eq:D2-Gegenbauer}, we calculate
\begin{equation} \label{eq:A-Gegenbauer}
U_1
=\lim_{z\to 1} D_2^{-1}(z)V(z)
=\frac{1}{\sqrt{1+2\nu}}
\begin{pmatrix}
\sqrt{\nu+1}&-\sqrt{\nu}\\
\sqrt{\nu} & \sqrt{\nu+1}
\end{pmatrix}.
\end{equation}

We combine this with $D_2(\infty)$ from \eqref{eq:Dinf_Gegenbauer}, to obtain
\begin{multline*}
D_2(\infty) U_1
=
\frac{2^{-\nu-\frac{1}{2}}}{\sqrt{1+2\nu}}
\begin{pmatrix}
\sqrt{2(\nu+1)} & 0\\
0 & \sqrt{\nu}
\end{pmatrix}
\begin{pmatrix}
\sqrt{\nu+1}&-\sqrt{\nu}\\
\sqrt{\nu} & \sqrt{\nu+1}
\end{pmatrix}\\
=
\frac{2^{-\nu-\frac{1}{2}}}{\sqrt{1+2\nu}}
\begin{pmatrix}
\sqrt{2}(\nu+1) & -2\sqrt{\nu(\nu+1)}\\
\nu &\sqrt{\nu(\nu+1)}
\end{pmatrix},
\end{multline*}
which we use in the right hand side of the Mehler--Heine asymptotic formula \eqref{eq:MHgeneral} to obtain part (c). 
\end{proof}

From part (b) of Corollary \ref{cor:global-Gegenbauer}, we have
$\det \left( 2^n P_n(x) \right) = \det F^{\rm inner}(x)  + \mathcal{O}(n^{-1})$ for $x \in (-1,1)$, with 
\begin{equation}
    \label{eq:determinant-gegenbauer}
\det F^{\rm inner}(x) = 
- 2^{-2\nu} (1-x^2)^{-\nu-\frac{1}{2}}
\cos^2\left((n+\nu+1)\arccos(x) -\frac{\nu\pi}{2}\right).
\end{equation} 
Thus $\det F^{\rm inner}(x)$ has double zeros on the
interval $(-1,1)$, which gives asymptotic information about the zeros of the Gegenbauer MVOPs as $n\to\infty$. For large $n$,
the zeros of $\det (2^n P_n(x))$ come in close pairs.
However, for finite $n$, the zeros are still simple.
See Figure \ref{fig:zeros_Gegenbauer} for a plot of $\det (2^n P_n)$
with $n=20$ and $\nu = \frac{1}{2}$. 

\begin{figure}[t]
	\centering
	\includegraphics[width=0.8\linewidth]{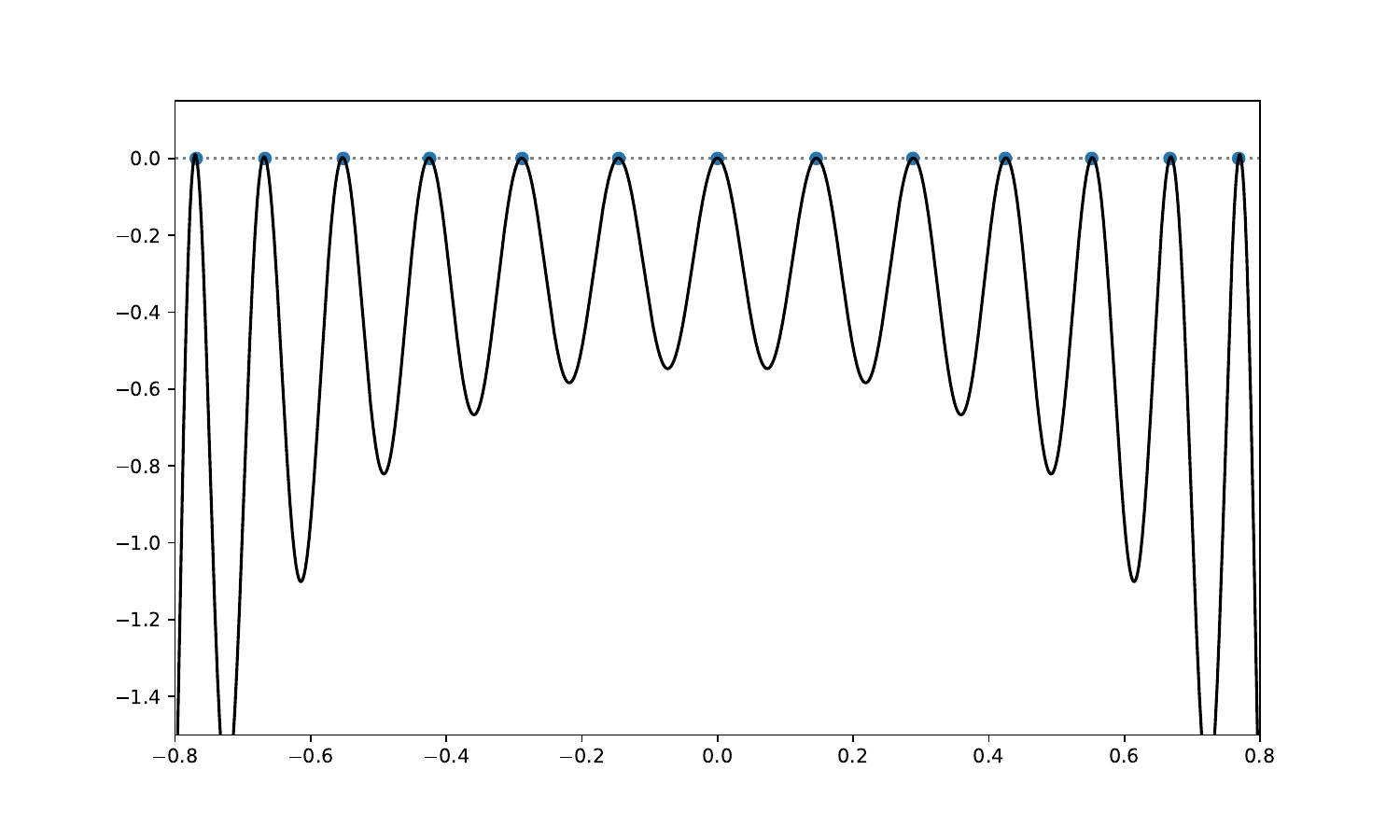}
	\caption{Plot of the determinant of scaled Gegenbauer MVOPs  $2^{20}P_{20}$ for $\nu=\frac12$ (solid line). The blue dots are zeros of \eqref{eq:determinant-gegenbauer}}
	\label{fig:zeros_Gegenbauer}
\end{figure}


Regarding the recurrence coefficients, we have the following result.
\begin{corollary}
    The recurrence coefficients for the Gegenbauer weight \eqref{example2} have the following asymptotic behavior as $n\to\infty$:
    \begin{equation}\label{eq:BnCnGegenbauer2}
        B_n  = 
        \frac{1}{n^2}
        \begin{pmatrix}
        0 & \sqrt{2}(1+\nu)\\
        \frac{\nu}{\sqrt{2}} & 0
        \end{pmatrix}
    +\mathcal{O}(n^{-3}), \qquad
        C_n  = \frac{1}{4}I_2+\mathcal{O}(n^{-2}).
    \end{equation}
\end{corollary}

\begin{proof}
    We have the matrix $U_1$ from \eqref{eq:A-Gegenbauer}, and 
    we can compute in a similar way
    \[
    U_{-1}
    =\lim_{z\to -1} D_2^{-1}(z)\widehat{V}(z)
    =\frac{1}{\sqrt{1+2\nu}}
    \begin{pmatrix}
    -\sqrt{\nu+1}& -\sqrt{\nu}\\
    \sqrt{\nu} & -\sqrt{\nu+1}
    \end{pmatrix}.    
    \]
    Therefore, we obtain by \eqref{eq:Dinf_Gegenbauer}
    \[
    \begin{aligned}
    D_2(\infty) U_1
    &=
    \frac{2^{-\nu-\frac{1}{2}}}{\sqrt{1+2\nu}}
    \begin{pmatrix}
    \sqrt{2}(\nu+1) & \sqrt{2\nu(\nu+1)} \\
    \nu & \sqrt{\nu(\nu+1)}
    \end{pmatrix},\\
    D_2(\infty)  U_{-1}
    &=
    \frac{2^{-\nu-\frac{1}{2}}}{\sqrt{1+2\nu}}
    \begin{pmatrix}
    -\sqrt{2}(\nu+1) & -\sqrt{2\nu(\nu+1)} \\
    \nu & -\sqrt{\nu(\nu+1)}
    \end{pmatrix}.
    \end{aligned}
    \]
Finally, the exponents of the eigenvalues at $z=1$
are $\alpha_1 = \nu-\frac{1}{2}$ and $\alpha_2 = \nu + \frac{3}{2}$, while at $z=-1$ they are 
$\beta_1=\nu-\tfrac{1}{2}$, and $\beta_2=\nu+\tfrac{3}{2}$.
  Collecting all this in \eqref{b2def}, 
    we obtain the coefficient $\mathcal{B}_2$
    in the expansion of $B_n$ and \eqref{eq:BnCnGegenbauer2}
    follows from Theorem \ref{theorem18}.
\end{proof}

\begin{remark}
The previous result is consistent with the explicit recurrence coefficients for matrix Gegenbauer polynomials given in \cite[Proposition 3.3]{KdlRR17}. If we denote by $\widetilde{P}_n(x)$ the MVOPs with respect to the $3\times 3$ Gegenbauer weight, and if  $\widetilde{B}_n$ and $\widetilde{C}_n$ are the recurrence coefficients for $\widetilde{P}_n(x)$, then
$Y\widetilde{B}_nY^T$ and $Y\widetilde{C}_nY^T$ give block diagonal matrices that contain the recurrence coefficients for the MVOPs $P_n(x)$ and $Q_n(x)$.
 
In the case given by the $2\times 2$ block in \eqref{example2}, we start with $3\times 3$ matrix Gegenbauer polynomials, we have $\ell=1$ and recurrence coefficients given explicitly by  
    \[
    \widetilde{B}_n
    =
    \begin{pmatrix}
        0 & \frac{\nu+1}{(n+\nu+1)(n+\nu+2)} & 0\\
        \frac{\nu}{(n+\nu)(n+\nu+1)} & 0 &
        \frac{\nu}{2(n+\nu)(n+\nu+1)} \\
        0 & \frac{\nu+1}{(n+\nu+1)(n+\nu+2)} & 0
    \end{pmatrix}
    \]
    and
    \[
    \widetilde{C}_n
    =
    \frac{n(n+2\nu+1)}{4(n+\nu)(n+\nu+1)}
    \begin{pmatrix}
    1 & 0 & 0\\
    0 & \frac{(n+\nu-1)(n+\nu+2)}{(n+\nu)(n+\nu+1)} & 0\\
    0 & 0 & 1
    \end{pmatrix}
    \]
Then, we conjugate with the matrix $Y$ in \eqref{defY}, with $\ell=1$: 
\[
\begin{aligned}
Y\widetilde{B}_n Y^T
&=
\frac{1}{(n+\nu+1)(n+\nu+2)}
\begin{pmatrix}
    0 & \sqrt{2}(\nu+1) & 0\\
    \frac{\nu}{\sqrt{2}} & 0 & 0\\
    0 & 0 & 0
\end{pmatrix}\\
&=
\frac{1}{n^2}
\begin{pmatrix}
    0 &\sqrt{2}(\nu+1) & 0\\
    \frac{\nu}{\sqrt{2}} & 0 & 0\\
    0 & 0 & 0
\end{pmatrix}+\mathcal{O}(n^{-3}),\\
Y\widetilde{C}_n Y^T
&=
\frac{n(n+2\nu+1)}{4(n+\nu)(n+\nu+1)}
\begin{pmatrix}
    1 & 0 & 0\\
    0 & \frac{(n+\nu-1)(n+\nu+2)}{(n+\nu)(n+\nu+1)} & 0\\
    0 & 0 & 1 
\end{pmatrix}\\
&=
\frac{1}{4}I_3
+
\frac{1}{4n^2}
\begin{pmatrix}
    -\nu(\nu+1) & 0 & 0\\
    0 & -2-\nu(\nu+1) & 0\\
    0 & 0 & -\nu(\nu+1)
\end{pmatrix}
+\mathcal{O}(n^{-3}).
    \end{aligned}
\]
We see that the $2\times 2$ upper blocks indeed agree with the terms up to $\mathcal{O}(n^{-2})$  that we obtain in the asymptotic expansions for $B_n$ and $C_n$ in \eqref{eq:BnCnGegenbauer2}.
\end{remark}

\section{RH steepest descent analysis}
We use the Riemann-Hilbert (RH) problem for MVOP to prove the theorems
stated in Section \ref{sec1}. The steepest descent analysis
of RH problems originates with the work of Deift and Zhou \cite{DZ93}
and was applied to orthogonal polynomials in \cite{BI99, DKMVZ99}
and in many different contexts in subsequent works. 
We follow in particular \cite{KMVV04}, where the RH method is
applied to Jacobi-type OPs in $[-1,1]$, see also \cite{Kui03}.

\subsection{RH problem}
The matrix valued orthogonal polynomial $P_n$ is characterized by a RH problem of size $2r \times 2r$, see \cite{CM12,GIM11}.
It is the upper left block in the solution of the following 
$2r \times 2r $ matrix valued RH problem: we seek $Y:\mathbb{C} \setminus [-1,1] \to\mathbb{C}^{2r\times 2r}$ such that
\begin{enumerate}
    \item $Y:\mathbb C \setminus [-1,1] \to \mathbb C^{2r \times 2r}$ is analytic.
    \item For $-1<x<1$, with this segment oriented from left to right, the matrix $Y$ admits boundary values $Y_{\pm}(x)=\lim\limits_{\varepsilon\to 0+} Y(x\pm i\varepsilon)$, which are related by 
    \begin{equation} \label{Yjump}
	Y_+(x) = Y_-(x) \begin{pmatrix} I_r & W(x) \\ 0_r & I_r
	\end{pmatrix}, \qquad -1 < x < 1. \end{equation}
	\item As $z\to\infty$, we have the asymptotic behavior
	\begin{equation} \label{Yasymp}
	Y(z) = \left(I_{2r} + \mathcal{O}(z^{-1})\right)
	\begin{pmatrix} z^n I_r & 0_r \\ 0_r & z^{-n} I_r \end{pmatrix}
	\quad \text{ as } z \to \infty. \end{equation}
	\item To ensure a unique solution we also need to specify endpoint conditions at $\pm 1$. As in \cite{GIM11}, we have the following endpoint behavior (by $r\times r$ blocks):
\begin{equation}\label{eq:endpointsY}
    Y(z)
    =
    \begin{pmatrix}
        \mathcal{O}(1) & \mathcal{O}(h_{\alpha}(z))\\
        \mathcal{O}(1) & \mathcal{O}(h_{\alpha}(z))
    \end{pmatrix}, \quad \text{ as } z \to 1,
\end{equation}
since we assume that the matrix part of the weight is not identically $0$ at $z=\pm 1$. Here
\begin{equation} \label{hdef}
h_{\alpha}(z)
=
\begin{cases}
|z-1|^{\alpha},& \qquad -1<\alpha<0,\\
\log(|z-1|),& \qquad \alpha=0,\\
1,& \qquad \alpha>0,
\end{cases}
\end{equation} 
and a similar behavior holds as $z \to -1$, with 
$\beta$ instead of $\alpha$ in \eqref{eq:endpointsY}, and $z+1$ instead of $z-1$ in \eqref{hdef}.
\end{enumerate}


We have
\begin{equation} \label{PninY}
	P_n(z) = \begin{pmatrix} I_r & 0_r \end{pmatrix} Y(z)
	\begin{pmatrix} I_r \\ 0_r \end{pmatrix}, \end{equation}
that is, $P_n$ is the left upper $r \times r$ block of the solution $Y$ to the
RH problem. The right upper $r \times r$ block is given by
the Cauchy transform
\[ C(P_nW)(z) := \frac{1}{2\pi i} \int_{-1}^1 \frac{P_n(s) W(s)}{s-z} ds,
\qquad z \in \mathbb C \setminus [-1,1], \]
which has the jump $C(P_nW)_+ = C(P_nW)_- + P_nW$ on $(-1,1)$.
The lower half of the matrix $Y$ is constructed in a similar
way out of the degree $n-1$ MVOP, 
\begin{equation} \label{Ysolution} 
	Y(z) = \begin{pmatrix} P_n(z) & C(P_nW)(z) \\[5pt]
		 -2\pi i \Gamma_{n-1}^{-1} P_{n-1}(z) & -2\pi i \Gamma_{n-1}^{-1} C(P_{n-1}W) (z) \end{pmatrix}, \end{equation}
where $\Gamma_{n-1} = \ds \int_{-1}^1 P_{n-1}(x) W(x) P_{n-1}(x)^* dx$
as in \eqref{MVOPdef},
see \cite{CM12, GIM11} for details.

\subsection{First transformation}
We use the conformal map $\varphi(x)$ given by \eqref{phidef} in the first
transformation $Y \mapsto T$, which is given by 
\begin{equation} \label{Tdef}
	T =
	\begin{pmatrix} 2^n I_r & 0_r\\ 	0_r & 2^{-n}I_r
	\end{pmatrix}
	Y
	\begin{pmatrix} \varphi^{-n} I_r & 0_r\\ 	0_r & \varphi^n I_r
	\end{pmatrix}.
\end{equation}

Using the properties $\varphi(z) = 2z + \mathcal{O}(z^{-1})$
as $z \to \infty$ and $\varphi_+(x) \varphi_-(x) = 1$ for $x\in(-1,1)$,
we obtain that $T$ solves the following RH problem:
\begin{enumerate}
    \item $T:\mathbb C \setminus [-1,1] \to \mathbb C^{2r \times 2r}$ is analytic.
    \item On $(-1,1)$ we have the jump
    \begin{equation} \label{Tjump} T_+ = T_-
	\begin{pmatrix} \varphi_-^{-2n} I_r & W \\
		0_r & \varphi_+^{-2n} I_r \end{pmatrix}.
		\end{equation}
	\item As $z\to\infty$ the matrix $T$ is normalized at infinity:
	\begin{equation} \label{Tasymp}
	T(z) =  I_{2r} + \mathcal{O}(z^{-1}) \quad \text{ as } z \to \infty.
\end{equation}
\item As $z\to\pm 1$, the matrix $T$ has the same endpoint behavior as $Y$ has.   
\end{enumerate}



\subsection{Second transformation}
The jump matrix on $(-1,1)$ can be factorized as
\[ \begin{pmatrix} I_r & 0_r \\
	\varphi_-^{-2n} W^{-1}	& I_r \end{pmatrix}
\begin{pmatrix} 0_r & W \\
	-W^{-1} & 0_r \end{pmatrix}
\begin{pmatrix} I_r & 0_r \\
	\varphi_+^{-2n} W^{-1} 	 & I_r \end{pmatrix}. \]
This leads to the second transformation where we open a lens around $[-1,1]$ as in Figure~\ref{fig:lens}, and we define
\begin{equation} \label{Sdef1}
	S = T
	\times 
	\begin{cases}
	\begin{pmatrix} I_r & 0_r \\
		-\varphi^{-2n} W^{-1}& I_r \end{pmatrix}, & 
	\textrm{ in upper part of the lens},\\
		\begin{pmatrix} I_r & 0_r \\
		\varphi^{-2n} W^{-1}& I_r \end{pmatrix}, & 
	\textrm{ in lower part of the lens},
	\end{cases}
	 \end{equation}
and 
\begin{equation} \label{Sdef2}
		S  = T \qquad \text{ outside of the lens}.
\end{equation}
In the
definition of $S$ we use the analytic continuation of the weight
matrix into the complex plane with branch cuts along
$(-\infty,-1]$ and $[1,\infty)$.
That is
\[ W(z) =  (1-z)^{\alpha} (1+z)^{\beta}  H(z). \]
We also make sure that the lens is inside the region 
where $H$ is analytic and invertible. 
\begin{figure}[h!]
	\centering
	\includegraphics[scale=1]{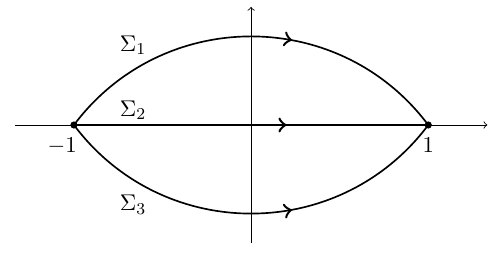}
\caption{Lens for the $T\mapsto S$ transformation. $\Sigma_S = \Sigma_1 \cup \Sigma_2 \cup \Sigma_3$ is the jump contour in the RH problem for $S$.}
\label{fig:lens}
\end{figure}

Then $S$ is defined and analytic in $\mathbb C \setminus \Sigma_S$
where $\Sigma_S=\Sigma_1\cup\Sigma_2\cup \Sigma_3$ consists of the interval $[-1,1]$
together with the upper and lower lip of the lens, see Figure \ref{fig:lens}. This matrix $S$ satisfies the following RH problem:
\begin{enumerate}
    \item $S(z)$ is analytic in $\mathbb{C}\setminus\Sigma_S$.
    \item For $z\in\Sigma_S$, we have the following jumps:
  \begin{align} \label{Sjump} S_+ = S_- \times
	\begin{cases} \begin{pmatrix} 0_r & W \\ -
			W^{-1} & 0_r \end{pmatrix} & \text{ on } (-1,1), \\
		\begin{pmatrix} I_r & 0_r \\
			\varphi^{-2n} W^{-1} & I_r \end{pmatrix}
		&  \text{ on } \Sigma_1 \text{ and } \Sigma_3. 
\end{cases} \end{align}  
    \item As $z\to\infty$, the matrix $S$ has the asymptotic behavior
\begin{equation} \label{Sasymp}
	S(z) = I_{2r} + \mathcal{O}(z^{-1}) \quad \text{ as } z \to \infty.
\end{equation}
    \item Outside the lens, $S$ has the same endpoint conditions as $T$ has. Inside the lens, the local behavior follows from the jump relations \eqref{Sjump}.
\end{enumerate}


\subsection{Global parametrix and proof of Proposition \ref{prop12}}
\label{sec:GlobalParametrix}
The jump matrix on the lips of the lens in \eqref{Sjump} tends to
the identity matrix as $n \to \infty$, because $\varphi(z)$ maps $\mathbb{C}\setminus[-1,1]$ onto the exterior of the unit circle,
and therefore $|\varphi(z)|>1$ for $z\in\mathbb{C}\setminus[-1,1]$. We ignore these jumps and look for a global parametrix $M$ satisfying the following RH problem:
\begin{enumerate}
    \item $M(z)$ is analytic in $\mathbb{C}\setminus[-1,1]$.
    \item On $(-1,1)$ we have the jump relation
    \begin{align} \label{Mjump} M_+ & = M_-
	\begin{pmatrix} 0_r & W \\ - W^{-1} & 0_r
	\end{pmatrix}.
\end{align}
    \item As $z\to\infty$, the matrix $M$ has the asymptotic behavior 
   \begin{equation} \label{Masymp}
	M(z) = I_{2r} + \mathcal{O}(z^{-1}). 
\end{equation} 
\end{enumerate}

The case $W = I_r$ can be readily solved. The solution is
\begin{align}  \label{M0def}
	M_0(z) & = \frac{1}{2} \begin{pmatrix}
		I_r & i I_r \\ i I_r & I_r \end{pmatrix}
		\begin{pmatrix} \gamma(z) I_r & 0_r \\ 0_r & \gamma(z)^{-1} 	I_r	\end{pmatrix} 
		\begin{pmatrix}
			I_r & -i I_r \\ -i I_r & I_r \end{pmatrix} \\
	& = 
	\frac{1}{2}
	\begin{pmatrix}
		 (\gamma(z) + \gamma(z)^{-1}) I_r
		&
		-i(\gamma(z) - \gamma(z)^{-1}) I_r \\[5pt]
		i (\gamma(z) - \gamma(z)^{-1}) I_r &
		(\gamma(z) + \gamma(z)^{-1}) I_r
	\end{pmatrix}, \quad \gamma(z) = \left(\frac{z-1}{z+1} \right)^{1/4}.  \nonumber
\end{align}
The solution for general $W$ requires the matrix Szeg\H{o} function $D$ from Proposition~\ref{prop12}.

In Lemma \ref{lemma31} and throughout the paper,  we use the notation $D(\infty)^{-\ast}$, 
where for an invertible matrix  $X$ we put  $X^{-\ast} := \left(X^{-1}\right)^{\ast} = 
\left(X^{\ast} \right)^{-1}$.

\begin{lemma} \label{lemma31} Let  $D$ be the matrix Szeg\H{o} function
	for $W$. Then $M$ defined by 
	\begin{equation} \label{Mdef}
		M(z) =
		\begin{pmatrix} D(\infty) & 0_r \\
			0_r & D(\infty)^{-\ast} \end{pmatrix}
		M_0(z) \begin{pmatrix} D(z)^{-1} & 0_r
			\\ 0_r & D(\overline{z})^{\ast} \end{pmatrix} \end{equation}
	satisfies the above RH problem for $M$.
\end{lemma}
\begin{proof}
	Note that $z \mapsto D(\overline{z})^\ast = \left(\overline{D(\overline{z}})\right)^T$ is analytic
	for $z \in \mathbb C \setminus [-1,1]$, as it
	involves two anti-holomorphic conjugations, and
	therefore $M$ is analytic. The asymptotic behavior \eqref{Masymp}
	is satisfied because $M_0(z) \to I_{2r}$
	and $D(z) \to D(\infty)$ as $z \to \infty$. 
	It remains to check the jump \eqref{Mjump}.
	
	For $x \in (-1,1)$, we have by \eqref{Mdef} and the
	jump property
	$M_{0,+} = M_{0,-} \begin{pmatrix} 0_r & I_r \\ -I_r & 0_r
	\end{pmatrix} $ of $M_0$ that
	\begin{align*} 
		M_-(x)^{-1} M_+(x)
		& = \begin{pmatrix} D_-(x) & 0_r \\
			0_r & D_+(x)^{-\ast}  \end{pmatrix} 
		\begin{pmatrix} 0_r & I_r \\ - I_r & 0_r \end{pmatrix}
		\begin{pmatrix} D_+(x)^{-1} & 0_r \\
			0_r & D_-(x)^\ast \end{pmatrix} \\
		& =  \begin{pmatrix} 0_r & D_-(x) D_-(x)^\ast \\
			- \left(D_+(x) D_+(x)^\ast \right)^{-1} & 0_r \end{pmatrix}		
	\end{align*}
	 and this is $\begin{pmatrix} 0_r & W(x) \\ -W(x)^{-1} & 0_r
	 	\end{pmatrix}$ because of the defining 
 	property \eqref{WDfactor} of $D$.  
\end{proof}	
	
We still have to show existence of the matrix Szeg\H{o} function, thereby proving Proposition~\ref{prop12}. 

We find $D(z)$ through a 
matrix valued factorization theorem
that we pose on the unit circle by means of the conformal map
$\varphi$, whose inverse is the rational function
\[ \varphi^{-1}(z) = \frac{z+z^{-1}}{2} . \]
Then $W \left( \frac{z+ z^{-1}}{2} \right)$, $|z|=1$,
is a matrix valued function
on the unit circle that is Hermitian positive definite except possibly
at $z=\pm 1$.
A classical result of Wiener and Masani \cite[Theorem 7.13]{WM57}, and  Helson and
Lowdenslager \cite[Theorem 9]{HL58} states that a factorization
\begin{equation} \label{WGfactor}
	W \left( \frac{z+z^{-1}}{2} \right) = G(z) G(z)^\ast,
	\qquad |z| = 1, \end{equation}
exists where $G$ is analytic and invertible in
the interior $|z| < 1$. In addition, $G$ is unique if we 
specify that $G(0)$ is Hermitian positive definite, but we do not insist on
the uniqueness here.
 The factorization \eqref{WGfactor} is valid
under the matrix  Szeg\H{o} condition
\begin{equation} \label{Szcondition}
	\frac{1}{2\pi} \int_{-\pi}^{\pi}
	\log \det W \left( \cos \theta \right) d\theta > - \infty, \end{equation}
which is certainly satisfied in the present situation. In the general setting
the identity \eqref{WGfactor} holds a.e.\ on the unit circle, 
but in our setting it is valid everywhere,
except possibly at $\pm 1$.
 

\begin{lemma} \label{lem:Ddef}
	Let $G$ solve the matrix factorization problem \ref{WGfactor}.
	Then the matrix Szeg\H{o} function is given by
    \begin{equation} \label{Ddef}
	D(z) = G\left(\frac{1}{\varphi(z)}\right), \qquad z \in \mathbb C \setminus [-1,1].
\end{equation}
\end{lemma}
\begin{proof}
Since $|\varphi(z)| > 1$ for every $z \in \mathbb C \setminus [-1,1]$,
we have that \eqref{Ddef} is well-defined and analytic for $|z|<1$. 
Also $D(z)$ is invertible for every $z \in \mathbb C \setminus  [-1,1]$, and  $D(\infty) = G(0)$ is also invertible, due to
the corresponding property of $G$.

To check the property \eqref{WDfactor} we let $x \in (-1,1)$
and write $z = \varphi_+(x)$. Then $|z| = 1$, $\overline{z} = 1/z = \varphi_-(x)$ and $x = \frac{z + z^{-1}}{2}$. From \eqref{Ddef}
we get $D_-(x) = G \left(\frac{1}{\varphi_-(x)} \right) = G(z)$
and thus by \eqref{WGfactor}
\[ D_-(x) D_-(x)^\ast = W\left(\frac{z+z^{-1}}{2}\right)	= W(x), 
\]
which is the first identity in \eqref{WDfactor}. The
second identity follows in similar fashion. We have
$ D_+(x) = G\left(\frac{1}{\varphi_+(x)}\right)	= G(\overline{z})$
and by \eqref{WGfactor}
\[ D_+(x) D_+(x)^\ast = W \left( \frac{\overline{z} + \overline{z}^{-1}}{2} \right)
= W \left(\frac{z^{-1}+z}{2} \right)= W(x). \]
\end{proof}

\begin{remark} \label{remark21}
	If $D$ is chosen such that $D(\infty)$ is Hermitian positive
	definite, then $D$ satis\-fies
	\[ D(\overline{z}) = D(z)^\ast, \qquad z \in \mathbb C \setminus
	[-1,1]. \]
	
	The identity $D_+(x) D_+(x)^\ast$ shows that
	$|D_+(x)| = \sqrt{W(x)}$ in the sense of the
	polar decomposition
	\begin{equation} \label{Dpolar}
		D_+(x) = \sqrt{W(x)} U^\ast(x), \qquad -1 < x < 1, \end{equation}
	where $U$ is a unitary matrix, and $\sqrt{W(x)}$ 
	denotes the positive definite square root of the  
	Hermitian positive definite matrix.
	Thus
	\[ \sqrt{W(x)} = D_+(x) U(x), \qquad -1 < x < 1. \]
\end{remark}

\subsection{Local parametrix around $z=1$}

\subsubsection{Statement}
We fix a disk $D(1,\delta)
= \{ z \in \mathbb C \mid |z-1| < \delta\}$,
around $1$, with radius $\delta > 0$ sufficiently small. The local parametrix $P$ should satisfy the following RH problem:
\begin{enumerate}
    \item $P(z)$ is analytic for $z\in D(1,\delta)\setminus \Sigma_S$.
    \item For $z\in D(1,\delta)\cap \Sigma_S$, the matrix $P(z)$ should have the same jumps as $S$ in this disk, see also Figure \ref{fig:P}:
   \begin{equation} \label{Pjump}
	P_+ = P_- \times
	\begin{cases} \begin{pmatrix} 0_r & W \\ - W^{-1} & 0_r \end{pmatrix} &
		\text{ on } (1-\delta,1), \\
		\begin{pmatrix}	I_r & 0_r \\ \varphi^{-2n} W^{-1} & I_r \end{pmatrix}
		& \begin{array}{l} \text{on the lips of the} \\ \text{lens inside the disk}.
		\end{array}
	\end{cases} \end{equation}	
	\item As $n\to\infty$, uniformly for $z\in\partial D(1,\delta)\setminus\Sigma_S$, we have the matching condition \begin{equation}\label{eq:matchingPM}
	    P(z)M^{-1}(z)=I_{2r}+\mathcal{O}(n^{-1}).
	\end{equation}
	\item As $z\to 1$, $P(z)$ has the same behavior as $S(z)$
	in the sense that $S(z) P^{-1}(z)$ remains bounded as $z \to 1$.
\end{enumerate}


\begin{figure}[t]
	\centering
	\includegraphics[scale=1.25]{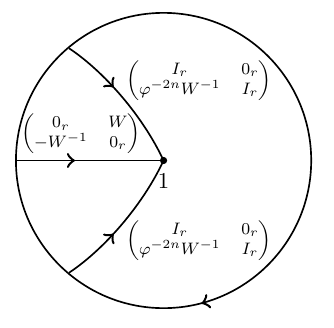}
	\caption{Contours and jumps in the RH problem for $P(z)$.}
	\label{fig:P}
\end{figure}


\subsubsection{Properties of $V$}
Recall that $V$ is defined in \eqref{Vdef} in 
terms of the modified eigenvalues \eqref{lambdajdef}.
We need the following properties.

\begin{lemma} \label{lemma34}
	$V$ is analytic in $D(1,\delta) \setminus [1-\delta,1]$
	and it satisfies \eqref{WVfactor},
	and
	\begin{align} \label{WVupper}
		W(z) & = V(z) \diag \left( e^{-\alpha_1 \pi i}, \ldots, e^{-\alpha_r \pi i}
		\right) V(\overline{z})^\ast && \text{in upper half plane within $D(1,\delta)$}, \\
		\label{WVlower}
		W(z) & = V(z) \diag \left( e^{\alpha_1 \pi i}, \ldots, e^{\alpha_r \pi i}
		\right) V(\overline{z})^\ast && \text{in lower half plane within $D(1,\delta)$}.
	\end{align}
\end{lemma}
\begin{proof}
	The analyticity is clear from \eqref{Vdef}. For $x \in (1-\delta,1)$, we have by \eqref{Vdef}
	\begin{equation}\label{eq:Vpm}
		V_{\pm}(x) = e^{\pm \alpha \pi i/2}
		(1-x)^{\alpha/2} (1+x)^{\beta/2} 
		Q(x) \widetilde{\Lambda}_{\pm}(x)^{1/2}.
	\end{equation}
and, 
	\begin{equation}\label{eq:Vpm2}
	V_{\pm}(x)^\ast = e^{\mp \alpha \pi i /2} 
	(1-x)^{\alpha/2} (1+x)^{\beta/2}
	\overline{\widetilde{\Lambda}_{\pm}(x)^{1/2}} Q(x)^\ast.
\end{equation}
We obtain \eqref{WVfactor} from \eqref{eq:Vpm} and \eqref{eq:Vpm2}
because of \eqref{Wdef}, \eqref{HQLambda} and the property 
\[ \widetilde{\Lambda}_{\pm}^{1/2}
	\overline{\widetilde{\Lambda}_{\pm}^{1/2}} =
		\Lambda. \]
The latter identity holds, since for each $j=1, \ldots, r$
and $x \in (1-\delta, 1]$, we have
\[ \widetilde{\lambda}_{j,\pm} (x)^{1/2} \overline{\widetilde{\lambda}_{j,\pm}(x)^{1/2}}
	= \left| \widetilde{\lambda}_j(x) \right| = \lambda_j(x). \]
see \eqref{lambdajdef}.

\medskip

To obtain \eqref{WVupper}, we first note that
an analytic function $f$ on $D(1,\delta) \setminus [1-\delta, 1]$
that is real valued on $(1,1+\delta)$
has the symmetry $\overline{f(\overline{z})} = f(z)$.
This applies to the functions $z \mapsto \widetilde{\lambda}_j(z)$,
and to $z \mapsto (z-1)^{\alpha/2}$, $z \mapsto
(z+1)^{\beta/2}$. Therefore we obtain from \eqref{Vdef}
\begin{align*} 
	V(\overline{z})^\ast
	& =  (z-1)^{\alpha/2} (z+1)^{\beta/2}
	\widetilde{\Lambda}(z)^{1/2} Q(\overline{z})^\ast.
	 \end{align*}
Then by this and \eqref{Vdef} the right-hand side of \eqref{WVupper} is
\begin{multline*}
	V(z) \diag \left( e^{-\alpha_1 \pi i}, \ldots,
		e^{-\alpha_r \pi i} \right) V(\overline{z})^\ast
		\\
	= (z-1)^{\alpha} (z+1)^{\beta}
		Q(z) \diag\left( e^{-\alpha_1 \pi i}
			\widetilde{\lambda}_1(z), \ldots,
			e^{-\alpha_r \pi i}
			\widetilde{\lambda}_r(z) \right) Q(\overline{z})^\ast.
\end{multline*}
Then using \eqref{alphajdef} and \eqref{lambdajdef},
wet get for $z \in D(1,\delta)$ with $\Im z > 0$,
\begin{multline} \label{WVupperproof}
	V(z) \diag \left( e^{-\alpha_1 \pi i}, \ldots,
	e^{-\alpha_r \pi i} \right) V(\overline{z})^\ast
	\\
	= e^{-\alpha \pi i} (z-1)^{\alpha} (1+z)^{\beta}
Q(z) \diag\left( (-1)^{n_1} \widetilde{\lambda}_1(z), \ldots,
(-1)^{n_r} \widetilde{\lambda}_r(z) \right) Q(\overline{z})^{\ast} \\
	= (1-z)^{\alpha} (1+z)^{\beta}
		Q(z) \Lambda(z) Q(\overline{z})^\ast.
\end{multline}
We finally note that $Q(z) \Lambda(z) Q(\overline{z})^\ast$
is analytic and agrees with $H(z)$ for $z \in (-1,1)$
because of \eqref{HQLambda}, and so it is equal to the
analytic continuation of $H$ into the complex plane. 
Thus \eqref{WVupperproof} is equal to the analytic
continuation of $W$ into the upper half-plane by \eqref{Wdef}.
This proves \eqref{WVupper}. 

The proof of \eqref{WVlower} is similar. The only difference
is that $e^{\alpha \pi i}(z-1)^{\alpha} = (1-z)^{\alpha}$ for $z$
in the lower half plane.
\end{proof}

\subsubsection{Reduction to constant jumps}

Having $V$ we seek $P$ in the form
\begin{equation} \label{Pdef} P(z) = E_n(z) P^{(1)}(z) \begin{pmatrix} \varphi(z)^{-n} V^{-1}(z) & 0_r \\
		0_r & \varphi(z)^n V(\overline{z})^\ast \end{pmatrix} \end{equation}
with an analytic prefactor $E_n(z)$, and an unknown $P^{(1)}(z)$ that are yet to
be determined. The properties of $V$ will guarantee that $P^{(1)}$ needs to have 
piecewise constant jumps. 

\begin{lemma}
	Suppose that $P$ is defined by \eqref{Pdef}
	with an analytic prefactor and invertible $E_n$.
	Then $P$ satisfies the jump conditions \eqref{Pjump}
	if and only if $P^{(1)}$ satisfies
	\begin{equation} \label{P1jump}
		P^{(1)}_+=P^{(1)}_- \times \begin{cases}
			\begin{pmatrix} 0_r & I_r \\
			- I_r & 0_r \end{pmatrix} & \text{on } (1-\delta,1), \\
	 	\begin{pmatrix} I_r  &  0_r \\
		 	e^{\vec{\alpha} \pi i} & I_r \end{pmatrix}
	 	& \begin{array}{l} \text{on the upper lip of the} \\
	 		\text{lens inside the disk,} \end{array} \\
 		\begin{pmatrix} I_r  &  0_r \\
 			e^{- \vec{\alpha} \pi i} & I_r \end{pmatrix}
 		& \begin{array}{l} \text{on the lower lip of the} \\
 			\text{lens inside the disk,} \end{array}
		 \end{cases}
	\end{equation}
	where
	$\vec{\alpha} = \diag \left( \alpha_1, \ldots, \alpha_r \right)$ and
$e^{\pm \vec{\alpha} \pi i} =
	\diag\left( e^{\pm \alpha_1 \pi i},\ldots, e^{\pm \alpha_r \pi i}
	\right).$
\end{lemma}
\begin{proof}
This is a straightforward calculation. The prefactor $E_n$
does not influence the jumps. Then it follows from \eqref{Pdef}
that
\begin{multline*} 
	\left(P_-^{(1)}(z) \right)^{-1} P_+^{(1)}(z) \\
	= \begin{pmatrix} \varphi(z)^{-n} V^{-1}(z) & 0_r \\
		0_r & \varphi(z)^{n} V(\overline{z})^\ast
	\end{pmatrix}_-  P_-^{-1}(z) P_+(z) 
	\begin{pmatrix} \varphi(z)^n V(z) & 0_r \\
		0_r & \varphi(z)^{-n} V(\overline{z})^{-\ast}
	\end{pmatrix}_+
	\end{multline*}
For $z = x \in (1-\delta, 1)$ this gives us, because of \eqref{Pjump} 
and the fact that $\varphi_+(x) \varphi_-(x) = 1$, 
\begin{multline*}
		\left(P_-^{(1)}(x) \right)^{-1} P_+^{(1)}(x) 
= \begin{pmatrix} \varphi_-(x)^{-n} V_-^{-1}(x) & 0_r \\
	0_r & \varphi_-(x)^{n} V_+(x)^\ast
\end{pmatrix} \\
\times \begin{pmatrix} 0_r & W(x) \\ -W^{-1} & 0_r
\end{pmatrix} 
\begin{pmatrix} \varphi_+(x)^n V_+(x) & 0_r \\
	0_r & \varphi_+(x)^{-n} V_-(x)^{-\ast} 
\end{pmatrix}\\
	= \begin{pmatrix} 0_r  & V_-^{-1}(x) W(x) V_-(x)^{-\ast} \\
	- V_+(x)^\ast W(x)^{-1} V_+(x) & 0_r \end{pmatrix}.
	\end{multline*}
Because of \eqref{WVfactor}, this reduces to
$\begin{pmatrix} 0_r & I_r \\ -I_r & 0_r \end{pmatrix}$
which gives us the required jump \eqref{P1jump} on $(1-\delta,1)$.
The jumps \eqref{P1jump} on the lips of the lens follow
in a similar way, where we now use the
identities \eqref{WVupper} and \eqref{WVlower}. 

Thus the jump conditions \eqref{Pjump} imply those in \eqref{P1jump}. It is easy to revert the arguments to show
that the jumps \eqref{P1jump} imply \eqref{Pjump}, which proves
the lemma. 
\end{proof}

\subsubsection{Bessel functions}

To construct $P^{(1)}(z)$, we use the standard size $2 \times 2$ local parametrix with
(modified) Bessel functions
from the paper \cite{KMVV04}, see formulas (6.23)--(6.25) therein,
that we reproduce here for ease of reference.
It involves the modified Bessel functions $I_{\alpha}$
and $K_{\alpha}$ of order $\alpha$, as well as the two
Hankel functions $H^{(1)}_{\alpha}$ and $H^{(2)}_{\alpha}$
of order $\alpha$.
\begin{multline} \label{Psialphadef}
	\Psi_{\alpha}(\zeta)  = \\
	\begin{cases} \begin{pmatrix}
		I_{\alpha}(2 \zeta^{1/2}) & \frac{i}{\pi} K_{\alpha}(2\zeta^{1/2}) \\
		2\pi i \zeta^{1/2} I_{\alpha}'(2\zeta^{1/2}) & 
		-2 \zeta^{1/2} K_{\alpha}'(2\zeta^{1/2}) 
		\end{pmatrix}, & |\arg \zeta| < \frac{2\pi}{3}, \\
		\begin{pmatrix}
		\frac{1}{2} H^{(1)}_{\alpha}(2 (-\zeta)^{1/2}) & \frac{1}{2} H^{(2)}_{\alpha}(2(-\zeta)^{1/2}) \\
		\pi \zeta^{1/2} \left(H^{(1)}_{\alpha}\right)'(2(-\zeta)^{1/2}) & 
		\pi \zeta^{1/2} \left(H^{(2)}_{\alpha}\right)'(2(-\zeta)^{1/2}) 
	\end{pmatrix} & \\
	\hfill{\times \begin{pmatrix} e^{\frac{1}{2} \alpha \pi i} & 0 \\
		0 & e^{-\frac{1}{2} \alpha \pi i} \end{pmatrix}}, 
	 & \frac{2\pi}{3} < \arg \zeta < \pi, \\
	\begin{pmatrix}
	\frac{1}{2} H^{(2)}_{\alpha}(2 (-\zeta)^{1/2}) & -\frac{1}{2} H^{(1)}_{\alpha}(2(-\zeta)^{1/2}) \\
	-\pi \zeta^{1/2} \left(H^{(2)}_{\alpha}\right)'(2(-\zeta)^{1/2}) & 
	\pi \zeta^{1/2} \left(H^{(1)}_{\alpha}\right)'(2(-\zeta)^{1/2}) 
\end{pmatrix} \\
	\hfill{\times \begin{pmatrix} e^{-\frac{1}{2} \alpha \pi i} & 0
			\\ 0 & e^{\frac{1}{2} \alpha \pi i} \end{pmatrix}}, & -\pi < \arg \zeta < -\frac{2\pi}{3}. 	
	\end{cases}
\end{multline}
We evaluate $\Psi_{\alpha}$ at $\zeta = n^2 f(z)$, where
\begin{equation} \label{fdef}
	f(z)= \frac{1}{4} \left(\log \varphi(z)\right)^2 \end{equation}
is a conformal map from $D(1,\delta)$ to a neighborhood of
$\zeta = f(1) = 0$. We may (and do) assume that the
lens is opened in such a way that $\arg f(z) = \pm 2\pi /3$
for $z$ on the lips of the lens within $D(1,\delta)$.
Then 
\begin{equation} \label{Psijump1}
	\Psi_{\alpha}(n^2 f(z))_+  = \Psi_{\alpha}(n^2 f(z))_-
	\begin{cases}
	\begin{pmatrix} 0 & 1 \\ -1 & 0 \end{pmatrix},&
	\quad \text{ on } (1-\delta,1), \\ 
	\begin{pmatrix} 1 & 0 \\ e^{\pm \alpha \pi i} & 0 \end{pmatrix},& 
	\quad \text{ on lips of the lens.} 
	\end{cases}
\end{equation}	

We use $\Psi_{\alpha}$ in block form with parameters $\alpha_1, \ldots, \alpha_r$. We also need the  permutation matrix $\Pi_r$ of
size $2r \times 2r$ with
\begin{equation} \label{Pirdef}
	 \left(\Pi_r\right)_{2j-1,j} = 1, \quad 
	 \left(\Pi_r\right)_{2j,j+r} = 1,
	 \qquad \text{for } j=1, \ldots, r,
\end{equation}
while $(\Pi_r)_{j,k} = 0$ otherwise. Thus for example,
\begin{equation*} 
	\Pi_2 = \begin{pmatrix} 1 & 0 & 0 & 0 \\
		0 & 0 & 1 & 0 \\
		0 & 1 & 0 & 0 \\
		0 & 0 & 0 & 1 \end{pmatrix}, \quad \Pi_3 = \begin{pmatrix}
		1 & 0 & 0 & 0 & 0 & 0 \\
		0 & 0 & 0 & 1 & 0 & 0 \\
		0 & 1 & 0 & 0 & 0 & 0 \\
		0 & 0 & 0 & 0 & 1 & 0 \\
		0 & 0 & 1 & 0 & 0 & 0 \\
		0 & 0 & 0 & 0 & 0 & 1 \end{pmatrix}. 
\end{equation*}

\begin{lemma}
	We define
	\begin{equation} \label{P1def}
		P^{(1)}(z)
		= 
		\Pi_r^{-1}
		\diag \left( \Psi_{\alpha_1}(n^2 f(z)), \ldots,
		\Psi_{\alpha_r}(n^2 f(z)) \right)  \Pi_r, \end{equation}
	where $\diag(\ldots)$ denotes a block diagonal matrix of size $2r \times 2r$
	with blocks of size $2 \times 2$.
	Then $P^{(1)}$  satisfies
	the jump properties \eqref{P1jump}.
\end{lemma}
\begin{proof}
The permutation matrix $\Pi_r$ has the
following property, which can be readily checked from
\eqref{Pirdef}. Given $2 \times 2$ matrices 
 $A_j = \begin{pmatrix} a_j & b_j \\ c_j & d_j
\end{pmatrix}$ for $j=1, \ldots, r$, one has 
\begin{equation} \label{AjPir}
		\Pi_r^{-1} \diag \left(A_1, A_2, \ldots, A_r\right)
		\Pi_r = \begin{pmatrix} \diag(\vec{a}) & \diag(\vec{b}) \\
			\diag(\vec{c)} & \diag(\vec{d}) \end{pmatrix}
\end{equation}
where $\vec{a} = (a_1, \ldots, a_r)$ and so on. 

The jumps \eqref{P1jump} follow by direct
calculation from the definition \eqref{P1def}, the
jumps \eqref{Psijump1} of $\Psi_{\alpha}$ and the property \eqref{AjPir}. 
\end{proof}

\begin{remark}
The reader may note that the jump conditions in \eqref{P1jump} remain the same
in case one or more of the $\alpha_j$'s is shifted by an even integer. Therefore
a construction of $P^{(1)}$  with such shifted parameters would
also satisfy the jump conditions. Then we could go on and construct $E_n$ as below
($E_n$ does not depend on the parameters $\alpha_j$) and define $P$ by \eqref{Pdef}.
However, we have to use the modified Bessel functions with exact orders
$\alpha_1, \ldots, \alpha_r$ in order to be able to match $P$ with $S$ in
the sense that $S(z) P(z)^{-1}$ should remain bounded near $z=1$.

The matching will be done in section \ref{sec36} below.
There we will use that $V(z)^{-1}$ and $V(\overline{z})^*$ have a certain behavior as $z\to 1$
with exponents $\pm \alpha_1/2, \ldots, \pm \alpha_r/2$,
see \eqref{Vinvat1}.
These exponents agree with the exponents coming from the Bessel parametrix at the origin,
see \eqref{Psiat0}, but only if we use the modified Bessel functions of orders $\alpha_1,\ldots, \alpha_r$.
\end{remark}

\subsubsection{Definition of $E_n$}

With $P^{(1)}$ given by \eqref{P1def},  we let $P$ be
as in \eqref{Pdef} with an analytic prefactor $E_n$
that is still to be determined. Then $P$ will have the correct
jumps from \eqref{Pjump} and we choose $E_n$
in such a way that it also satisfies the matching condition
\eqref{eq:matchingPM} on the circle $|z-1| = \delta$.

The leading
term in the asymptotic behavior of $\Psi_{\alpha}(n^2 f(z))$
as $n \to \infty$ (with fixed $z \in \partial D(1,\delta)$)
is given in formula (6.29) of \cite{KMVV04}. It does not
depend on $\alpha$. Thus for every $j=1, \ldots, r$ we have
\begin{multline*} 
	\Psi_{\alpha_j}(n^2 f(z)) =
\begin{pmatrix} (2\pi n)^{-1/2} f(z)^{-1/4} & 0 \\ 0 & (2\pi n)^{1/2} f(z)^{1/4} \end{pmatrix} \\
   \times 
		\left(\frac{1}{\sqrt{2}} \begin{pmatrix} 1 & i \\ i & 1 \end{pmatrix} + \mathcal{O}(n^{-1}) \right)
		\begin{pmatrix} 
		e^{2n f(z)^{1/2}} & 0 \\ 0 & e^{-2n f(z)^{-1/2}}
		\end{pmatrix}. 
	\end{multline*}
 By \eqref{P1def}, \eqref{AjPir}, and \eqref{fdef}  this leads to
(we use principal branches of the fractional powers)
\begin{multline} \label{P1asymp}
	P^{(1)}(z) 
	= \begin{pmatrix} (2 \pi n)^{-1/2} f(z)^{-1/4} I_r & 0_r \\
		0_r & (2\pi n)^{1/2} f(z)^{1/4} I_r \end{pmatrix} 
	\\ \times \left( \frac{1}{\sqrt{2}}\begin{pmatrix} I_r & i I_r \\
			i I_r & I_r \end{pmatrix}  +  \mathcal{O}(n^{-1})
		\right) 
		 \begin{pmatrix} \varphi(z)^n I_r & 0_r \\ 0_r
		& \varphi(z)^{-n} I_r \end{pmatrix}
\end{multline}
as $n \to \infty$. 

To obtain \eqref{eq:matchingPM}, 
we ignore the $\mathcal{O}(n^{-1})$ term and define $E_n$ in view of 
\eqref{Pdef} and \eqref{P1asymp} by
\begin{multline} \label{Endef}
	E_n(z) =  M(z)  \begin{pmatrix} V(z) & 0_r \\ 0_r
		& V(\overline{z})^{-\ast} \end{pmatrix}
	\frac{1}{\sqrt{2}} \begin{pmatrix} I_r & - i I_r \\ 
		-i I_r & I_r \end{pmatrix} \\
	\times
	\begin{pmatrix} (2 \pi n)^{1/2} f(z)^{1/4} I_r & 0_r \\
		0_r & (2\pi n)^{-1/2} f(z)^{-1/4} I_r \end{pmatrix}.
\end{multline}
Then the matching condition \eqref{eq:matchingPM}
can be readily verified from 
\eqref{Pdef} and \eqref{Endef}.

The definition \eqref{Endef} shows that  $E_n(z)$ is analytic
in $D(1,\delta) \setminus (1-\delta,1]$.  We need that it has
analytic extension to $D(1,\delta)$, and this is what we are going to prove
in the next subsection. Once we have that, we will have
completed the construction of the local parametrix at $1$.

\subsubsection{Analyticity of $E_n$ across $(1-\delta,1)$}

The analyticity of $E_n$ across $(1-\delta,1)$
follows from the following lemma.

\begin{lemma} \label{Enjump}
	We have
	$E_{n,+}(x)= E_{n,-}(x)$ for $x \in (1-\delta, 1)$.
\end{lemma}
\begin{proof} Let $x \in (1-\delta,1)$. 
	We first note that by \eqref{Mjump} and \eqref{WVfactor}
	\begin{multline*}
		\begin{pmatrix} V_-(x)^{-1} & 0_r \\ 0_r
			& V_+(x)^\ast \end{pmatrix}
		M_-^{-1}(x)  	  M_+(x)
		\begin{pmatrix} V_+(x) & 0_r \\ 0_r & V_-(x)^{-\ast} \end{pmatrix} \\
		= \begin{pmatrix} V_-(x)^{-1} & 0_r \\ 0_r
			& V_+(x)^{\ast} \end{pmatrix}
		\begin{pmatrix} 0_r & W(x) \\ - W(x)^{-1} & 0_r \end{pmatrix}
		\begin{pmatrix} V_+(x) & 0_r \\ 0_r & V_-(x)^{-\ast} \end{pmatrix}
		\\
		=
		\begin{pmatrix} 0_r & V_-(x)^{-1} W(x) V_-(x)^{-\ast} \\ -V_+(x)^\ast W(x)^{-1} V_+(x) & 0_r \end{pmatrix} 
		= \begin{pmatrix} 0_r & I_r \\ -I_r & 0_r \end{pmatrix}.
	\end{multline*}
	 This gives by \eqref{Endef}
	\begin{multline*}
		E_{n,-}^{-1}(x) E_{n,+}(x) =
		\begin{pmatrix} (2 \pi n)^{-1/2} f_-(x)^{-1/4} I_r & 0_r \\
			0_r & (2\pi n)^{1/2} f_-(x)^{1/4} I_r \end{pmatrix} \\
		\times \frac{1}{\sqrt{2}} \begin{pmatrix} I_r & i I_r \\ 
			i I_r & I_r \end{pmatrix}
		\begin{pmatrix} 0_r & I_r \\ -I_r & 0_r \end{pmatrix}
				\frac{1}{\sqrt{2}} \begin{pmatrix} I_r & -i I_r \\ 
			-i I_r & I_r \end{pmatrix} \\
		\times
			\begin{pmatrix} (2 \pi n)^{1/2} f_+(x)^{1/4} I_r & 0_r \\
			0_r & (2\pi n)^{-1/2} f_+(x)^{-1/4} I_r \end{pmatrix}.
	\end{multline*}
	The product of the three matrices in the middle line
	is equal to $\begin{pmatrix} -i I_r & 0_r \\ 0_r &  i I_r \end{pmatrix}$, and we get
	\[ E_{n,-}^{-1}(x)  E_{n,+}(x)
	=  \begin{pmatrix} -i f_-^{-1/4}(x) f_+^{1/4}(x) I_r & 0_r \\
		0_r & i f_-^{1/4}(x) f_+^{-1/4}(x) I_r \end{pmatrix}
	= I_{2r}.  \]
	The final identity holds since $f(x) < 0$ for $x\in(1-\delta,1)$ and due to the
	choice of principal branch of the fourth root, one has
	$f_+^{1/4}(x) = i f_-^{1/4}(x)$ for  $x \in (1-\delta, 1)$.
\end{proof}

From \eqref{Enjump} it follows that $E_n$ has analytic continuation across $(1-\delta, 1)$ by Morera's theorem. Thus $E_n$ is analytic in the punctured disk
$D(1,\delta) \setminus \{1\}$. 

\subsubsection{Removable singularity} \label{subsec357}
It remains to show
that the isolated singularity at $1$ is removable.

\begin{lemma} \label{Enremove}
	The isolated singularity of $E_n$ at $1$ is removable.
\end{lemma}
\begin{proof}
	From \eqref{Endef} it is clear that the $n$-dependence
	of $E_n$ is only in the last factor in the
	right hand side of \eqref{Endef}, and we have
	in view of \eqref{M0def}, \eqref{Mdef} and \eqref{Endef}
\begin{multline} \label{Ensum} 
	E_n(z) = \frac{1}{2} \begin{pmatrix} D(\infty) & 0_r \\
		0_r & D(\infty)^{-\ast} \end{pmatrix}
		\begin{pmatrix} I_r  & i I_r \\
			i I_r & I_r \end{pmatrix} \\
		\times \left[	
	 (\pi n)^{1/2} E^{(1)}(z) 
	 \begin{pmatrix} I_r & 0_r \end{pmatrix}
  + \frac{1}{2} (\pi n)^{-1/2} E^{(2)}(z) \begin{pmatrix} 0_r & I_r
  	\end{pmatrix}
 \right] \end{multline}
with $E^{(j)}(z)$, $j=1,2$, of size $2r \times r$
and independent of $n$, namely 
\begin{align}  E^{(1)}(z) & =  \nonumber
	\begin{pmatrix} \gamma(z) I_r & 0_r \\
		0_r & \gamma(z)^{-1} I_r \end{pmatrix}
	\begin{pmatrix} I_r & - i I_r \\
		-i I_r & I_r \end{pmatrix} 
	\begin{pmatrix} D(z)^{-1} V(z) & 0_r \\
		0_r & D(\overline{z})^\ast V(\overline{z})^{-\ast} \end{pmatrix} \\
	& \qquad \times \nonumber
	\begin{pmatrix} I_r & - i I_r \\
	-i I_r & I_r \end{pmatrix} 
\begin{pmatrix} f(z)^{1/4} I_r \\
0_r  \end{pmatrix} \\
& =	\label{E1z}
	\begin{pmatrix} \gamma(z) f(z)^{1/4} (D(z)^{-1} V(z) - D(\overline{z})^\ast V(\overline{z})^{-\ast})   \\
	\gamma(z)^{-1} f(z)^{1/4}	\left(-i D(z)^{-1} V(z) - i D(\overline{z})^\ast V(\overline{z})^{-\ast} \right)
		    \end{pmatrix}, \end{align}
and similarly 
\begin{align} 
	E^{(2)}(z) & =  \label{E2z}
	\begin{pmatrix} \gamma(z) f(z)^{-1/4} (-i D(z)^{-1} V(z) - i D(\overline{z})^\ast V(\overline{z})^{-\ast})  \\  
		\gamma(z)^{-1} f(z)^{-1/4} \left(- D(z)^{-1} V(z) + D(\overline{z})^\ast V(\overline{z})^{-\ast} \right)
		   \end{pmatrix}   
\end{align}
Both $E^{(1)}$ and $E^{(2)}$ are analytic in the punctured disk $D(1,\delta) \setminus \{1\}$.
Since both $\gamma(z)$ and $f(z)^{1/4}$ behave
like $\approx (z-1)^{1/4}$ as $z \to 1$ (see their definitions
in \eqref{M0def} and \eqref{fdef}), we conclude that
both
\begin{equation} \label{Omega1def} 
	\Omega_1(z) :=	D(z)^{-1} V(z) + D(\overline{z})^\ast V(\overline{z})^{-\ast}, \end{equation}
and
\begin{equation} \label{Omega2def}
	\Omega_2(z) := 
	(z-1)^{-1/2} \left( D(z)^{-1} V(z) - D(\overline{z})^\ast V(\overline{z})^{-\ast} \right) \end{equation}
are analytic in $D(1,\delta) \setminus \{1\}$.
It suffices to prove that both $\Omega_1$
and $\Omega_2$ have  removable  singularities at $z=1$.

\medskip
We first show that there cannot be an essential
singularity. From the definition \eqref{Vdef} it is
clear that
\begin{align} \label{Vat1}
	V(z) = \mathcal{O}\left((z-1)^{\alpha_{\min}/2}\right), \quad
	V^{-1}(z) = \mathcal{O}\left((z-1)^{-\alpha_{\max}/2}\right). \end{align}
where $\alpha_{\min} = \min(\alpha_1, \ldots, \alpha_r)$
and $\alpha_{\max} = \max(\alpha_1, \ldots, \alpha_r)$. 
The same estimates 
\begin{align} \label{Dat1}
	D(z) = \mathcal{O}\left((z-1)^{\alpha_{\min}/2}\right), \quad
	D^{-1}(z) = \mathcal{O}\left((z-1)^{-\alpha_{\max}/2}\right), \end{align}
hold for $D(z)$ and $D^{-1}(z)$.
To see this we argue that the Jacobi
prefactor $(1-x)^{\alpha} (1+x)^{\beta}$ has the scalar Szeg\H{o} function
$\frac{(z-1)^{\alpha/2} (z+1)^{\beta/2}}{\varphi(z)^{(\alpha + \beta)/2}} $	and
\[ D(z) = \frac{(z-1)^{\alpha/2} (z+1)^{\beta/2}}{\varphi(z)^{(\alpha + \beta)/2}}
	D_H(z), \]
where $D_H(z)$ is the matrix Szeg\H{o} function for $H$.
Since $H$ is bounded and analytic at $z=1$, also 
 $D_H$ is bounded at $z=1$. It thus follows that $D(z) = \mathcal{O}\left((z-1)^{\alpha/2}\right)$ as $z \to 1$, which is
the first statement of \eqref{Dat1}, since $\alpha_{\min} = \alpha$.
The second statement follows in a similar fashion
since $z \mapsto D(\overline{z})^{-\ast}$ is the matrix Szeg\H{o} function for
$W^{-1}$ and the eigenvalues of $W^{-1}$ 
have exponents $-\alpha_1,  \ldots, - \alpha_r$ at $z=1$.
From \eqref{Vat1} and \eqref{Dat1} we see that both
\eqref{Omega1def} and \eqref{Omega2def} have the behavior
$\mathcal{O}((z-1)^{-p})$ as $z \to 1$ for some $p \geq 0$,
which implies that the isolated singularity at $z=1$
cannot be an essential singularity. It can be at most
a pole of order $\leq p$.

\medskip
To exclude the possibility of a pole we consider $x \in (-1,1)$.
From \eqref{Wdef}, \eqref{Dpolar} and \eqref{HQLambda} 
 we have
\begin{align*}
		D_+^{-1}(x) & = (1-x)^{-\alpha/2} (1+x)^{-\beta/2}
		U(x) \left(Q(x) \Lambda(x) Q(x)^\ast \right)^{-1/2}  \\
		& = (1-x)^{-\alpha/2} (1+x)^{-\beta/2} U(x) Q(x) \Lambda(x)^{-1/2} Q(x)^\ast,
	\end{align*}
where $U(x)$ and $Q(x)$ are unitary.
Then by \eqref{Vdef}, \eqref{lambdajdef} and \eqref{Lambdadef}
\begin{align} \nonumber D_+^{-1}(x) V_+(x) & = e^{\alpha\pi i /2} 
		U(x) Q(x) \Lambda(x)^{-1/2} \widetilde{\Lambda}_+(x)^{1/2} \\
		& = U(x) Q(x) \diag \left(e^{\alpha_1\pi i/2},
			\ldots, e^{\alpha_r\pi i/2} \right),
			\quad -1 < x < 1, \label{DVprod}
\end{align}
where we also used \eqref{alphajdef}.

The three factors on the right-hand side of \eqref{DVprod} are unitary matrices that remain bounded as $x \to 1-$. 
Thus \eqref{DVprod} remains bounded as $x \to 1-$.
The same reasoning applies to $D_-^{-1}(x) V_-(x)$
and to their Hermitian transposes.
Because of \eqref{Omega1def} we then have that $\Omega_1(x)$ 
remains bounded
as $x \to 1-$, while by \eqref{Omega2def} we have that $\Omega_2(x) =
\mathcal{O}\left((x-1)^{-1/2}\right)$ as $x \to 1-$, and both behaviors
exclude the possibility of a pole at $1$.
Thus $\Omega_1$ and $\Omega_2$ have removable singularities
at $1$, and this completes the proof.
\end{proof}

\subsubsection{Proof of Lemma \ref{lemma17}}  \label{subsec358}
From the proof of Lemma \ref{Enremove} we also
obtain the existence of the limit defining $U_1$
as claimed in Lemma \ref{lemma17}.

\begin{proof}[Proof of Lemma \ref{lemma17}.]
In the proof of Lemma \ref{Enremove} we established that
$\Omega_1$ and $\Omega_2$ are analytic in a neighborhood of $1$, where
$\Omega_1$ and $\Omega_2$ are defined  by \eqref{Omega1def}
and \eqref{Omega2def}. From these definitions we see
that
\[ D(z)^{-1} V(z) = \frac{1}{2} \left(
	\Omega_1(z) + (z-1)^{1/2} \Omega_2(z)\right), \]
and therefore the limit defining $U_1$ in \eqref{U1def} exists
and $U_1 = \frac{1}{2} \Omega_1(1)$.
For $z = x \in (1-\delta,1)$ we have by \eqref{DVprod}
that $D_+^{-1}(x) V_+(x)$ is unitary, and the unitarity
is preserved in the limit $x \to 1-$. Thus $U_1$ is unitary.

The statements for $U_{-1}$ in Lemma \ref{lemma17} follow
similarly. 
\end{proof}

\subsection{$S(z) P(z)^{-1}$ remains bounded near $z=1$} \label{sec36}
We finally check the last item in the RH problem for $P$.

\begin{lemma}
	$S(z) P(z)^{-1}$ remains bounded
	near $z=1$.
\end{lemma}
\begin{proof}
We first note that by \cite[Remark 7.1]{KMVV04}
we have $\det \Psi_{\alpha}(\zeta) = 1$ for every $\zeta$
where it is defined. Then by \eqref{P1def} also
$P^{(1)}(z) = 1$ for every $z \in D(1,\delta) \setminus \Sigma_S$.
Since $\det V(z) = \overline{\det V(\overline{z})}$  we then also get that $E_n$ defined by \eqref{Endef}
has determinant $1$ (as also $M$ has determinant $1$). 
Hence by \eqref{Pdef} also 
\begin{equation} \label{detP} 
	\det P(z) = 1, \text{ for }z \in D(1,\delta) \setminus \Sigma_S,
	\end{equation}
and in particular the inverse $P(z)^{-1}$ exists.

Next, because $S$ and $P$ have the same jumps inside $D(1,\delta)$,
the product $S(z) P(z)^{-1}$ has analytic continuation to $D(1,\delta) \setminus \{1\}$ with an isolated singularity at $z=1$.
We have to show that the isolated singularity is removable.

By construction  both $S$ and $P$ can have at most
power like singularities, say $S(z) = \mathcal{O}\left((z-1)^{-p}\right)$
and $P(z) = \mathcal{O}\left((z-1)^{-p}\right)$ as $z \to 0$, for some $p \geq 0$. Then also $ P^{-1}(z) = \mathcal{O}\left((z-1)^{-p}\right)$
since $\det P = 1$, and $SP^{-1}(z) = \mathcal{O}\left((z-1)^{-2p} \right)$
which implies that $S P^{-1}$ does not have an essential
singularity at $1$. 

The behavior of $\Psi_{\alpha}$ near $0$ is given by 
formulas (6.19)--(6.21) in \cite{KMVV04}. For $\alpha \neq 0$,
we have
\begin{equation} \label{Psiat0}
	\Psi_{\alpha}(\zeta) = \begin{pmatrix}
	\mathcal{O}( |\zeta|^{\alpha/2}) & \mathcal{O}(|\zeta|^{-|\alpha|/2}) \\
	\mathcal{O}( |\zeta|^{\alpha/2}) & \mathcal{O}(|\zeta|^{-|\alpha|/2}) 
	\end{pmatrix} \end{equation}
as $\zeta \to 0$ with $|\arg \zeta| < 2 \pi /3$.
Then from \eqref{P1def} and \eqref{AjPir} we get that,
\begin{multline} P^{(1)}(z) = \\ \begin{pmatrix}
	\diag\left( \mathcal{O}(|z-1|^{\alpha_1/2}), \ldots, \mathcal{O}(|z-1|^{\alpha_r/2})    \right) &  \diag\left( \mathcal{O}(|z-1|^{-|\alpha_1|/2}), \ldots,  \mathcal{O}(|z-1|^{-|\alpha_r|/2})   \right) \\
	\diag\left( \mathcal{O}(|z-1|^{\alpha_1/2}), \ldots, 
	 \mathcal{O}(|z-1|^{\alpha_r/2})   \right) &  \diag\left( \mathcal{O}(|z-1|^{-|\alpha_1|/2}), \ldots,  \mathcal{O}(|z-1|^{-|\alpha_r|/2}) 
  \right) 
\end{pmatrix} \label{P1at1}
	\end{multline}
as $z \to 1$ outside the lens.  From \eqref{Vdef} we have
\begin{equation} \label{Vinvat1}
\begin{aligned} 
		V(z)^{-1} & =
		\diag \left( \mathcal{O}(|z-1|^{-\alpha_1/2}), \ldots,  \mathcal{O}(|z-1|^{-\alpha_r/2}) \right) Q(z)^{-1}, \\
		V(\overline{z})^{\ast} & =
		\diag \left( \mathcal{O}(|z-1|^{\alpha_1/2}), \ldots,  \mathcal{O}(|z-1|^{\alpha_r/2}) \right) Q(\overline{z})^{\ast}.
		\end{aligned}
\end{equation}
We use \eqref{P1at1} and \eqref{Vinvat1} in \eqref{Pdef}
and we note that $E_n(z)$ and $\varphi(z)^{\pm}$
 remain bounded as $z \to 1$. 
Then the definition \eqref{Pdef} of $P(z)$ tells us that
\[ P(z) = \begin{cases}
		\begin{pmatrix} \mathcal{O}(1) & \mathcal{O}(|z-1|^{\alpha}) \\
		\mathcal{O}(1) & \mathcal{O}(|z-1|^{\alpha}) \end{pmatrix},  
		& \text{ if }  -1 < \alpha < 0, \\
		\begin{pmatrix} \mathcal{O}(1) & \mathcal{O}(1) \\
			\mathcal{O}(1) & \mathcal{O}(1) \end{pmatrix},  & \text{ if } \alpha > 0, 
\end{cases}
 \]
 as $z \to 1$ from outside the lens, where we recall
 that $\alpha = \min(\alpha_1,\ldots, \alpha_r)$, and
 all $\alpha_j - \alpha$ are non-negative integers.
For $\alpha = 0$, logarithmic terms appear in \eqref{Psiat0},
and then the above reasoning leads to
\[ P(z) = \begin{pmatrix} \mathcal{O}(\log |z-1|) & \mathcal{O}(\log |z-1|) \\
	\mathcal{O}(\log|z-1|) & \mathcal{O}(\log|z-1|) \end{pmatrix},
	\quad \text{ if } \alpha = 0. \]
as $z \to 1$ from outside the lens.
Because of \eqref{detP} we obtain from the above that
\begin{equation} \label{Pinvat1}
	P^{-1}(z) =
	\begin{cases}
		\begin{pmatrix} \mathcal{O}(|z-1|^{\alpha}) &
			\mathcal{O}(|z-1|^{\alpha}) \\
		\mathcal{O}(1) & \mathcal{O}(1) \end{pmatrix},
	& \text{ if } -1 < \alpha < 0, \\
	\begin{pmatrix} \mathcal{O}(\log |z-1|) & \mathcal{O}(\log |z-1|) \\
		\mathcal{O}(\log|z-1|) & \mathcal{O}(\log|z-1|) \end{pmatrix}, 
	& \text{ if } \alpha = 0, \\
	\begin{pmatrix} \mathcal{O}(1) &
			\mathcal{O}(1) \\
			\mathcal{O}(1) & \mathcal{O}(1) \end{pmatrix},
	& \text{ if } \alpha > 0, 
	\end{cases}
\end{equation} 
as $z \to 1$ from outside the lens.

From item 4.\ in the RH problem for $S$ we get that
$S(z)$ behaves in the same way as $Y(z)$ when $z \to 1$
from outside the lens. That is,
\begin{equation} \label{Sat1} S(z) = \begin{pmatrix} \mathcal{O}(1) & \mathcal{O}(h_{\alpha}(z)) \\
	\mathcal{O}(1) & \mathcal{O}(h_{\alpha}(z)) \end{pmatrix},
\end{equation}   with $h_{\alpha}$ as in \eqref{hdef}.
Then by \eqref{Pinvat1} and \eqref{Sat1}
it follows that 
\[ S(z) P^{-1}(z) = 	\begin{pmatrix} \mathcal{O}(|z-1|^{\alpha}) &
	\mathcal{O}(|z-1|^{\alpha}) \\
	\mathcal{O}(|z-1|^{\alpha}) & 
	\mathcal{O}(|z-1|^{\alpha}) \end{pmatrix}, \quad
	 \text{ if } -1 < \alpha < 0, \]
as $z \to 1$ from outside the lens. This behavior
shows that $S(z) P^{-1}(z)$ cannot have a pole 
at $z=1$, since  $\alpha > -1$. If $\alpha > 0$
then \eqref{Pinvat1} and \eqref{Sat1} give us that
$SP^{-1} = \mathcal{O}(1)$ and again there is no pole.
If $\alpha = 0$ then $SP^{-1}$ has a potential
logarithmic behavior, but again it is not enough
for a pole.

We already excluded the possibility of an essential singulariy
and thus $SP^{-1}$ has a removable singularity at $z=1$.
The lemma follows.
\end{proof}

\subsection{Local parametrix around $z=-1$}
The local parametrix $\widetilde{P}$ around $-1$ is constructd
in a similar way. It satisfies the following RH problem:
\begin{enumerate}
    \item $\widetilde{P}(z)$ is analytic for $z\in D(-1,\delta)\setminus \Sigma_S$.
    \item For $z\in D(-1,\delta)\cap \Sigma_S$, the matrix $\widetilde{P}(z)$ should have the same jumps as $S(z)$ in this disk, see also Figure \ref{fig:Pt}:
   \begin{equation} \label{Ptjump}
	\widetilde{P}_+ = \widetilde{P}_- \times
	\begin{cases} \begin{pmatrix} 0_r & W \\ - W^{-1} & 0_r \end{pmatrix} &
		\text{ on } (1-\delta,1), \\
		\begin{pmatrix}	I_r & 0_r \\ \varphi^{-2n} W^{-1} & I_r \end{pmatrix}
		& \begin{array}{l} \text{on the lips of the} \\ \text{lens inside the disk}.
		\end{array}
	\end{cases} \end{equation}	
	\item As $n\to\infty$, uniformly for $z\in\partial D(-1,\delta)\setminus\Sigma_S$, we have the matching condition \begin{equation}\label{eq:matchingPtM}
	    \widetilde{P}(z)M^{-1}(z)=I_{2r}+\mathcal{O}(n^{-1}).
	\end{equation}
	\item $S(z) \left(\widetilde{P}(z)\right)^{-1}$ remains bounded as $z \to -1$.
\end{enumerate}


\begin{figure}[t]
	\centering
	\includegraphics[scale=1.2]{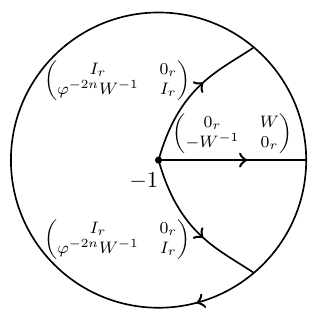}
	\caption{Contours and jumps in the RH problem for $\widetilde{P}(z)$.}
	\label{fig:Pt}
\end{figure}

The local parametrix takes the form 
\begin{equation} \label{Ptdef} 
\widetilde{P}(z) = \widetilde{E}_n(z) \widetilde{P}^{(1)}(z) \begin{pmatrix} \widetilde{\varphi}(z)^{-n} \widetilde{V}^{-1}(z) & 0_r \\
		0_r & \widetilde{\varphi}(z)^n \widetilde{V}(\overline{z})^{\ast} \end{pmatrix} 
\end{equation}
which is similar to \eqref{Pdef}. All quantities with a tilde are
slight modifications of their non-tilded counterparts. 
We use $\widetilde{\varphi}(z) = \varphi(-z)$, and we note that changing $\varphi(z)\rightarrow\widetilde{\varphi}(z)$ does not alter the jumps for $\widetilde{P}(z)$ in \eqref{Ptjump}.
Let $m_j$ be the order of vanishing of $\lambda_j$ at $-1$, and put
\[ \beta_j = \beta + m_j. \]
Then, with appropriate branches of the square roots,
\[ \widetilde{V}(z) =
	(1-z)^{\alpha/2} (-1-z)^{\beta/2}  Q(z) 
	\diag \left((-1)^{m_1} \lambda_1(z), \ldots, (-1)^{m_r} \lambda_r(z) \right)^{1/2} \]
for $z \in D(-1,\delta) \setminus [-1,-1+\delta]$.
$\widetilde{P}^{(1)}$ is built out of the $2\times 2$ Bessel parametrix \eqref{Psialphadef}, but now with parameters $\beta_1, \ldots, \beta_r$, 
namely, similar to \eqref{P1def},
\[
\widetilde{P}^{(1)}(z)
=
\Pi_r^{-1}
\diag
\left(\sigma_3\Psi_{\beta_1}(n^2\tilde{f}(z))\sigma_3,\ldots,\sigma_3\Psi_{\beta_r}(n^2\tilde{f}(z))\sigma_3\right)
\Pi_r,
\]
with $\sigma_3 = \begin{pmatrix} 1 & 0 \\ 0 & -1 \end{pmatrix}$
and $\widetilde{f}(z) = \frac{1}{4} \left( \log \widetilde{\varphi}(z)\right)^2$.
The analytic prefactor takes the form
\begin{multline} \label{tildeEndef}
	\widetilde{E}_n(z) =  
	M(z)  
	\begin{pmatrix} \widetilde{V}(z) & 0_r \\ 0_r
		& \widetilde{V}(\overline{z})^{-\ast} \end{pmatrix}
	\frac{1}{\sqrt{2}} \begin{pmatrix} I_r &  i I_r \\ 
		i I_r & I_r \end{pmatrix} \\
	\times
	\begin{pmatrix} (2 \pi n)^{1/2} \tilde{f}(z)^{1/4} I_r & 0_r \\
		0_r & (2\pi n)^{-1/2} \tilde{f}(z)^{-1/4} I_r \end{pmatrix}
\end{multline}
which is analogous to \eqref{Endef}.
The items in the RH problem for $\widetilde{P}$
then follow in the same way as we proved them for $P$. We do not give
any more details.

\subsection{Final transformation}

The final transformation $S \mapsto R$ is
\begin{equation} \label{Rdef} 
	R(z) = \begin{cases} S(z)M^{-1}(z),  & z \in \mathbb C \setminus \left(\overline{D(1,\delta)}  \cup \overline{D(-1,\delta)} \cup \Sigma_S \right),  \\
	S(z) P^{-1}(z), & z\in D(1,\delta) \setminus \Sigma_S,\\
	S(z) \widetilde{P}^{-1}(z), & z\in D(-1,\delta) \setminus \Sigma_S.
	\end{cases} \end{equation}
Then $R$ is defined and analytic in $\mathbb C \setminus \left( \Sigma_S \cup \partial D(1,\delta) \cup \partial D(-1,\delta) \right)$
with analytic continuation across $(-1,1)$ and on the parts of $\Sigma_S$ 
inside the disks. This follows immediately from the
fact that the jumps of $M$ and $S$ agree on $(-1,1)$,
the jumps of $P^{(1)}$ and $S$ agree on $\Sigma_S \cap D(1,\delta)$, and the jumps of $\widetilde{P}^{(1)}$ and $S$ agree
on $\Sigma_S \cap D(-1,\delta)$. 
The isolated singularities at $\pm 1$ are removable,
since $R$ remains bounded near the endpoints, as follows
from item 4.\ in the RH problems for $P$ and $\widetilde{P}$.
Therefore $R(z)$ satisfies the following RH problem
on the oriented contour $\Sigma_R$ shown in Figure \ref{fig:R}:
\begin{enumerate}
    \item $R(z)$  is analytic in $\mathbb{C}\setminus \Sigma_R$. 
    \item For $z\in\Sigma_R$, the matrix has the following jumps:
    \[
R_+(z) = R_-(z)
\begin{cases}
M(z)		
\begin{pmatrix} 
I_r & 0_r \\
\varphi^{-2n}(z) W^{-1}(z) & I_r \end{pmatrix}
M(z)^{-1}, &  \begin{array}{l} z \text{ on the lips of the lens} \\
	\text{outside of the disks}, \end{array} \\
P(z)M(z)^{-1}, & z\in \partial D(1,\delta), \\
\widetilde{P}(z) M(z)^{-1}, & z \in \partial D(-1,\delta).
\end{cases}
\]
\item As $z\to\infty$, we have the asymptotic behavior $R(z)=I_{2r}+\mathcal{O}(z^{-1})$.
\end{enumerate}

Since $M(z)$ and $W(z)$ are independent of $n$, and $|\varphi(z)|>1$ on the lips of lens outside the disks, 
we can verify that 
$R_+ = R_-(I_{2r} + \mathcal{O}(n^{-1}))$ on the
two circles and $R_+ = R_-(I_{2r} + \mathcal{O}(e^{-cn}))$ (with $c > 0$)
on the lips of the lens outside the disks. 
The conclusion of the steepest descent analysis then is that 
\begin{equation} \label{Rasymp} 
	R(z) = I_{2r} + \mathcal{O}\left(\frac{1}{n (1+|z|)}\right)  \quad \text{ as } n \to \infty, \end{equation}
uniformly for $z\in\mathbb{C}\setminus \Sigma_R$.

\begin{figure}[t]
	\centering
	\includegraphics[scale=1]{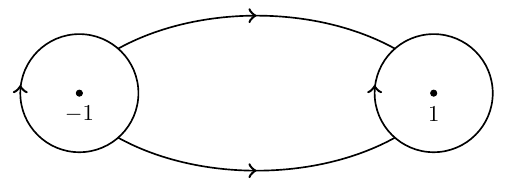}
	\caption{System of contours $\Sigma_R$ in the RH problem for $R(z)$.}
	\label{fig:R}
\end{figure}

\subsection{Asymptotic expansion of $R$}
The large $n$ behavior \eqref{Rasymp} will suffice for
the proof of main term in Theorems \ref{theorem14}
and for the proof of Theorems \ref{theorem15}--\ref{theorem17}
that deal with the asymptotic behavior of the MVOP.
For the  large $n$ behavior of the recurrence coefficients
as stated in Theorem \ref{theorem18} we need more
information on $R$. In fact it will be true that $R$
has a full asymptotic expansion
\begin{equation} \label{Rexpansion}
	R(z) \sim I_{2r} + \sum_{k=1}^{\infty} \frac{R_k(z)}{n^k} 
	\end{equation}
as $n \to \infty$, that is uniform for $z \in \mathbb C \setminus (\partial D(1,\delta) \cup \partial D(-1,\delta))$.
Furthermore the expansion has a double asymptotic property
\[ 
\left\| R(z) - I_{2r} - \sum_{k=1}^{\ell} \frac{R_{k}(z)}{n^k}\right\|
	\leq \frac{C_{\ell}}{|z| n^{\ell+1}}, \quad C_{\ell} > 0,
		\]
for $\ell \geq 1$ and $|z| > 2$. This is analogous
to \cite[Theorem 7.10]{DKMVZ99} or \cite[Lemma 8.3]{KMVV04}, and the proof is similar.
The matrix valued functions $R_k(z)$ are meromorphic with
poles in $\pm 1$ only and $R_k(z) = \mathcal{O}(z^{-1})$ as $z \to \infty$.

The asymptotic expansion of $R$ follows from an expansion
of the jump matrices of $R$ on the two circles $\partial D(\pm 1,\delta)$. We write 
\[ \Delta(z)  = \begin{cases} P(z) M(z)^{-1} - I_{2r}
		& \text{ for } z \in \partial D(1,\delta) \\
		\widetilde{P}(z) M(z)^{-1} - I_{2r}
		& \text{ for } z \in \partial D(-1,\delta).
		\end{cases} \]
Then $\Delta$ also depends on $n$ (which is supressed in the notation)
and	 $\Delta$ has an asymptotic expansion
\begin{equation} \label{Deltaexpansion}
\Delta(z) = \sum_{k=1}^{\infty} \frac{\Delta_k(z)}{n^k}, 
\end{equation}
as $n \to \infty$, with
\[ \Delta_k(z) = \frac{1}{2^k g(z)^k}
	M(z) \begin{pmatrix} V(z) & 0_r \\ 0_r & V(\overline{z})^{-\ast} \end{pmatrix} \Pi_r^{-1}
	\Psi_k \Pi_r \begin{pmatrix} V^{-1}(z) & 0_r \\ 0_r & V(\overline{z})^{\ast} \end{pmatrix} M^{-1}(z) \]
 with
\[ g(z) = \begin{cases} \log \varphi(z), & \text{ on } 
	\partial D(1,\delta), \\
	\log \tilde{\varphi}(z), & \text{ on } 
	\partial D(-1,\delta), \end{cases} \]
and $\Psi_k$ is a piecewise constant matrix
\[ \Psi_k = \begin{cases} 
	\diag\left( (\alpha_j,k-1) \begin{pmatrix} 
		\frac{(-1)^k}{k} \left(\alpha_j^2 + \frac{k}{2}-\frac{1}{4}\right) & - (k-\frac{1}{2})i \\
		(-1)^k (k-\frac{1}{2}) i & \frac{1}{k} \left(\alpha_j^2 +\frac{k}{2} - \frac{1}{4} \right) \end{pmatrix}\right), 
		& \text{ on } \partial D(1,\delta), \\
\diag\left( (\beta_j,k-1) \begin{pmatrix} 
	\frac{(-1)^k}{k} \left(\beta_j^2 + \frac{k}{2}-\frac{1}{4}\right) &  (k-\frac{1}{2})i \\
	(-1)^{k+1} (k-\frac{1}{2}) i & \frac{1}{k} \left(\beta_j^2 +\frac{k}{2} - \frac{1}{4} \right) \end{pmatrix} \right),
	& \text{ on } \partial D(-1,\delta).
	\end{cases} 
	\]
Thus $\Psi_k$ is a block diagonal matrix with $2 \times 2$ blocks
as $j$ is varying from $1$ to $r$. The numbers $(\alpha_j,k-1)$
and $(\beta_j,k-1)$ come from asymptotic expansions of Bessel
functions. In general we have $(\nu,0) = 1$ and
\begin{equation}\label{nukdef}
		(\nu,k)
		=
		\frac{(4\nu^2-1)(4\nu^2-9)\cdots (4\nu^2-(2k-1)^2)}{4^k k!}, \qquad k \geq 1.
	\end{equation}
	
	The analogue of Lemma 8.2 in \cite{KMVV04} holds. That is,
	for some $\delta_0 > \delta$, we have that
	$\Delta_k$ has an analytic continuation to 
	$(D(1,\delta_0) \setminus \{1\}) \cup D(-1,\delta_0) \setminus \{-1\})$ 
	with poles of order $\leq \left[ \frac{k+1}{2} \right]$
	at $z=1$ and $z=-1$. 
	
	The matrix valued functions $R_k(z)$, for $k\geq 1$, are obtained
	from additive RH problems arising from the relation $R_+(z)=R_-(z)(I_{2r}+\Delta(z))$ for $z\in\partial D(1,\delta)\cup\partial D(-1,\delta)$ together with \eqref{Rexpansion} and \eqref{Deltaexpansion}. The first one is 
	\begin{equation}\label{eq:RHPR1}
		R_{1+}(z)=R_{1-}(z)+\Delta_1(z), \qquad z\in\partial D(1,\delta)\cup D(-1,\delta),
	\end{equation}
	with $R_1(z) = \mathcal{O}(z^{-1})$ as $z \to \infty$.
	Since $\Delta_1(z)$ has simple poles at $z=\pm 1$, we write
	\begin{equation}\label{eq:D1residues}
		\Delta_1(z)=\frac{A^{(1)}}{z-1}+\mathcal{O}(1), \quad z\to 1,\qquad
		\Delta_1(z)=\frac{B^{(1)}}{z+1}+\mathcal{O}(1), \quad z\to -1,
	\end{equation}
	for some constant matrices $A^{(1)}$ and $B^{(1)}$. Then the solution of the additive RH problem for $R_1(z)$ is given by
	\begin{equation}\label{eq:solRHR1}
		R_1(z)
		=
		\begin{cases}
			\displaystyle
			\frac{A^{(1)}}{z-1}+\frac{B^{(1)}}{z+1}, &\qquad
			z\in\mathbb{C}\setminus\left(\overline{D(1,\delta)}\cup \overline{D(-1,\delta)}\right),\\[3mm]
			\displaystyle
			\frac{A^{(1)}}{z-1}+\frac{B^{(1)}}{z+1}-\Delta_1(z), &\qquad
			z\in D(1,\delta)\cup D(-1,\delta).
		\end{cases}
	\end{equation}
	
	From the previous formulas, we have for $z \in D(1,\delta) \setminus \{1\}$,
	\begin{multline*}
		\Delta_1(z) =
		\frac{1}{2g(z)}
		M(z)
		\begin{pmatrix}
			V(z) & 0_r\\
			0_r & V(\overline{z})^{-\ast}
		\end{pmatrix} \Pi_r^{-1} 
		\diag \begin{pmatrix}
			- (\alpha_j^2 + \frac{1}{4}) & - \frac{1}{2} i \\
			- \frac{1}{2} i & \alpha_j^2 + \frac{1}{4} 
		\end{pmatrix} 
		\Pi_r \\ \times
		\begin{pmatrix}
			V^{-1}(z) & 0_r\\
			0_r & V(\overline{z})^{\ast}
		\end{pmatrix}
		M^{-1}(z).
	\end{multline*}
	for $z \in \partial D(1,\delta)$.
	We note $\ds \lim_{z \to 1} \frac{ (z-1)^{1/2}}{g(z)}  = \frac{1}{\sqrt{2}}$ and using \eqref{Mdef}, the explicit form of $M_0(z)$ and the local expansions of the matrices $D(z)^{-1}V(z)$ and $D(\overline{z})^{\ast}V(\overline{z})^{-\ast}$,
	where we recall that $U_1$ is unitary,
	\begin{align*} \lim_{z \to 1} (z-1)^{1/4} M(z) \begin{pmatrix} V(z) & 0_r
		\\ 0_r & V(\overline{z})^{-\ast} \end{pmatrix} 
	& = \frac{1}{2^{3/4}} \begin{pmatrix} D(\infty) U_1 & 0_r \\
			0_r & D(\infty)^{-\ast} U_1 \end{pmatrix}
		\begin{pmatrix} I_r & i I_r \\ - i I_r & I_r \end{pmatrix}
	\\[5pt]
	\lim_{z \to 1} (z-1)^{1/4} \begin{pmatrix} V^{-1}(z) & 0_r
		\\ 0_r & V(\overline{z})^{\ast} \end{pmatrix} M^{-1}(z) 
	& = \frac{1}{2^{3/4}} 
	\begin{pmatrix} I_r & -i I_r \\  i I_r & I_r \end{pmatrix}
	\begin{pmatrix} U_1^{-1} D(\infty)^{-1} & 0_r \\
		0_r & U_1^{-1} D(\infty)^{\ast} \end{pmatrix}.
	\end{align*} 
	Then we can calculate the residue $A^{(1)}$ in \eqref{eq:D1residues}, 
	\begin{multline}\label{eq:A1}
		A^{(1)} = \lim_{z \to 1} (z-1) \Delta_1(z) \\
	=	\frac{1}{8}
		\begin{pmatrix} 
			D(\infty) U_1 & 0_r \\
			0_r & D(\infty)^{-\ast} U_1 
		\end{pmatrix}
		\begin{pmatrix}
			I_r & iI_r\\
			-iI_r & I_r
		\end{pmatrix} 
		 \Pi_r^{-1} 
		\diag \begin{pmatrix}
			- (\alpha_j^2 + \frac{1}{4}) & - \frac{1}{2} i \\
			- \frac{1}{2} i & \alpha_j^2 + \frac{1}{4} 
		\end{pmatrix} 
		\Pi_r \\ 
		\times
		\begin{pmatrix}
			I_r & -iI_r\\
			iI_r & I_r
		\end{pmatrix}
		\begin{pmatrix} 
		 U_1^{-1}	D(\infty)^{-1} & 0_r \\
			0_r & U_1^{-1} D(\infty)^{\ast} 
		\end{pmatrix}\\
		=
		\frac{1}{4}
			\begin{pmatrix} 
			D(\infty) U_1 & 0_r \\
			0_r & D(\infty)^{-\ast} U_1 
		\end{pmatrix} \Pi_r^{-1}
		\diag \left( (\alpha_j^2-\frac{1}{4}) \begin{pmatrix}
			- 1 & i  \\
			i & 1  
		\end{pmatrix} \right) 
		\Pi_r \\ \times 
		\begin{pmatrix} U_1^{-1} D(\infty)^{-1} & 0_r \\
			0_r & U_1^{-1} D(\infty)^{\ast} \end{pmatrix}.
	\end{multline}
 Similarly, 
	\begin{multline}\label{eq:B1}
		B^{(1)} = \lim_{z \to -1} (z+1) \Delta_1(z) \\
		= \frac{1}{4}
		\begin{pmatrix} 
			D(\infty) U_{-1} & 0_r \\
			0_r & D(\infty)^{-\ast}  U_{-1}  
		\end{pmatrix} \Pi_r^{-1} 
		\diag \left( (\beta_j^2-\frac{1}{4}) \begin{pmatrix} 1 & i  \\
			i & -1 \end{pmatrix} \right) \Pi_r \\ \times
		\begin{pmatrix} 
			U_{-1}^{-1} D(\infty)^{-1} & 0_r \\
			0_r & U_{-1}^{-1} D(\infty)^{\ast}
		\end{pmatrix}.
	\end{multline}


\section{Proofs of the theorems}

\subsection{Proof of Theorem \ref{theorem14}} \label{subsec41}

\begin{proof}
Let $U$ be an open neighborhood of $[-1,1]$ in the complex plane.
We may assume that the lens around $[-1,1]$ and
the disks $D(\pm 1,\delta)$ are fully
contained in $U$. Then for $z \in \mathbb C \setminus U$, 
we have by \eqref{PninY}, \eqref{Tdef}, and \eqref{Sdef2},
\begin{equation} \label{PninT}
	P_n(z)  =
	\begin{pmatrix} I_r & 0_r \end{pmatrix} T(z) \begin{pmatrix} I_r \\ 0_r \end{pmatrix} \left(\frac{\varphi(z)}{2} \right)^n
	=
	\begin{pmatrix} I_r & 0_r \end{pmatrix} S(z) \begin{pmatrix} I_r \\ 0_r \end{pmatrix} \left(\frac{\varphi(z)}{2} \right)^n.
\end{equation}
Then $S=RM$ by \eqref{Rdef}, and using \eqref{Mdef}, we obtain
\[
\begin{aligned}
	\frac{2^n P_n(z)}{\varphi(z)^n}
	&=
	\begin{pmatrix} I_r & 0_r \end{pmatrix} R(z)
	\begin{pmatrix} D(\infty) & 0_r \\
		0_r & D(\infty)^{-\ast} \end{pmatrix}
	M_0(z) \begin{pmatrix} D(z)^{-1} & 0_r
		\\ 0_r & D(\overline{z})^{\ast} \end{pmatrix}
	\begin{pmatrix} I_r \\ 0_r \end{pmatrix} \\
	&=
	\begin{pmatrix} R_{11}(z)D(\infty) & R_{12}(z)D(\infty)^{-\ast} \end{pmatrix}
	\begin{pmatrix} M_{0,11}(z)  \\ M_{0,21}(z) \end{pmatrix}D(z)^{-1} 
\end{aligned}
\] 
where $R_{11} = \begin{pmatrix} I_r & 0_r \end{pmatrix} R
	\begin{pmatrix} I_r \\ 0_r \end{pmatrix}$
	and $R_{12} = \begin{pmatrix} I_r & 0_r \end{pmatrix} R
	\begin{pmatrix} 0_r \\ I_r \end{pmatrix}$ denote
	$r \times r$ submatrices of $R$ and similarly
	for $M_{0,11}$ and $M_{0,21}$.
Because of \eqref{M0def} these latter matrices are
multiples of the identity matrix, and we obtain 
\[ \frac{2^n P_n(z)}{\varphi(z)^n}
	= \left[\frac{1}{2} \left(\gamma(z) + \gamma(z)^{-1}\right)
	R_{11}(z)D(\infty)
	- \frac{1}{2i} \left(\gamma(z) - \gamma(z)^{-1}\right) R_{12}(z)D(\infty)^{-\ast} 
	 \right] D(z)^{-1} 
\] 
with $\gamma$ as in  \eqref{M0def}.
Using \eqref{Rexpansion} we conclude that
\begin{equation} \label{Pnexpansion}
	\frac{2^n P_n(z)}{\varphi(z)^n} \sim
	\sum_{k=0}^{\infty} \frac{\widetilde{\Pi}_k(z)}{n^k} 
	\end{equation} has a full asymptotic
expansion in inverse powers of $n$, with analytic
matrix valued functions $\widetilde{\Pi}_k$.

From 
\[ R(z) = I_{2r} + \frac{R_1(z)}{n} + \mathcal{O}(n^{-2}) \] 
uniformly
for $z \in \mathbb C \setminus U$, we obtain 
\begin{multline} \label{Pnexpansion2}
\frac{2^n P_n(z)}{\varphi(z)^n}
=
\left[ \frac{1}{2} \left(\gamma(z) + \gamma(z)^{-1}\right)
\left(I_r+ \frac{\left(R_1(z)\right)_{11}}{n} + \mathcal{O}(n^{-2})\right)D(\infty)
 \right. \\
 \left. - \frac{1}{2i} \left(\gamma(z) - \gamma(z)^{-1}\right) \left( \frac{\left(R_1(z)\right)_{12}}{n} + \mathcal{O}(n^{-2}) \right) D(\infty)^{-\ast}
\right]
D(z)^{-1}.
\end{multline}

Recall that  $D(\infty)$ and $D(z)^{-1}$ are invertible matrices
that are independent of $n$. Then we arrive at the leading term in the
expansion \eqref{Pnexpansion}
\begin{equation} \label{Pi0tildedef} 
	\widetilde{\Pi}_0(z) = \lim_{n \to \infty} \frac{2^n P_n(z)}{\varphi(z)^n} =
	\frac{1}{2} \left(\gamma(z)+\gamma(z)^{-1}\right) D(\infty)D(z)^{-1},
\end{equation}
which is \eqref{Pi0def} by simple rewriting of the
scalar prefactor. 

Next, we see from \eqref{Pnexpansion2} that 
the $n^{-1}$ term in \eqref{Pnexpansion} has the coefficient
\begin{equation} \label{Pitilde1z} \widetilde{\Pi}_1(z)
	= \left[	\frac{1}{2} \left(\gamma(z)+\gamma(z)^{-1}\right)
	(R_1(z))_{11} D(\infty) 
	- \frac{1}{2i} \left(\gamma(z) - \gamma(z)^{-1}\right) 
	 (R_1(z))_{12} D(\infty)^{-\ast} \right]
	 D(z)^{-1}.\end{equation}
We use \eqref{eq:solRHR1} and \eqref{eq:A1}, \eqref{eq:B1}
and the property \eqref{AjPir} of the permutation matrix $\Pi_r$
to conclude
\begin{align} \nonumber (R_{1}(z))_{11}
	& = \frac{(A^{(1)})_{11}}{z-1} +  \frac{(B^{(1)})_{11}}{z+1} \\
	& = - \frac{1}{4(z-1)} D(\infty) U_1 
		\diag\left(\alpha_1^2-\frac{1}{4}, \ldots, 
			\alpha_r^2 - \frac{1}{4} \right) U_1^{-1} D(\infty)^{-1} \nonumber \\
	& 	\quad \label{R1z11def}
	+  \frac{1}{4(z+1)} D(\infty) U_{-1} 
	\diag\left(\beta_1^2-\frac{1}{4}, \ldots, 
	\beta_r^2 - \frac{1}{4} \right) U_{-1}^{-1} D(\infty)^{-1}
\end{align}
and similarly
\begin{align}  (R_{1}(z))_{12} & =
	\frac{i}{4(z-1)} D(\infty) U_1 
	\diag\left(\alpha_1^2-\frac{1}{4}, \ldots, 
	\alpha_r^2 - \frac{1}{4} \right) U_1^{-1} D(\infty)^{\ast} \nonumber \\ \label{R1z12def}
	& \quad
	+  \frac{i}{4(z+1)} D(\infty) U_{-1} 
	\diag\left(\beta_1^2-\frac{1}{4}, \ldots, 
	\beta_r^2 - \frac{1}{4} \right) U_{-1}^{-1} D(\infty)^{\ast}
\end{align}
Inserting \eqref{R1z11def} and \eqref{R1z12def}
into \eqref{Pitilde1z}, we obtain
\begin{align*}
	\widetilde{\Pi}_1(z)
	= \frac{1}{2} \left(\gamma(z)+\gamma(z)^{-1}\right)
		D(\infty) \Pi_1(z) 	
		D(z)^{-1}  
\end{align*}
with
\begin{multline*} \Pi_1(z) =
	- \frac{1}{4(z-1)} U_1 
\diag\left(\alpha_1^2-\frac{1}{4}, \ldots, 
\alpha_r^2 - \frac{1}{4} \right) U_1^{-1} \\ 
	+	 \frac{1}{4(z+1)} U_{-1} 
	\diag\left(\beta_1^2-\frac{1}{4}, \ldots, 
	\beta_r^2 - \frac{1}{4} \right) U_{-1}^{-1} \\ 
	 - \frac{\gamma(z) - \gamma(z)^{-1}}{\gamma(z) + \gamma(z)^{-1}}
	\left[\frac{1}{4(z-1)} U_1 
	\diag\left(\alpha_1^2-\frac{1}{4}, \ldots, 
	\alpha_r^2 - \frac{1}{4} \right) U_1^{-1} \right. \\
	 \left.	+  \frac{1}{4(z+1)} U_{-1} 
	\diag\left(\beta_1^2-\frac{1}{4}, \ldots, 
	\beta_r^2 - \frac{1}{4} \right) U_{-1}^{-1} \right].	
\end{multline*}
This leads to \eqref{Pi1def} since
\begin{align*}
	\frac{1}{z-1} \left[ 1 + \frac{\gamma(z) - \gamma(z)^{-1}}{\gamma(z) + \gamma(z)^{-1}} \right]
	& =\frac{2}{\varphi(z) - 1}, \\
	\frac{1}{z+1} \left[ 1 - \frac{\gamma(z) - \gamma(z)^{-1}}{\gamma(z) + \gamma(z)^{-1}} \right]
	& =\frac{2}{\varphi(z) + 1},
\end{align*} 
as can be verified by direct calculation.
\end{proof}

\subsection{Proof of Theorem \ref{theorem15}} \label{subsec42}

\begin{proof}
We have by \eqref{PninY} and \eqref{Tdef}
\[
P_n(z)
=
\begin{pmatrix} I_r  & 0_r \end{pmatrix} Y(z) \begin{pmatrix} I_r  \\ 0_r \end{pmatrix}
=
 \begin{pmatrix} I_r  & 0_r \end{pmatrix} T(z) \begin{pmatrix} I_r  \\ 0_r \end{pmatrix}  \left( \frac{\varphi(z)}{2}\right)^n. \\
\]
For $z$ in the upper part of the lens outside
of the disks $D(\pm 1,\delta)$, we then have by \eqref{Sdef1}
and \eqref{Rdef}
\[
\begin{aligned}
	2^n P_n(z)
	&=
	 \begin{pmatrix} I_r  & 0_r \end{pmatrix}
	R(z)M(z)
	\begin{pmatrix} \varphi(z)^n I_r  \\ \varphi(z)^{-n}W(z)^{-1} \end{pmatrix}.
\end{aligned}
\]
We take the limit $z \to x \in (-1+\delta,1-\delta)$ and split the previous formula into two terms. Then we use \eqref{Mjump} to obtain
\begin{multline*}
	2^n P_n(x) =
	\begin{pmatrix} I_r  & 0_r \end{pmatrix} R(x)
	M_+(x) \begin{pmatrix} I_r \\ 0_r \end{pmatrix} \varphi_+(x)^n \\
	 + \begin{pmatrix} I_r & 0_r \end{pmatrix} R(x) 
	M_-(x) \begin{pmatrix} 0_r & W(x) \\ -W(x)^{-1} & 0_r
	\end{pmatrix} \begin{pmatrix} 0_2 \\ \varphi_+(x)^{-n} W(x)^{-1}
	\end{pmatrix} \\
	 = \begin{pmatrix} I_r  & 0_r \end{pmatrix} R(x) 
	\left(M_+(x) \varphi_+(x)^{n} + M_-(x) \varphi_-(x)^{n} \right)
	\begin{pmatrix} I_r \\ 0_r \end{pmatrix}.
\end{multline*}
where we also used $\varphi_+(x) \varphi_-(x) = 1$.

Note that $\varphi_{\pm}(x) = e^{\pm i \theta(x)}$, 
with $\theta(x) = \arccos(x)$. 
Then by the above and \eqref{Rasymp} to obtain
\[ 	2^n P_n(x) =
	\begin{pmatrix} I_r & 0_r \end{pmatrix} M_+(x)
		\begin{pmatrix} I_r \\ 0_r \end{pmatrix} e^{in \theta(x)}
	+	\begin{pmatrix} I_r & 0_r \end{pmatrix} M_-(x)
	\begin{pmatrix} I_r \\ 0_r \end{pmatrix} e^{-in \theta(x)} 
	+ \mathcal{O}(n^{-1}) \]
as $n \to \infty$. Using \eqref{Mdef} and \eqref{M0def} and
noting that
\[ 
\frac{1}{2} \left(\gamma_{\pm}(x) + \gamma_{\pm}^{-1}(x)\right)
	=  \frac{1}{\sqrt{2} \sqrt[4]{1-x^2}}
	e^{\pm \frac{i}{2} \theta(x)\mp \frac{\pi i}{4}, 
		\qquad -1 < x < 1,} \]
we arrive at \eqref{eq:innerasymp0}, with a $\mathcal{O}(n^{-1})$
term that is uniform for $x \in (-1+\delta, 1-\delta)$.

If the weight matrix $W$ is real symmetric, then $P_n(x)$
is real valued for real $x$. Then the normalized Szeg\H{o}
function has the symmetry \eqref{Dreal} which implies
$D_-(x) = \overline{D_+(x)}$ for $-1 < x < 1$. 
Hence the two terms within parantheses in \eqref{eq:innerasymp0}
are each other's complex conjugates,
and \eqref{eq:innerasymp} follows from \eqref{eq:innerasymp0}.

Since $\delta > 0$ can be taken arbitrarily small, the
asymptotic formulas \eqref{eq:innerasymp0}
and \eqref{eq:innerasymp} are valid uniformly for $x$ in any
compact subset of $(-1,1)$.
\end{proof}


\subsection{Proofs of Theorems \ref{theorem16} and \ref{theorem17}}
\label{subsec43}

\begin{proof}[Proof of Theorem \ref{theorem16}]

Let $x \in (1-\delta,1)$ in the upper part of the lens. 
Then, starting from \eqref{PninY} and following the transformations \eqref{Tdef}, 
\eqref{Sdef1}, \eqref{Rdef}, we have
\begin{align*} P_n(x)  = \begin{pmatrix} I_r & 0_r \end{pmatrix}	
	Y_+(x)  \begin{pmatrix} I_r \\ 0_r \end{pmatrix} 
		& = 2^{-n} \begin{pmatrix} I_r & 0_r \end{pmatrix}	
		T_+(x)  \begin{pmatrix} I_r \\ 0_r \end{pmatrix} 
		\varphi_+(x)^n  \\
		& = 2^{-n} \begin{pmatrix} I_r & 0_r \end{pmatrix}	
		S_+(x)  \begin{pmatrix} I_r \\ \varphi_+(x)^{-2n} W^{-1}(x) \end{pmatrix} \varphi_+(x)^n \\
		& =
		2^{-n} \begin{pmatrix} I_r & 0_r \end{pmatrix}	
		R(x) P_+(x)  \begin{pmatrix} \varphi_{+}(x)^n \\ \varphi_+(x)^{-n} W^{-1}(x) \end{pmatrix}.
\end{align*}
Inserting the formula \eqref{Pdef} for the local parametrix
$P$, and using $W = V_- V_-^\ast$ from \eqref{WVfactor} we get
\begin{align*}
		2^n P_n(x) & =
		\begin{pmatrix} I_r & 0_r \end{pmatrix}	
		R(x) E_n(x) P^{(1)}_+(x)  \begin{pmatrix} V_+^{-1}(x) \\ 
			V_-^{-1}(x) \end{pmatrix}. \end{align*}
The definition \eqref{Vdef} of $V$ and the factorization
\eqref{HQLambda} gives
 that $V_{\pm}(x)  = \sqrt{W(x)} Q(x)  e^{\pm \vec{\alpha}\pi i/2}$, 
 with $\vec{\alpha} = \diag(\alpha_1, \ldots, \alpha_r)$,
 where 
 \[ \sqrt{W(x)} = (1-x)^{\alpha/2} (1+x)^{\beta/2}
 	Q(x) \diag \left( \lambda_1(x)^{1/2}, \ldots,
 		\lambda_r(x)^{1/2} \right) Q(x)^\ast \]
is the positive square root of $W(x)$. Thus
\begin{align} \label{Pnnear1formula1}
 2^n	P_n(x) \sqrt{W(x)} & =
		\begin{pmatrix} I_r & 0_r \end{pmatrix}	
		R(x) E_n(x) P^{(1)}_+(x)  \begin{pmatrix} e^{-\frac{1}{2}
		\vec{\alpha}\pi i}  \\ 
		e^{\frac{1}{2}\vec{\alpha}\pi i}  \end{pmatrix} 
			 Q(x)^\ast.
		\end{align}
	
Recall that $P^{(1)}$ is given by \eqref{P1def}
in terms of the Bessel parametrices $\Psi_{\alpha_j}$ for $j=1, \ldots,r$.
Given $x \in (1-\delta, 1)$ we have from \eqref{fdef} that $f(x) = - \arccos(x)^2 < 0$ and 
by \eqref{Psialphadef}, see also the second line in
\eqref{Psialphadef},
\begin{multline} 
	\Psi_{\alpha,+}(n^2 f(x)) \\ =
	\begin{pmatrix} 
		\frac{1}{2} H_{\alpha}^{(1)}(2n \sqrt{-f(x)}) & \frac{1}{2} H_{\alpha}^{(2)}(2n \sqrt{-f(x)}) \\
	\pi i n \sqrt{-f(x)} \left( H_{\alpha}^{(1)}\right)'(2n \sqrt{-f(x)}) & 
	\pi i n \sqrt{-f(x)} \left( H_{\alpha}^{(2)}\right)'(2n \sqrt{-f(x)})
 \end{pmatrix}	\\ \times
\begin{pmatrix}  e^{\frac{1}{2} \alpha \pi i} & 0 \\
	0 & e^{- \frac{1}{2} \alpha \pi i} \end{pmatrix},
\end{multline}
where $H_{\alpha}^{(1)}$ and $H_{\alpha}^{(2)}$ are the
Hankel functions of order $\alpha$, and
$\left(H_{\alpha}^{(1)}\right)'$ and $\left(H_{\alpha}^{(2)}\right)'$ are their derivatives.
We use parameters
$\alpha_1, \ldots, \alpha_r$ and the short hand notation
\[ H_{\vec{\alpha}}^{(j)}(\xi) = \diag \left(H_{\alpha_1}^{(j)}(\xi),
\ldots, H_{\alpha_r}^{(j)}(\xi) \right), \qquad j =1,2, \]
and similarly for $\left(H_{\vec{\alpha}}^{(j)}\right)'$.
Thus by \eqref{P1def} and \eqref{Psialphadef}
\begin{multline*} 
	P^{(1)}_+(x) 
	= \begin{pmatrix} 
		\frac{1}{2} H_{\vec{\alpha}}^{(1)}(2n \sqrt{-f(x)}) & \frac{1}{2} H_{\vec{\alpha}}^{(2)}(2n \sqrt{-f(x)}) \\
		\pi i n \sqrt{-f(x)} \left( H_{\vec{\alpha}}^{(1)}\right)'(2n \sqrt{-f(x)}) & 
		\pi i n \sqrt{-f(x)} \left( H_{\vec{\alpha}}^{(2)}\right)'(2n \sqrt{-f(x)})
	\end{pmatrix} \\ 
\times 
	\begin{pmatrix} e^{\frac{1}{2} \vec{\alpha}\pi i} & 0_r \\ 0_r & e^{-\frac{1}{2}\vec{\alpha}\pi i} \end{pmatrix}.
	\end{multline*}
Hence, because of relation $J_{\alpha} = \frac{1}{2} \left( H_{\alpha}^{(1)} + H_{\alpha}^{(2)}\right)$ between the Bessel
function of the first kind and the Hankel functions, we obtain
\begin{align} 
	P^{(1)}_+ (x)  \begin{pmatrix} e^{-\frac{1}{2}\vec{\alpha}\pi i} \\ e^{\frac{1}{2} \vec{\alpha}\pi i} \end{pmatrix} 
  & = \begin{pmatrix} I_r & 0_r \\
  	0_r & 2\pi i n I_r \end{pmatrix} \begin{pmatrix} 
  	 J_{\vec{\alpha}}\left(2n \sqrt{-f(x)}\right) \\
  	\sqrt{-f(x)} \left( J_{\vec{\alpha}}\right)'\left(2n \sqrt{-f(x)}\right) 
  \end{pmatrix} \label{PtildeJalpha}
\end{align}
with $J_{\vec{\alpha}} = \diag \left(J_{\alpha_1}, \ldots,
 J_{\alpha_r}\right)$ and
 similarly for $\left(J_{\vec{\alpha}} \right)'$.
 Using this in \eqref{Pnnear1formula1}, we obtain
 \begin{multline} \label{Pnnear1formula2}
 	P_n(x) \sqrt{W(x)}  =
 	2^{-n} \begin{pmatrix} I_r & 0_r \end{pmatrix}	
 	R(x) E_n(x) 
 	\begin{pmatrix} I_r & 0_r \\
 		0_r & 2\pi i n I_r \end{pmatrix} \\ \times
 	\begin{pmatrix} 
 		J_{\vec{\alpha}}\left(2n \sqrt{-f(x)}\right) \\
 		\sqrt{-f(x)} \left( J_{\vec{\alpha}}\right)'\left(2n \sqrt{-f(x)}\right) 
 	\end{pmatrix}
 	Q(x)^\ast.
 \end{multline}

Next we focus on the product $E_n(x) \begin{pmatrix} I_r & 0_r
	\\ 0_r & 2 \pi i n I_r \end{pmatrix}$.   By \eqref{Ensum}
it is equal to
\begin{equation} \label{Enformula2}
	E_n(x) \begin{pmatrix} I_r & 0_r
		\\ 0_r & 2 \pi i n I_r \end{pmatrix}
	= \frac{\sqrt{\pi n}}{2}
		\begin{pmatrix} D(\infty) & 0_r \\ 0_r & D(\infty)^{-\ast} \end{pmatrix}
		\begin{pmatrix} I_r & i I_r \\ i I_r & I_r \end{pmatrix}
		\begin{pmatrix} E^{(1)}(x) & i E^{(2)}(x) \end{pmatrix}.
\end{equation}
The $n$-dependence appears only in the prefactor $\frac{\sqrt{\pi n}}{2}$. Since $R(x) = I_{2r} + \mathcal{O}(n^{-1})$ by \eqref{Rasymp}, we obtain from \eqref{Pnnear1formula2} and
 \eqref{Enformula2} that
\begin{multline} \label{Pnnear1formula3}
	2^n P_n(x) \sqrt{W(x)} 
	= \sqrt{\pi n}{2} 
  	 D(\infty) 
  	\begin{pmatrix} I_r & i I_r \end{pmatrix} 
  	  	\begin{pmatrix} E^{(1)}(x) & 
	  i  E^{(2)}(x) \end{pmatrix} 
  \left( I_{2r} + \mathcal{O}(n^{-1}) \right) \\
  \times
  	   \begin{pmatrix} 
  	  	J_{\vec{\alpha}}(2n \sqrt{-f(x)}) \\
  	  	\sqrt{-f(x)} \left( J_{\vec{\alpha}}\right)'(2n \sqrt{-f(x)}) 
  	  \end{pmatrix} Q(x)^\ast.
 \end{multline} 

Note that $E^{(1)}$ and $E^{(2)}$ are matrices
of size $2r \times r$ that are explicitly given in 
\eqref{E1z} and \eqref{E2z}. They are both analytic 
in the disk $D(1,\delta)$ around $z=1$, as was shown
in the proof of Lemma \ref{Enremove}.
For $z \in D(1,\delta)$ we may readily verify 
from \eqref{E1z} and \eqref{E2z}, and from the formula
\eqref{M0def} for $\gamma(z)$ that 
\begin{multline} 
	\begin{pmatrix} I_r & i I_r \end{pmatrix}
	\begin{pmatrix} E^{(1)}(z) & i E^{(2)}(z) \end{pmatrix}
	= \frac{f(z)^{1/4}}{(z^2-1)^{1/4}}	\\
	\times
	\begin{pmatrix} \left((z+1)^{1/2} + (z-1)^{1/2}\right)
		D(z)^{-1} V(z) & 
		\left((z+1)^{1/2} - (z-1)^{1/2}\right)
		D(\overline{z})^\ast V(\overline{z})^{-\ast}  \end{pmatrix} \\
	\times
	\begin{pmatrix} I_r & f(z)^{-1/2} I_r \\
		I_r & - f(z)^{-1/2} I_r \end{pmatrix}.
	\end{multline}
We use this in \eqref{Pnnear1formula3} for $z = x \in (1-\delta,1)$ with $+$ boundary
values, where we note the identities
\[
\frac{f_+(x)^{1/4}}{(x^2-1)_+^{1/4}} = 
	\frac{\sqrt{\arccos(x)}}{\sqrt{2} (1-x^2)^{1/4}},\qquad   f_+(x)^{-1/2} = -i (\sqrt{-f(x)})^{-1}
 \]
 and
 $(x-1)^{1/2}_+ = i \sqrt{1-x}$ to obtain		
\begin{multline} \label{PnrootW1}
	P_n(x) \sqrt{W(x)} 
	= \frac{\sqrt{\pi n \arccos x}}{2^{n+1} (1-x^2)^{1/4}} D(\infty) \\
	\times 
	\begin{pmatrix} \frac{\sqrt{1+x} + i \sqrt{1-x}}{\sqrt{2}}
		D_+(x)^{-1} V_+(x) & 
		\frac{\sqrt{1+x} - i \sqrt{1-x}}{\sqrt{2}}
		D_-(x)^\ast V_-(x)^{-\ast}  \end{pmatrix} \\
	\times
	\begin{pmatrix} I_r & -iI_r \\ I_r & iI_r \end{pmatrix}
	\left(I_{2r} + \mathcal O(n^{-1}) \right)
	\begin{pmatrix}  
		J_{\vec{\alpha}}(2n \sqrt{-f(x)}) \\
		 \left( J_{\vec{\alpha}}\right)'(2n \sqrt{-f(x)}) 
	\end{pmatrix} Q(x)^\ast,
\end{multline} 
which proves \eqref{Pnnear1} in view of the definition
\eqref{Azdef}, and the fact that
$D_-(x) D_-(x)^\ast = V_-(x)V_-(x)^\ast$ by \eqref{WVfactor}, which can be rewritten as
\[ D_-(x)^\ast V_-(x)^{-\ast} = D_-(x)^{-1} V_-(x). \]
\end{proof}

\begin{proof}[Proof of Theorem \ref{theorem17}]

We use $x = \cos \frac{\theta}{n}$ in \eqref{Pnnear1}.
As $n \to \infty$, we then have $x \to 1-$, $A_+(x) \to A(1) = U_1$, $A_-(x) \to U_1$
and $\frac{\arccos x}{(1-x^2)^{1/4}} \to 1$.
Because of \eqref{Pnnear1} we then have
\begin{equation} \label{MHproof1} \lim_{n \to \infty}
	\frac{2^n}{\sqrt{\pi n}} P_n(x) \sqrt{W(x)} Q(x)
		= D(\infty) U_1 J_{\vec{\alpha}}(\theta)
	\qquad \text{with } x = \cos \frac{\theta}{n}. \end{equation}
We write
\begin{equation} \label{MHproof2} 
	\sqrt{W(x)} Q(x) = Q(x) \Lambda^{1/2}(x) (1-x)^{\alpha/2} (1+x)^{\beta/2}, \end{equation}
where we used \eqref{Wdef} and \eqref{HQLambda}.
For each $j=1, \ldots, r$ we have
\begin{equation*} 
	\frac{\lambda_j(x) (1-x)^{\alpha} (1+x)^{\beta}}{(1-x)^{\alpha_j}}
	 \to  c_j 2^{\alpha_j} \quad \text{ as } x \to 1-, \end{equation*}
	 because of \eqref{eq:cj}.
Since $1-x = \frac{\theta^2}{2n^2} + \mathcal O(n^{-4})$
as $n \to \infty$, when $x = \cos \frac{\theta}{n}$, we get
\[ n^{2\alpha_j} \lambda_j(x) (1-x)^{\alpha} (1+x)^{\beta}
	\to c_j \theta^{2\alpha_j} \]
as $n \to \infty$. It follows that
\begin{multline} \label{MHproof3} 
 \lim_{n \to \infty}	\Lambda^{1/2}(x) (1-x)^{\alpha/2} (1+x)^{\beta/2}
	\diag\left(c_1^{-1/2} n^{\alpha_1}, \ldots, c_r^{-1/2} n^{\alpha_r}\right) \\
	= \diag\left(\theta^{\alpha_1},
		\ldots, \theta^{\alpha_r} \right), \end{multline}
again with $x = \cos \frac{\theta}{n}$.
Combining \eqref{MHproof1}, \eqref{MHproof2} and \eqref{MHproof3},
we obtain \eqref{eq:MHgeneral}.
\end{proof}

\subsection{Proof of Theorem \ref{theorem18}} \label{subsec44}

The recurrence coefficients can be obtained from the solution
of the RH problem for $Y$. To emphasize the dependence on $n$, we write $Y^{(n)}$, and similarly for other matrices that depend
on $n$. For a matrix $X$ of size $2r \times 2r$
we use $X_{ij}$ for $i,j=1,2$ to denote its submatrices of 
size $r \times r$, that is,
\[ X = \begin{pmatrix} X_{11} & X_{12} \\ X_{21} & X_{22} \end{pmatrix}. \]

We start with explicit formulas for the recurrence coefficients $B_n$ and $C_n$ in terms of the matrices $R^{(n)}$ and
$R^{(n+1)}$.
\begin{lemma}
	We have
	\begin{equation} \label{Bnformula} 
		B_n = \lim_{z\to \infty} z \left( R_{11}^{(n)}(z) -
		R_{11}^{(n+1)}(z) \right) \end{equation}
	and
	\begin{equation} \label{Cnformula}
		C_n = \lim_{z \to \infty}
		\left( \frac{i}{2} D(\infty) D(\infty)^\ast	
			+ z R_{12}^{(n)}(z) \right)
		\left(-\frac{i}{2} D(\infty)^{-\ast} D(\infty)^{-1}
			+ z R_{21}^{(n)}(z) \right).
	\end{equation}	
\end{lemma}
\begin{proof} 
The matrix valued function
	\begin{equation} \label{Undef} 
		U^{(n)}(z) = Y^{(n+1)}(z) Y^{(n)}(z)^{-1} 
	\end{equation} 
	is entire, since the jump matrix for $Y^{(n)}$ is indpendent of $n$, and it is equal to
	\begin{align} \nonumber U^{(n)}(z) & = \begin{pmatrix} z I_r & 0_r \\ 0_r & 0_r \end{pmatrix}	
		+ Y_1^{(n+1)} \begin{pmatrix} I_r & 0_r \\ 0_r & 0_r \end{pmatrix} - 
		\begin{pmatrix} 	I_r & 0_r \\ 0_r & 0_r \end{pmatrix} Y_1^{(n)} \\ \label{Unformula}
		& = \begin{pmatrix} zI_r + \left(Y_1^{(n+1)}\right)_{11}
			- \left( Y_1^{(n)} \right)_{11} & 
			- \left(Y_1^{(n)} \right)_{12} \\
			\left( Y_1^{(n+1)} \right)_{21} & 0_r \end{pmatrix}. 
	\end{align}
where $Y_1^{(n)}$ denotes the residue matrix in
\begin{equation} \label{Y1def} 
	Y^{(n)}(z) \begin{pmatrix} z^{-n} I_r & 0_r \\
		0_r & z^n I_r \end{pmatrix} =  
	I_{2r} + \frac{Y_1^{(n)}}{z} + 
	\mathcal{O}(z^{-2}) \end{equation}
	as $z \to \infty$
with a constant matrix $Y_1^{(n)}$.
Then by \cite[Theorem 2.16]{GIM11}
\begin{align} \label{BnwithU} 
		B_n & = zI_r - U_{11}^{(n)}(z) = \left(Y_1^{(n)}\right)_{11} - \left( Y_1^{(n+1)} \right)_{11}. \\ \label{CnwithY}
	C_n & = \left(Y_1^{(n)}\right)_{12} 
		\left(Y_1^{(n)}\right)_{21} \end{align}

The transformations $Y \mapsto T \mapsto S$ in \eqref{Tdef} and
\eqref{Sdef2}, and
the definition \eqref{Undef} show
that for $z$ outside of the lens
\[ U^{(n)}(z) = 
	\begin{pmatrix} 2^{-n-1} I_r & 0_r \\ 0_r & 2^{n+1} I_r \end{pmatrix}
	S^{(n+1)}(z) \begin{pmatrix} \varphi(z) I_r & 0_r \\ 0_r &
		\varphi(z)^{-1} I_r \end{pmatrix} S^{(n)}(z)^{-1}
	\begin{pmatrix} 2^n I_r & 0_r \\ 0_r & 2^{-n} I_r \end{pmatrix}.
	\] 
All factors in this product remain bounded as $z \to \infty$, except for
the diagonal matrix with $\varphi(z)$.
Since $\varphi(z) = 2z + \mathcal{O}(z^{-1})$ as $z \to \infty$
we obtain 
\begin{equation} \label{UnwithS} 
	U^{(n)}(z) = 2z 
\begin{pmatrix} 2^{-n-1} I_r & 0_r \\ 0_r & 2^{n+1} I_r \end{pmatrix}
S^{(n+1)}(z) \begin{pmatrix} I_r & 0_r \\ 0_r &
	0_r \end{pmatrix} S^{(n)}(z)^{-1}
\begin{pmatrix} 2^n I_r & 0_r \\ 0_r & 2^{-n} I_r \end{pmatrix}
+ \mathcal{O}(z^{-1}) \end{equation}
as $z \to \infty$.

We use the final transformation $S^{(n)} = R^{(n)} M$,
see \eqref{Rdef}, where we note that 
\begin{multline} \label{MinUasymp} 
		M(z) \begin{pmatrix} I_r & 0_r \\ 0_r & 0_r \end{pmatrix}
		M(z)^{-1}  \\
		= \begin{pmatrix} I_r & 0_r \\ 0_r & 0_r \end{pmatrix}
		+ \frac{1}{2iz} \begin{pmatrix} 0_r & D(\infty) D(\infty)^\ast \\ D(\infty)^{-\ast} D(\infty)^{-1} & 0_r	\end{pmatrix}  + 
		\mathcal{O}(z^{-2}) \quad
	\text{ as } z \to \infty, \end{multline}
which follows from direct calculation from \eqref{M0def}
and \eqref{Mdef}.
Using $S^{(n)} = R^{(n)} M$ and \eqref{MinUasymp} in \eqref{UnwithS}
we arrive at (where we also use $R^{(n+1)}(z) = I_{2r} + \mathcal{O}(z^{-1})$ and  
$R^{(n)}(z)^{-1} = I_{2r} +  \mathcal{O}(z^{-1})$)
\begin{multline} \label{UnwithR}
 U^{(n)}(z) = 2z 
\begin{pmatrix} 2^{-n-1} I_r & 0_r \\ 0_r & 2^{n+1} I_r \end{pmatrix}
R^{(n+1)}(z) \begin{pmatrix} I_r & 0_r \\ 0_r &
	0_r \end{pmatrix}  R^{(n)}(z)^{-1}
\begin{pmatrix} 2^n I_r & 0_r \\ 0_r & 2^{-n} I_r \end{pmatrix} \\
+ \frac{1}{2i} 
 \begin{pmatrix} 0_r & 4^{-n} D(\infty) D(\infty)^{\ast} \\ 
	4^{n+1} D(\infty)^{-\ast} D(\infty)^{-1}  & 0_r \end{pmatrix}  
  + \mathcal{O}(z^{-1}) \end{multline}
as $z \to \infty$.

Using \eqref{UnwithR} in the first identity of \eqref{BnwithU} we note
that the term with $D(\infty)$ does not contribute to $B_n$
and we get
\begin{equation} \label{BnwithR1} B_n = \lim_{z \to \infty}
	\left( zI_r - z R^{(n+1)}_{11}(z) \left(R^{(n)}(z)^{-1}\right)_{11} \right) 
\end{equation}
We use the Laurent expansion
$\ds R^{(n)}(z) = I_{2r} + \frac{R_1^{(n)}}{z} + \mathcal{O}(z^{-2})$, where we note  
$\ds  R^{(n)}(z)^{-1} = I_{2r} - \frac{R_1^{(n)}}{z} + \mathcal{O}(z^{-2})$, and we obtain from \eqref{BnwithR1} that
\[ B_n = \lim_{z \to \infty} 
	\left(z 
	\left( R_1^{(n)}(z) - R_1^{(n+1)}(z) \right)_{11} \right)
	\]
which is the same as \eqref{Bnformula}.

To obtain the formula \eqref{Cnformula} for $C_n$ we go back to \eqref{CnwithY}.
The transformations $Y \mapsto T \mapsto S = RM$ then show that
\begin{align*}
	Y_{12}^{(n)}(z) & =
	2^{-n} \left[ R_{11}^{(n)}(z) M_{12}(z) + R_{12}^{(n)}(z) M_{22}(z) \right] \varphi(z)^{-n} \\
	Y_{21}^{(n)}(z) & =
	2^{n} \left[ R_{21}^{(n)}(z) M_{11}(z) + R_{22}^{(n)}(z) M_{21}(z) \right] \varphi(z)^{n} 
\end{align*}
and by \eqref{Y1def}, since $\varphi(z) = 2z + \mathcal{O}(z^{-1})$
as $z \to \infty$,
\begin{equation} \label{Y1entrylimits}
\begin{aligned}
	\left(Y_1^{(n)}\right)_{12} & =
		\frac{1}{4^{n}} \lim_{z \to \infty} 
		z \left[  R_{11}^{(n)}(z) M_{12}(z) + R_{12}^{(n)}(z) M_{22}(z) \right], \\
	\left(Y_1^{(n)}\right)_{21} & =
	4^{n} \lim_{z \to \infty} 
	z \left[  R_{21}^{(n)}(z) M_{11}(z) + R_{22}^{(n)}(z) M_{21}(z) \right].
\end{aligned}
\end{equation}
Since $R^{(n)}(z) \to I_{2r}$ and $M(z) \to I_{2r}$ as $z \to \infty$
with
\begin{align*} \lim_{z\to \infty} z M_{12}(z) & =
	\frac{i}{2} D(\infty) D(\infty)^{\ast}, \\
 \lim_{z \to \infty} z M_{21}(z) & = 
	- \frac{i}{2} D(\infty)^{-\ast} D(\infty)^{-1}, \end{align*}
see the formulas \eqref{M0def} and \eqref{Mdef} for $M$,
we obtain \eqref{Cnformula} from \eqref{CnwithY}
and \eqref{Y1entrylimits}. 
\end{proof} 

We can now turn to the proof of Theorem \ref{theorem18}.

\begin{proof}[Proof of Theorem \ref{theorem18}.]
Due to the asymptotic expansion \eqref{Rexpansion} for 
$R = R^{(n)}$, and the formulas \eqref{Bnformula}
and \eqref{Cnformula} we see that $B_n$ and $C_n$
have an asymptotic expansion in inverse powers of $n$.

Because of \eqref{Rexpansion} we have  
\[ z\left(R^{(n)}(z) - R^{(n+1)}(z)\right)
	\sim \sum_{k=1}^{\infty} 
		 \left(\frac{1}{n^k} - \frac{1}{(n+1)^k}\right) R_k(z)
\] 
and the $k$th term in this series is $\mathcal{O}(n^{-k-1})$
as $k \to \infty$, uniformly for $z$ in a neighborhood of $\infty$. Keeping only the first term we have
\[ z(R^{(n)}(z) - R^{(n+1)}(z))
	 = \frac{R_1(z)}{n^2} + \mathcal{O}(n^{-3}). \] 
From the explicit form \eqref{eq:solRHR1} we then obtain
\begin{align*} 
	\lim_{z \to \infty} z \left(R^{(n)}(z) - R^{(n+1)}(z)\right)
	= \frac{A^{(1)} + B^{(1)}}{n^2} + \mathcal{O}(n^{-3})
	\end{align*}
as $n \to  \infty$. We thus conclude from \eqref{Bnformula}
\[ B_n = \frac{A_{11}^{(1)} + B^{(1)}_{11}}{n^2} + \mathcal{O}(n^{-3}).\]
The explicit formulas \eqref{eq:A1} and \eqref{eq:B1}
for $A^{(1)}$ and $B^{(1)}$ then lead
to the formula \eqref{b2def} for
$\mathcal{B}_2 = A_{11}^{(1)} + B_{11}^{(1)}$.

\medskip

Now we turn to the recurrence coefficient $C_n$.
From \eqref{Cnformula} with $\lim\limits_{z \to \infty} 
	z R_{12}^{(n)}(z) = \mathcal{O}(n^{-1})$,
	$\lim\limits_{z \to \infty} 
	z R_{21}^{(n)}(z) = \mathcal{O}(n^{-1})$,
	we find
	\begin{multline} \label{Cnexpansion} 
		C_n = \frac{1}{4} I_r 
		- \frac{i}{2} 
		\left(\lim_{z \to \infty}
			z R_{12}^{(n)}(z) \right) 
			D(\infty)^{-\ast} D(\infty)^{-1} \\
		+ \frac{i}{2} D(\infty) D(\infty)^\ast
		\left(\lim_{z \to \infty} z R_{21}^{(n)}(z) \right) +
			\mathcal{O}(n^{-2})   \end{multline}
	as $n \to \infty$. We have the leading term $\frac{1}{4} I_r$
	for $C_n$ and to prove \eqref{eq:BnCnasymp} it remains
	to show that the $n^{-1}$ term in the expansion
	for $C_n$ vanishes.

By \eqref{Rexpansion} and \eqref{Bnformula} we have 
\[ \lim_{z \to \infty} z R_{12}^{(n)}(z)
	= \frac{1}{n} \lim_{z\to\infty} \left(zR_1(z) \right)_{12}
	+ \mathcal{O}(n^{-2})
	= \frac{1}{n} \left(A^{(1)}_{12} + B^{(1)}_{12} \right)
	+ \mathcal{O}(n^{-2}), 
\]
and similarly for $R_{21}^{(n)}$. 
Thus by \eqref{Cnexpansion}, $C_n = \frac{1}{4} I_r + 
\frac{\mathcal{C}_1}{n} + \mathcal{O}(n^{-2})$ as $n \to \infty$
with 
\begin{align} \label{C1def} 
	\mathcal{C}_1  =  
- \frac{i}{2}  \left(A^{(1)}_{12} + B^{(1)}_{12} \right)
D(\infty)^{-\ast} D(\infty)^{-1} 
+ \frac{i}{2} D(\infty) D(\infty)^\ast
 \left(A^{(1)}_{21} + B^{(1)}_{21} \right).   \end{align}

The formulas \eqref{eq:A1} and \eqref{eq:B1} can be written
as 
\begin{multline} \begin{pmatrix} D(\infty)^{-1} & 0_r \\ 0_r & D(\infty)^{\ast} \end{pmatrix} A^{(1)} \begin{pmatrix} D(\infty) & 0_r \\ 0_r & D(\infty)^{-\ast} \end{pmatrix} \\
	=  \begin{pmatrix}
		-X & i X \\ iX & X \end{pmatrix},
	\qquad X = \frac{1}{4} U_1  \diag\left(\alpha_j^2-\frac{1}{4}\right) U_1^{-1},  \end{multline}
and
\begin{multline} \begin{pmatrix} D(\infty)^{-1} & 0_r \\ 0_r & D(\infty)^{\ast} \end{pmatrix} B^{(1)} \begin{pmatrix} D(\infty) & 0_r \\ 0_r & D(\infty)^{-\ast} \end{pmatrix} \\ 
	=  \begin{pmatrix}
		\widetilde{X} & i \widetilde{X} \\ i\widetilde{X} & -\widetilde{X} \end{pmatrix},
	\qquad \widetilde{X} = \frac{1}{4} U_{-1}  \diag\left(\beta_j^2-\frac{1}{4}\right) U_{-1}^{-1}.  
\end{multline}
We observe that in both these formulas the off-diagonal 
blocks agree, which means that 
\begin{align*} D(\infty)^{-1} A^{(1)}_{12} D(\infty)^{-\ast} & = 
	D(\infty)^\ast A^{(1)}_{21} D(\infty), \\
 D(\infty)^{-1} B^{(1)}_{12} D(\infty)^{-\ast} & =  
	D(\infty)^\ast B^{(1)}_{21} D(\infty). \end{align*}
These two identities and \eqref{C1def} show
that indeed $\mathcal{C}_1 = 0_r$ and the proof is complete.
\end{proof}

\section*{Acknowledgements}
A. D. acknowledges financial support from Direcci\'on General de Investigaci\'on e Innovaci\'on, Consejer\'ia de Educaci\'on e Investigaci\'on of Comunidad de Madrid (Spain), and Universidad de Alcal\'a under grant CM/JIN/2021-014, and Comunidad de Madrid (Spain) under the Multiannual Agreement with UC3M in the line of Excellence of University Professors (EPUC3M23), and in the context of the V PRICIT (Regional Programme of Research and Technological Innovation). Research supported by Grant PID2021-123969NB-I00, funded by MCIN/AEI/ 10.13039/501100011033, and by grant PID2021-122154NB-I00 from Spanish Agencia Estatal de Investigaci\'on.

\medskip

A. D. acknowledges financial support and hospitality from IMAPP, Radboud Universiteit Nijmegen, and in particular Prof. Erik Koelink, during a visit to Nijmegen in June 2022.

\medskip

A.B.J.K. is supported by long term structural funding-Methusalem grant of the
Flemish Government, and by FWO Flanders projects EOS 30889451 and G.0910.20.

\medskip

P. R. acknowledges the support of Erasmus+ travel grant and  SeCyT UNC.

\section*{Competing interests}
 The authors have no competing interests to declare that
are relevant to the content of this article.

\end{document}